\documentclass[mathpazo]{jpde}

\usepackage{array}
\RequirePackage[OT1]{fontenc}
\RequirePackage{amsmath,natbib}
\usepackage{enumerate}
\usepackage{amsfonts}
\usepackage{booktabs}
\usepackage{extarrows}
\usepackage{graphicx}
\usepackage{verbatim}
\usepackage{amsthm}
\usepackage{url}
\usepackage{bm}
\usepackage{natbib}
\usepackage{amsmath}
\usepackage{amsfonts}
\usepackage{amssymb}
\usepackage{color,xcolor}

\newcommand{\bfbeta}{\mbox{\boldmath $\beta$}}

\newcommand{\bfxi}{\mbox{\boldmath $\xi$}}

\newcommand{\EB}{\mathrm{E}}
\newcommand{\E}{\mathrm{E}}
\newcommand{\R}{\mathbb{R}}
\newcommand{\diag}{\mbox{diag}}

\newcommand{\vxi}{{\boldsymbol{\xi}}}

\newcommand{\vzeta }{{\boldsymbol{\zeta}}}

\newcommand{\vertiii}[1]{{\vert\kern-0.25ex\vert\kern-0.25ex #1
          \kern-0.25ex\vert\kern-0.25ex\vert}}
\newcommand{\opnorm}[1]{\vertiii{ #1 }_{2}}
\newcommand{\hsnorm}[1]{\left\Vert #1 \right\Vert_{\mathrm{F}}}

\newcommand{\psinorm}[1]{\Vert #1 \Vert_{\psi_2}}

\def \psione {{\psi_1}}
\def \psitwo {{\psi_2}}

\providecommand{\norm}[1]{\left\lVert#1\right\rVert}
\DeclareMathOperator*{\argmin}{argmin}
\begin{document}
%%%%% title : short title may not be used but TITLE is required.

\title{REVIEW ARTICLE\\
~~\\
%User-friendly
%Sharper and Constants-specified
Concentration Inequalities for Statistical Inference}

%%%%% author(s) :
% single author:
% \author[name in running head]{AUTHOR\corrauth}
% [name in running head] is NOT OPTIONAL, it is a MUST.
% Use \corrauth to indicate the corresponding author.
% Use \email to provide email address of author.
% \footnote and \thanks are not used in the heading section.
% Another acknowlegments/support of grants, state in Acknowledgments section
% \section*{Acknowledgments}
 \author[Zhang H., Chen S.X.]{Huiming Zhang\affil{1}\comma\affil{3}\comma\affil{4},
      Song Xi Chen\affil{1}\comma\affil{2}\comma\affil{3}\corrauth}
 \address{\affilnum{1}\ School of Mathematical Sciences,
          %Peking University,
          %Beijing 100871, P.R. China. \\
           \affilnum{2}\ Guanghua School of Management, %Peking University,
           %Beijing 100871, P.R. China.\\
            \affilnum{3}\ Center for Statistical Sciences, Peking University,
           Beijing, P.R. China;\\\affilnum{4} Department of Mathematics, Faculty of Science and Technology, University of Macau, Macau, China.}
					
 \emails{{\tt zhanghuiming@pku.edu.cn} (Zhang H.), {\tt csx@gsm.pku.edu.cn} (Chen S.X.)}

% multiple authors:
% Note the use of \affil and \affilnum to link names and addresses.
% The author for correspondence is marked by \corrauth.
% use \emails to provide email addresses of authors
% e.g. below example has 3 authors, first author is also the corresponding
%      author, author 1 and 3 having the same address.

% \footnote and \thanks are not used in the heading section.
% Another acknowlegments/support of grants, state in Acknowledgments section
% \section*{Acknowledgments}

%%%%% Begin Abstract %%%%%%%%%%%
\begin{abstract}

This paper gives a review of concentration inequalities which are widely employed in non-asymptotical analyses of mathematical statistics in a wide range of settings, from distribution-free to distribution-dependent, from sub-Gaussian to sub-exponential, sub-Gamma, and sub-Weibull random variables, and from the mean to the maximum concentration.
This review provides results in these settings with some fresh new results. Given the increasing popularity of high-dimensional data and inference, results in the context of high-dimensional linear and Poisson regressions are also provided. We aim to illustrate the concentration inequalities with known constants and to improve existing bounds with sharper constants.

\end{abstract}
%%%%% end %%%%%%%%%%%

%%%%% AMS/Chinese Library Classifications/Keywords %%%%%%%%%%%

%%%% maketitle %%%%%
\maketitle

%%%% Start %%%%%%
\tableofcontents
%\newpage
\section{Introduction}
In probability theory and statistical inference, researchers often need to bound the probability of a difference between a random {quantity from its target}, usually the error bound of estimation.
 Concentration inequalities (CIs) are tools for attaining such bounds, and play important roles in deriving theoretical results for various inferential situations in statistics and probability. The recent developments in high-dimensional (HD) statistical inference, and statistical and machine learning have generated renewed interests in the CIs, as reflected in  \cite{Koltchinskii11}, \cite{Vershynin18}, \cite{Wainwright19} and \cite{Fanj20}.
As the CIs are diverse in their forms and the underlying distributional requirements, and are scattered around in references,  there is an increasing need for a review which collects existing results
together with some new results (sharper and constants-specified CIs) from the authors for researchers and graduate students working in statistics and probability.  This motivates the writing of this review.

{CIs enable us to obtain non-asymptotic results for estimating, constructing confidence intervals, and doing hypothesis testing with a high-probability guarantee. For example, the first-order optimized condition for HD linear regressions should be held with a high probability to guarantee the well-behavior of the estimator. The concentration inequality for error distributions is to ensure the concentration from first-order optimized conditions to the estimator.} {Our review focuses on four types of CIs:
\begin{center}
$P(Z_n > \mathrm{E}Z_n + t),\quad  P(Z_n < \mathrm{E}Z_n - t), \quad
P(|Z_n - \mathrm{E}Z_n| > t)$ and  $ \mathrm{E}(\max\limits_{i=1,\cdots,n}  |{X_i }|)$
\end{center}
where  $Z_n: = f({X_1}, \cdots ,{X_n})$ and ${X_1}, \cdots ,{X_n}$ are random variables.} We present two types of CIs: distribution-free and distribution-dependent. Distribution free CIs are free of distribution assumptions, while the distribution-dependent CIs are based on exponential moment conditions %(such as subclasses of Gaussian, exponential Gamma, and Weibull)
reflecting the tail property for the particular class of distributions. {Concentration phenomenons for a sum of sub-Weibull random variables will lead to a mixture of two tails: sub-Gaussian for small deviations and sub-Weibull for large deviations from the mean, and it is closely related to Strong Law of Large Numbers, Central Limit Theorem, and Law of the Iterative Logarithm. }We provide applications of the CIs to empirical processes and high-dimensional data settings. The latter includes the linear and Poisson regression with a diverging number of covariates.  We organize the materials in the forms of Lemmas,  Corollaries,  Propositions, and Theorems. Lemmas and Corollaries are on existing results usually without proof except for a few fundamental ones. Propositions are also for existing results but with sharper or more precise constants {and sometimes come with proofs}. Theorems are for new results. This review contains 27 Lemmas, 21 Corollaries, 15 Propositions, and 4 Theorems.

The review is organized as follows.
Section 2 outlines distribution-free CIs. CIs for Sub-Gaussian, Sub-exponential, sub-Gamma, and sub-Weibull random variables are given in Section 3, 4, 5, and 6 respectively. Section 7 reports concentration for the maximal of random variables and suprema of empirical processes. Applications for high dimensional linear and Poisson regression are outlined in Section 8. {Section 9 discusses extensions to other settings. }

\section{Distribution-free Concentration Bounds}\label{Distribution-free}

The purpose here is to introduce distribution-free CIs. We first review Markov's, Chebysheff's and Chernoff's tail probability bounds that constitute fundamental inequalities for deriving most of the concentration bounds; see Chap. 1 of \cite{Durrett2019} or Appendix B in \cite{Giraud2014} for the proofs.

\begin{lemma}[Markov's inequality]\label{Markov}
Let $\varphi(x): \mathbb{R} \rightarrow \mathbb{R}^{+}$ be any non-decreasing positive function. For any real valued random variable (r.v.) $X,$
$
{P}(X \geq a) \leq \mathrm{E}[\varphi(X)]\frac{1}{\varphi(a)} ,~\forall~ a \in \mathbb{R}.
$
\end{lemma}

By letting $\varphi(x)=x^2$, the following Chebyshev's inequality is merely an application of Markov's inequality for $|X-\mathrm{E}X|$.
\begin{lemma}[Chebyshev's Inequality]
Let $X$ be a r.v. with expectation $\mathrm{E}X$ and variance $\operatorname{Var}X$. Then, for any $a \in \mathbb{R}^{+}:$
$
{P}(|X-\mathrm{E}X| \geq a) \leq \frac{\operatorname{Var}X}{a^{2}}.
$
\end{lemma}
The Chebyshev's inequality prescribes a polynomial rate of convergence depending on the {variance assumption}. Another  application of Markov's inequality is the Chernoff's bound which is sharper by optimizing the upper bounds.

\begin{lemma}[Chernoff's inequality]
For a r.v. $X$ with $\mathrm{E}{e^{tX}}< \infty$,
$P(X \ge a) \le {\inf}_{t > 0} \left\{ {{e^{ - ta}}\mathrm{E}{e^{tX}}} \right\}.$
\end{lemma}
\begin{proof}
Lemma \ref{Markov} with $\varphi(x)=e^{t x}$ implies
$P(X \geq a) \leq {{e^{ - ta}}\mathrm{E}{e^{tX}}}$ and minimize $t$ on $t>0$.
\end{proof}
The Jensen's inequality and its truncated version [Lemma 14.6 in \cite{Buhlmann11}] are another powerful tool to derive useful inequalities by the convexity.
\begin{lemma}[Jensen's inequality]
For any convex function $\varphi: \mathbb{R}^{d} \rightarrow \mathbb{R}$ and any r.v. $X$ in $\mathbb{R}^{d},$ such that
$\varphi(X)$ is integrable, we have $\varphi(\mathrm{E}X) \leq \mathrm{E}[\varphi(X)]$.
\end{lemma}
\begin{lemma}[Truncated Jensen's inequality]\label{lem:Truncated}
Let $g(\cdot)$ be an increasing function on $[0, \infty),$ which is concave on
$[c, \infty)$ for some $c \geq 0 .$ Then
$\mathrm{E} g(|Z|) \leq g[\mathrm{E}|Z|+c {P}(|Z|<c)]$ for r.v. $Z$.
\end{lemma}

The Chebyshev's, Markov's, Chernoff's and Jensen's inequalities are also valid for conditional expectations (Chapter 4 of \cite{Durrett2019}).
The Chernoff's bound typically lead to a tighter bound than Markov's inequality by  optimization via an exponential $\varphi(x)$ function. A sharper bound for the sum of independent random variables(r.vs) was attempted in \cite{Hoeffding63}. The following is a slightly sharper bound %for Hoeffding's lemma,
  from Theorem 1.2 in \cite{Bosq1998}.

\begin{corollary}[Hoeffding's inequality]\label{lm:Hoeffding}
Let ${X_1}, \cdots ,{X_n}$ be independent r.vs on $\R$ satisfying bound condition
$
{a_i}\le {{X_i}} \le {b_i}~\text{for}~i = 1,2, \cdots ,n.
$
 Then for $t,u> 0$
\begin{enumerate}
\item[\rm{(}a\rm{)}] \textbf{Hoeffding's lemma}: ${\rm{E}}e^{u \sum\limits_{i = 1}^n ({{X_i}}-{\rm{E}}{X_i})}  \le e^{ \frac{u^2}{8}\sum\limits_{i = 1}^n {(b_i-a_i)^2} }~\text{and}~{\rm{E}} e^{u| \sum\limits_{i=1}^{n} ({{X_i}}-{\rm{E}}{X_i})|}\leq 2e^{\frac{u^{2}}{8}  \sum\limits_{i=1}^{n} (b_i-a_i)^{2}}$;

\item[\rm{(}b\rm{)}] \textbf{Hoeffding's inequality}: $P(|\sum_{i = 1}^n ({{X_i}}-{\rm{E}}{X_i}) | \ge t) \le 2e^{-{2 t^{2}}/{\sum_{i=1}^{n}\left(b_{i}-a_{i}\right)^{2}}}$.
\end{enumerate}
\end{corollary}

With exp-decay, Corollary \ref{lm:Hoeffding} has a sharper bound than the Markov's inequality or Chebyshev's inequality with the requirement of first or moment condition on $X$. Hoeffding's inequality has many applications in statistics as shown in the next example.
\iffalse
The exponential bounds have term 2 in $2e^{-{2 t^{2}}/{\sum_{i=1}^{n}\left(b_{i}-a_{i}\right)^{2}}}$ (and in subsequent sections) which is from separate bounds on the right tail probability $P(\sum_{i = 1}^n {{X_i}}  \ge t)$ and the left tail
 $P(\sum_{i = 1}^n {{X_i}}  \le -t)$. So it is enough to show the the exponential bound of the right tail, since by the symmetry of centralized r.vs, the right tail follows from the left tail inequality applied to $-\sum_{i = 1}^n {{X_i}}$.
\fi

\textbf{The proof of Hoeffding's lemma.}  Without loss of generality, we assume $\mathrm{E}X_i=0$. This is from the fact that the concentration inequality is location shift-invariance. Since $f(x)=e^x$ is convex, for $u>0$, then
$e^{u x} \leq \frac{b_{i}-x}{b_{i}-a_{i}} e^{u a_{i}}+\frac{x-a_{i}}{b_{i}-a_{i}} e^{u b_{i}},~a_{i} \leq x \leq b_{i}.$
Taking expectation, it gives by $\mathrm{E}X_i=0$
\begin{equation} \label{hoeffd1}
\mathrm{E}e^{u X_{i}}  \le \frac{b_{i}}{b_{i}-a_{i}} e^{u a_{i}}-\frac{a_{i}}{b_{i}-a_{i}} e^{u b_{i}} =\left[1-s+s e^{u\left(b_{i}-a_{i}\right)}\right] e^{-s u\left(b_{i}-a_{i}\right)} \triangleq e^{f(r)},
\end{equation}
where $r=u(b_i-a_i), s=-a_i/(b_i-a_i)$ and $f(r)=-sr+\log(1-s+se^r).$ We can show that
$f'(r) =  - s + \frac{{s{e^r}}}{{1 - s + s{e^r}}},~~f''(r) = \frac{{(1 - s)s{e^r}}}{{{{(1 - s + s{e^r})}^2}}}\le \frac{{\rm{1}}}{{\rm{4}}}$
for all $r\ge 0$. Note that $f(0)=f'(0)=0$. Consider the Taylor's expansion of $f$, there exists $\xi\in[0,1]$ such that
$f(r)=r^{2} f^{\prime \prime}(\xi r) / 2 \leq r^{2} / 8=u^{2}\left(b_{i}-a_{i}\right)^{2} / 8.$
Substitute it to \eqref{hoeffd1}, we get the Hoeffding's lemma.

 The last assertion of Lemma~\ref{lm:Hoeffding}(a) is by letting $Z=u\sum_{i=1}^{n} ({{X_i}}-{\rm{E}}{X_i})$, so that
\begin{equation}\label{eq:abZ}
{\rm{E}}e^{|Z|}={\rm{E}} e^{-Z} \cdot 1(Z \leq 0)+{\rm{E}}e^{Z} \cdot 1(Z>0)\leq 2e^{\frac{1}{8} u^{2} \sum_{i=1}^{n} (b_i-a_i)^{2}}.
\end{equation}

\textbf{The proof of Hoeffding's inequality.} Let $S_n =\sum_{i=1}^n X_i$ and $c_i=a_i-b_i$. For any $t,u> 0$,
\begin{align}\label{eq:SN0}
 P(S_n -\mathrm{E}S_n \ge t) &= P(e^{u (S_n -\mathrm{E}S_n) } \ge e^{ut}) \le \mathop {\inf }\limits_{u > 0} e^{-ut} \prod_{i=1}^n \mathrm{E}e^{u (X_i - \mathrm{E}X_i)}~[\text{Chernoff's inequality}]\nonumber \\
 [\text{Hoeffding's lemma}] &  \le \mathop {\inf }\limits_{u > 0} e^{-ut} \prod{}_{i=1}^n e^{ u^2 c_i^2/8} = \mathop {\inf }\limits_{u > 0} e^{-ut + u^2 \sum_{i=1}^nc_i^2/8}=e^{ {-2t^{2}}/\sum_{i=1}^{n} c_{i}^{2}}.
\end{align}
The smallest bounded is attained at $u = 4t/\sum_{i=1}^nc_i^2$ and $P(-[S_n -\mathrm{E}S_n] \ge t) \le e^{ {-2t^{2}}/\sum_{i=1}^{n} c_{i}^{2}}$ similarly. Hence, the Hoeffding's inequality is verified via
$$
P(|S_n -\mathrm{E}S_n| \ge t)\le P(S_n -\mathrm{E}S_n \ge t) + P(-[S_n -\mathrm{E}S_n] \ge t)\le 2e^{ {-2t^{2}}/\sum_{i=1}^{n} c_{i}^{2}}.
$$

Corollary \ref{lm:Hoeffding} has a sharper bound than the Markov's inequality or Chebyshev's inequality with the requirement of first or moment condition on $X$. A second approach for proving Hoeffding's lemma is given in Lemma 1.8 of \cite{Rigollet19}.  Hoeffding's inequality has many applications in statistics as shown in the next example.

\begin{example}[Empirical distribution function, EDF]\label{eg:edf}
Let $\{ {X_i}\} _{i = 1}^n \stackrel{\rm IID}{\sim} F(x)$ for a distribution $F$. Let $\mathbb{F}_{n}(x):={\frac  1n}\sum _{{i=1}}^{n}{\rm{1}}_{{\{X_{i}\leq x\}}}(x),~x\in {\mathbb{R}}$ be the empirical distribution. By Hoeffding's inequality ($a_i-b_i=1/n$),
$P(|\mathbb{F}_{n}(x)-F(x)|>\varepsilon  )\leq 2e^{-2n\varepsilon ^{2}},~\forall \varepsilon >0.$
\end{example}

McDiarmid's inequality (also called bounded difference inequality, see \cite{McDiarmid89}) is a concentration inequality for a multivariate function of random sequence $\{X_i\}_{i=1}^n$, says $f(X_{1},...,X_{n})$. As a generalization of Hoeffding's inequality, it does not require any distribution assumptions about r.vs and the $f(X_{1},...,X_{n})$ may be dependent sum of r.vs. The only requirement is the bounded difference condition by replacing $X_{j}$ by $X_{j}^{'}$ meanwhile maintaining the others fixed in $f(X_{1},...,X_{n})$.
\begin{lemma}[McDiarmid's inequality]\label{lm:bd}
Suppose $X_{1},\cdots,X_{n}$ are independent r.vs all taking values in the set $A$, and
assume $f:A^n\rightarrow\mathbb{R}$ satisfies the \emph{bounded difference condition}
\begin{center}
${\sup}_{{x_{1},\cdots,x_{n},x_{k}^{'}\in A}}\vert f(x_{1},\cdots,x_{n})-f(x_{1},\cdots,x_{k-1},x_{k}^{'},x_{k+1},\cdots,x_{n})\vert\le c_{k}.$
\end{center}
Then, $P\left( {\left| {f({X_1},\cdots,{X_n}) - {\rm{E}}\left\{f({X_1},\cdots,{X_n})\right\}} \right| \ge t} \right) \le 2e^{-
2{t^2}/\sum_{i = 1}^n {c_i^2}}~~\forall t>0.$
\end{lemma}

 If $f(x_{1},\cdots,x_{n})=\sum_{i=1}^{n} x_{i}$ then $c_k=x_{k}-x_{k}^{'}$. Let $\mathcal{B}(\boldsymbol{c})$ be the set of functions $f: \mathcal{X}^{n} \rightarrow \mathbb{R}$ s.t., for any $x=\left(x_{1}, \ldots, x_{n}\right)$ and $y=\left(y_{1}, \ldots, y_{n}\right)$ in $\mathcal{X}^{n}$,
$|f(x)-f(y)| \le \sum_{i=1}^{n} c_{i} 1_{\left\{x_{i} \neq y_{i}\right\}}$, i.e. The BDC is a Lipschitz property of $f$ with respect to the Hamming distance.

One method of proof is by the martingale argument, which needs to check the Azuma-Hoeffding's inequality below, see Section 2.2.2 in \cite{Wainwright19}. Theorem 3.3.14 of \cite{Gine15} gives another proof based on the entropy method.
\begin{lemma}[Azuma-Hoeffding's inequality]\label{le:RAH}
Let $\{X_{n}\}_{n=0}^{\infty}$ be a sequence of martingale (or supermartingale), adapted to an increasing filtration $\{\mathcal{F}_{n}\}_{n=0}^{\infty}$. Suppose $\{X_{n}\}_{n=0}^{\infty}$ satisfies the \emph{bounded difference condition}
$
a_{k}\le X_{k}-X_{k-1}\le b_{k},~~\rm{a.s.}
$
for $k=1, \ldots, n$. Then,
\begin{center}
${P}\left(\left|X_{n}-X_0\right|>t\right) \leq 2e^{
-2{t^2}/\sum_{i = 1}^n {(b_{k}-a_{k})^2}},~t\ge 0$.
\end{center}
\end{lemma}
Three typical examples with bounded differences function are the concentration for U-statistics (a dependent summation) and the integral error of the kernel density estimation, and concentration of bounded empirical processes.
\begin{example}[U-statistics]\label{eg:u}
Let $\{X_{i}\}_{i=1}^n$ be independent and identically distributed (IID) r.vs and $g: \mathbb{R}^{2} \rightarrow \mathbb{R}$ be the bounded and symmetric function. Define a \emph{U-statistic of order 2} as
$U_n={\tiny \left(\begin{array}{l}
n \\
2
\end{array}\right)}^{-1} \sum_{i<j} g\left(X_{i}, X_{j}\right):=f(x_{1},...,x_{n}).$
Its bounded difference condition is
\begin{align*}
&~~~~\left|f\left(x_{1}, \ldots, x_{k-1}, x_{k}, x_{k+1}, \ldots x_{n}\right)-f\left(x_{1}, \ldots, x_{k-1}, x_{k}^{\prime}, x_{k+1}, \ldots x_{n}\right)\right| \\
&=\frac{1}{\tiny \left(\begin{array}{l}
n \\
2
\end{array}\right)}|\sum_{j=1, j \neq k}^{n}[g(x_{k}, x_{j})-g(x_{k}^{\prime}, x_{j})]| \leq \frac{2\cdot 2(n-1)\|g\|_{\infty}}{n(n-1)}=\frac{4\|g\|_{\infty}}{n}.
\end{align*}
So we have
$
{P}(|U_n-\mathrm{E}U_n|>t) \leq 2 e^{-{n t^{2}}/{8\|g\|_{\infty}^{2}}}
$.
\end{example}

\begin{example}[$L_{1}$-error in kernel density estimation]
Let $\{X_{i}\}_{i=1}^n \stackrel{\rm IID}{\sim} F(x)$ with density function $f(x)$. Define the \emph{kernel density estimator} by
$\hat { f }_{n,h} ( x ) = \frac { 1 } { n } \sum _ { i = 1 } ^ { n } \frac { 1 } { h } K \left( \frac { x - X _ { i } } { h } \right)$
where $K(\cdot)>1$ is the kernel function and $h>0$ is a smoothing parameter called the bandwidth. Usually, the kernel function $K(\cdot)$ is symmetric probability density and $h>0$ with $h \to 0$ and $nh \to \infty$. Define the $L_{1}$-error of $\hat { f }_{n,h} ( x ) $ by
${Z_n}=g({X}_{1}, \ldots,{X}_{{n}})=\int|\hat { f }_{n,h} ( x )-f({x})| {dx}. $
By $\int K(u) d u=1$, the McDiarmid's inequality with bound difference condition
\begin{center}
$\left|g\left({x}_{1}, \ldots, {x}_{{n}}\right)-g({x}_{1}, \ldots, {x}_{{i}}^{\prime}, \ldots, {x}_{{n}})\right|\leq \frac{1}{n } \int\left|K\left(\frac{x-x_{i}}{h}\right)-K\left(\frac{x-x_{i}^{\prime}}{h}\right)\right| d \left(\frac{x}{h}\right) \le \frac{2}{n}$
\end{center}
gives $P(|Z_n-\mathrm{E}Z_n|\ge t)\leq 2 e^{-{2 t^{2}}/{n(\frac{2}{n})^{2}} }=2 e^{- n t^{2}/2 }$, which is free of the bandwidth.
\end{example}

The McDiarmid's inequality is particularly useful when deriving CIs for supremum of empirical processes: $f(X):=\sup _{f \in \mathcal{F}} \frac{1}{n} \sum_{i=1}^{n} f(X_{i})$ for independent data $\{X_{i}\}_{i=1}^n$, see Section \ref{bracketingEP} for more details.
\begin{example}[The supremum of bounded EP]\label{eg:u}
 Given $f \in \mathcal{F}$, WLOG, assume that
\begin{center}
$\forall f \in \mathcal{F}, \quad \mathrm{E}\left[f(X_{i})\right]=0, \quad \text { and } \quad f(X_{i}) \in\left[a_{i}, b_{i}\right]$
\end{center}
 For any $x=\left(x_{1}, \ldots, x_{n}\right) \in \mathcal{X}^{n},$ define
$f(x)=\sup _{f \in \mathcal{F}} \frac{1}{n} \sum_{i=1}^{n} f(x_{i})$,
\begin{center}
$|\mathop {\sup }\limits_{f \in \mathcal{F}} \frac{1}{n}\sum\limits_{i = 1}^n f(x_{i})  - \mathop {\sup }\limits_{f \in \mathcal{F}} \frac{1}{n}\sum\limits_{i = 1}^n f(y_{i}) | \le \mathop {\sup }\limits_{f \in \mathcal{F}} |\frac{1}{n}\sum\limits_{i = 1}^n {(f(x_{i})}  - f(y_{i}))| \le \sum\limits_{i = 1}^n {{\frac{{{b_i} - {a_i}}}{n}}} {1_{\left\{ {{x_i} \ne {y_i}} \right\}}}.$
\end{center}
It is clear that {$f \in \mathcal{B}(\mathbf{c}),$ with $c_{i}=\left(b_{i}-a_{i}\right) /n.$}  Therefore, the BDC implies $\forall  M \ge 0$,
\[P( {\sup\limits_{f \in \mathcal{F}} \frac{1}{{ n }}\sum\limits_{i = 1}^n f(X_{i})  \le {\frac{1}{\sqrt n}{\rm{E}}[ {\mathop {\sup }\limits_{f \in \mathcal{F}} \frac{1}{\sqrt n}\sum\limits_{i = 1}^n f(X_{i}) } ]} + \frac{M}{{\sqrt n }}}) \ge 1 - {e^{ -{2M^2}/{[\frac{1}{n}\sum\limits_{i = 1}^n {{{\left( {{b_i} - {a_i}} \right)}^2}]}} }}.\]
\end{example}

\section{Sub-Gaussian Distributions}\label{sec-sg}
\subsection{Motivations}

In probability, there is a well-known inequality for bounding the Gaussian tail. If $X \sim N(0,1)$,  \cite{Gordon1941} obtained for $x>0$
\begin{equation}\label{eq:Mills}
\left( {\frac{1}{x} - \frac{1}{{{x^3}}}} \right)\cdot \frac{e^{ - {x^2}/2}}{{\sqrt {2\pi } }}< \left(\frac{x}{x^{2}+1} \right)\cdot\frac{e^{ - {x^2}/2}}{{\sqrt {2\pi } }} \le P(X \ge x) \le \frac{1}{x}\cdot \frac{e^{ - {x^2}/2}}{{\sqrt {2\pi }}},
\end{equation}
which is called \emph{Mills's inequality}, relating  to Mills's ratio \citep{Mills1926}. The upper bound in \eqref{eq:Mills} is mostly used to derive law of the iterated logarithm \citep{Durrett2019}. However, if $x$ tends to zero the upper bound goes to $+\infty$ which makes it meaningless. So the Mill's inequality is useful only for larger $x$. % since the factor $e^{-x^{2} / 2}$ dominates.
We need a better inequality. % showing that the distribution concentrates most of its probability in a small region.
In fact, the upper bound in \eqref{eq:Mills} can be strengthened as in Lemma B.3 in \cite{Giraud2014}:
$P(|X| \ge x) \le {e^{ - {x^2}/2}}.$  We refer it as the \emph{sharper Mill's inequality}.

\begin{example} [$O({a^{-2}})$-decay tail inequality is not enough]
$\{X_{ij}\} \stackrel{\rm i.i.d.}{\sim} N(0,\sigma^2)$ for $i=1,\cdots,n$ and $j=1,\cdots,p_n$ with $p_n \gg n \to \infty$. Note that $\sum\limits_{i = 1}^n {{X_{ij}}}\stackrel{\rm i.i.d.}{\sim} N(0,n\sigma^2)$. By Chebyshev's inequality,
\[{T_n}: = P\left( {\mathop {\max }\limits_{1 \le j \le p} |\sum\limits_{i = 1}^n {{X_{ij}}} | \ge t\sqrt n } \right) \le \sum\limits_{j = 1}^p P \left( {|\sum\limits_{i = 1}^n {{X_{ij}}} | \ge t\sqrt n } \right) \le \frac{{pn{\sigma ^2}}}{{n{t^2}}} = \frac{{p{\sigma ^2}}}{{{t^2}}}.\]
Sharper Mill's inequality gives
\[{T_n} \le \sum\limits_{j = 1}^p P \left( {|\sum\limits_{i = 1}^n {{X_{ij}}} | \ge t\sqrt n } \right) = p{e^{ - {{(t\sqrt n )}^2}/(2n{\sigma ^2})}} = p{e^{ - {t^2}/(2{\sigma ^2})}}.\]
\end{example}
Put $t = \sqrt p$ and letting $p = {p_n} \to \infty$, Chebyshev's inequality derives a upper bound of ${T_n} \le p{\sigma ^2}/{t^2} = {\sigma ^2}\ne 0$. However, sharper Mill's inequality guarantees ${T_n} \le p{e^{ - {p^2}/(2{\sigma ^2})}} \to 0$.\\

In statistics, people want to study a general class of error distributions (beyond Gaussian) whose \textit{moment generating function} (MGF): $\mathrm{E}e^{s X}$ have similar Gaussian properties with $s$ in specific subset of $\mathbb{R}$. To derive sharper Mill's inequality, % for a class for distributions,
 it is natural to define the class of sub-Gaussian r.v. as follows.

\begin{definition}[Sub-Gaussian distribution, \cite{Kahane60}]\label{def:Sub-Gaussian}
A r.v. $X \in \mathbb{R}$ with mean zero is \emph{sub-Gaussian} with a
\textit{variance proxy} $\sigma^{2}$ (denoted $X \sim \operatorname{subG}\left(\sigma^{2}\right)$) if its MGF satisfies
$\mathrm{E}e^{s X}\le e^{\frac{\sigma^{2} s^{2}}{2}},~\forall s \in \mathbb{R}.$
\end{definition}
The minimal $\sigma^2$ is called the \emph{optimal variance proxy} \citep{chow1966}
\begin{center}
$\sigma^2_{opt}(X):= \inf \big\{ \sigma^2 > 0 : {\mathbb{{E}}} \exp(tX) \leq \exp\{\sigma^2 t^2 / 2\}, \quad \forall \, t \in \mathbb{R} \big\}$.
\end{center}
With Definition \ref{def:Sub-Gaussian} and Chernoff's inequality, we will get the exponential decay of the tail as the alternative definition of sub-Gaussian:
\begin{center}
${P}\left( X \ge  t \right)\le \mathop {\inf }\limits_{s > 0} e^{-s t}\mathrm{E} e^{s X} \le \mathop {\inf }\limits_{s > 0} e^{-s t + \frac{\sigma^2s^2}{2}}\xlongequal{s={t}/{\sigma^2}} e^{- \frac{t^2}{2 \sigma^2}}$.
\end{center}
 This argument is called \emph{Cramer-Chernoff method}, and it is applied in proving Hoeffding's lemma for sum of independent variables. In general, let $Z_{1}, \ldots, Z_{n}$ be $n$ independent centralized r.vs, and suppose there
exists a convex function $g(t)$ and a domain $D_0$ containing $\{0\}$ such that $\mathrm{E}e^{t \sum_{i=1}^{n} Z_{i}} \leq e^{n g(t)},~\forall t \in D_0 \subset \mathbb{R}.$ Denote $g^{*}(s)=\sup _{t \in D_0}\{t s-g(t)\}$ as the \textit{convex conjugate function} of $g,$ therefore the Chernoff's inequality implies
\begin{center}
$P(|\frac{1}{n} \sum_{i=1}^{n} Z_{i}|>s) \leq 2e^{-n g^{*}(s)},~\forall s>0$,
\end{center}
which has rich applications in high-dimensional statistics, machine learning, random matrix theory, and other fields on non-asymptotic results.

Note that $\operatorname{subG}(\sigma^{2})$ denotes a class of distributions rather than a single distribution. Trivially, the Gaussian distribution is a special case of sub-Gaussian.

\begin{example} [Normal distributions]
Consider the normal r.v. $X\sim N(\mu, \sigma^{2})$. With the MGF of $X$: $\mathrm{E}e^{s X}:= e^{\frac{\sigma^{2} s^{2}}{2}},~\forall s \in \mathbb{R}$, it is sub-Gaussian with the variance proxy $\sigma^{2}=\operatorname{Var}(X)$.
\end{example}

\begin{example}[Bounded r.vs]
By Hoeffding's lemma, $\mathrm{E}e^{s X} \le e^{ \frac{1}{8}{s^2}(b-a)^2}~\text{for}~s> 0$ for the centralized bounded variable $X \in [{a},{b}]$. So $X$ is essentially sub-Gaussian with variance proxy $\sigma^{2}=\frac{1}{4}{(b-a)^2}$. For Bernoulli variable $X \in \{0,1\}$, we have $X \sim \operatorname{subG}\left(\frac{1}{4}\right)$.
\end{example}

There are at least seven equivalent forms for sub-Gaussian as shown in the following.
\begin{corollary}[Characterizations of sub-Gaussian]\label{Sub-gaussian distribution}
Let $X$ be a r.v. in $\R$ with $\E X = 0$. Then, the following are  equivalent for finite positive  constants $\{K_i\}_{i=1}^7$.   %  definitions of Sub-Gaussian r.vs include:
\begin{enumerate}[\rm{(}1\rm{)}]
\item \label{p: sub-gaussian MGF}
      The MGF of $X$:
      $
      \E e^{s X} \le e^{K_1^2 s^2}~\text{for all } s \in \R
      $;

\item \label{p: sub-gaussian tail}
      The tail of $X$:
      $
      P \{ |X| \ge t \} \le 2 e^{-t^2/K_2^2}~ \text{for all } t \ge 0;
      $
\item \label{p: sub-gaussian moments}
      The moments of $X$:
      $
     (\E |X|^k)^{1/k} \le K_3 \sqrt{k}~\text{for all integer}~k \ge 1;
      $
% is bounded at some point

\item \label{p: sub-gaussian MGF square finite}
      The exponential moment of $X^2$ \index{Moment Generating Function}:
      $
      \E e^{X^2/K_4^2}\le 2;
      $
\item \label{p: sub-gaussian MGF square}
      The local MGF of $X^2$:
      $\E e^{\l^2 X^2} \le e^{K_5^2 \l^2}~\text{for all $\l$ in a local set} |\l| \le \frac{1}{K_5}. $

\item  There is a constant $K_6 \ge 0$ such that
$\E e^{{\lambda X^{2}}/{ K_6^2}} \le ({1-\lambda})^{-1/2}~\text { for all } \lambda \in[0,1)$.

%\item Union bound condition: $\exists c>0$ s.t. $\mathrm{E}\left[\max \left\{|X_{1}|, \ldots,|X_{n}|\right\}\right] \leq c \sqrt{\log n}$ for all $n \geq c$, where $\{X_i\}_{i=1}^n$ are IID copies of $X$.
\end{enumerate}
\end{corollary}

\begin{remark}
The $\E X = 0$ is for convenience as the zero mean is used in the proof of Corollary \ref{Sub-gaussian distribution}(1), see \cite{Vershynin2010} for the details and the proof of the equivalences (1)-(5). The equivalences (6) is given in Theorem 2.6 of \cite{Wainwright19}%and the equivalences (7) is present in Page24 of \cite{Sen2018}.
The moment condition for integers $k$ in (3) can be relaxed to \emph{even integers} $k$ by the symmetrization technique. By symmetry of $X$, let us consider a negative independent copy $-X^{\prime}$ which is independent of $X$ and has the same distribution as $X$.  If (3) is true and $\mathrm{E}(-X^{\prime})=0$, from Jensen's inequality ${\rm{E}}{e^{\theta ( - {X^\prime })}} \ge {e^{\theta {\rm{E}}( - {X^\prime })}} = 1$ since $-{X^\prime }$ has zero mean. So we have by the independence of $ X^{\prime}$ and $X$:
\begin{align*}
{\rm{E}}{e^{\theta X}}& \le {\rm{E}}{e^{\theta X}}{\rm{E}}{e^{\theta ( - {X^\prime })}}  = {\rm{E}}{e^{\theta (X - {X^\prime })}}= 1 + \sum\limits_{k = 1}^\infty {\frac{{{\theta ^{2k}}{\rm{E}}{{(X - {X^\prime })}^{2k}}}}{{(2k)!}}}  \le 1 + \sum\limits_{k = 1}^\infty {\frac{{{\theta ^{2k}}{\rm{E(}}|X| + |{X^\prime }|{)^{2k}}}}{{(2k)!}}} \\
[\text{By}~\eqref{eq:Jensen}]& \le 1 + \sum\limits_{k = 1}^\infty  {\frac{{{\theta ^{2k}}{{\rm{2}}^{2k}}{\rm{E}}|X{|^{2k}}}}{{(2k)!}}} < 1 + \sum\limits_{k = 1}^\infty  {\frac{{{{(2\theta K_3\sqrt {2k} )}^{2k}}}}{{{k^k}k!}}}  = 1 + \sum\limits_{k = 1}^\infty  {\frac{{{{(8{\theta ^2}K_3^2)}^k}}}{{k!}}}  = {e^{8{\theta ^2}K_3^2}}~\forall~\theta \in \R
\end{align*}
where the last inequality is due to ${(2k)!}>k^k\cdot k!$.
\end{remark}

\subsection{The variance proxy and sub-Gaussian norm}

We show that the $\sigma^{2}$ in Definition \ref{def:Sub-Gaussian} is indeed the upper bounds of variance of $X$:
\begin{center}
 $\operatorname{Var}X \le \sigma^{2} .$
\end{center}
The $\sigma^{2}$
not only characterizes the speed of decay in the sub-Gaussian tail probability, but also bounds the variance of $n^{-1 / 2} \sum_{i=1}^{n} X_{i}$. This is because, by the sub-Gaussian MGF
\begin{align}\label{eq:var}
{\frac{\sigma^{2} s^{2}}{2}}+o(s^2)= e^{\frac{\sigma^{2} s^{2}}{2}} -1 \ge \mathrm{E}e^{s X}-1=s\mathrm{E} X +\frac{s^{2}}{2}\mathrm{E} X^2+\cdots=  \frac{s^{2}}{2}\cdot\operatorname{Var}X+o\left(s^{2}\right)
\end{align}
Dividing $s^{2}$ on both sides of  \eqref{eq:var} and taking $s \rightarrow 0$.

In the theory of empirical process, sub-Gaussian definitions are characterized by Orlicz norms (see Definition \ref{def:IncOrliczNorm} below and \cite{van96}) to derive exponential tail inequality for empirical process.
\begin{definition}[Sub-Gaussian norm]\label{def: sub-gaussian}
  For a sub-Gaussian r.v. $X$, the sub-Gaussian norm of $X$, denoted $ \|X\|_\psitwo$,  is defined by:
$\|X\|_\psitwo = \inf \{ t>0 :\; \E e^{X^2/t^2} \le 2 \}.$
\end{definition}

From Corollary \ref{Sub-gaussian distribution}(4), $\|X\|_\psitwo$
is the smallest $K_4$. An alternative definition of the sub-Gaussian norm is $\|X\|_{\psi _2} := \sup_{p \ge 1} p^{-1/2} ({\rm{E}} |X|^p)^{1/p}$ \citep{Vershynin2010}.
%and this versions of definitions of the sub-Gaussian norm are equal up to a constant, see \cite{Zajkowski19} for more discussion.
The definition for sub-Gaussian norm makes Corollary~\ref{Sub-gaussian distribution} easily presented.
% and it is concise to use the sub-Gaussian norm to determine the sub-Gaussian variance bound $\sigma^{2}$ in Definition \ref{def:Sub-Gaussian} without zero-mean assumptions.
In fact, if $\E e^{X^2/\|X\|_\psitwo^2} \le 2$,
\begin{equation}
P(|X| \ge t)=P({e^{  {X^2}/{\|X\|_\psitwo^2}}} \ge {e^{  {t^2}/{\|X\|_\psitwo^2}}}) \le {{\E{e^{{X^2}/\|X\|_\psitwo^2}}}}/{{{e^{{t^2}/\|X\|_\psitwo^2}}}} \le 2{e^{ - {t^2}/{\|X\|_\psitwo^2}}}.\label{eq: sub-gaussian tail}
\end{equation}
\begin{example}[The sub-Gaussian norm of bounded r.vs.]\label{eg:sub-Gaussianboundedr.v.}
Consider a r.v. $|X| \le M<\infty$. Set $\E e^{X^2/t^2}\le e^{M^2/t^2}\le 2 $ and $t\ge M/\sqrt{\log 2}$, By the definition of the sub-Gaussian norm, we have $\|X\|_\psitwo= M/\sqrt{\log 2}$. The equality holds if $|X| = M<\infty$.
\end{example}

\begin{example}[The sub-Gaussian norm of Gaussian r.vs.]\label{eg:sub-GaussianGaussianr.v.}
For a $N(0,\sigma^{2})$ and $t > \sqrt 2 \sigma $,
${\rm{E}}{e^{{X^2}/{t^2}}}=\int {{e^{{x^2}/{t^2}}}\frac{e^{ - {{{{x^2}}/{2{\sigma ^2}}}}}}{ \sqrt{{{2\pi\sigma ^2}}}}} dx=
\frac{t}{{{{({t^2} - 2{\sigma ^2}{\rm{)}}}^{1/2}}}}\int {{e^{{x^2}/{t^2}}}{{{\rm{(2}}\pi {\textstyle{{{\sigma ^2}{t^2}} \over {{t^2} - 2{\sigma ^2}}}}{\rm{)}}}^{ - 1/2}}} \exp \{  - \frac{{{x^2}}}{{2({\textstyle{{{\sigma ^2}{t^2}} \over {{t^2} - 2{\sigma ^2}}}})}}\} dx=
\frac{{ t}}{{{{({t^2} - 2{\sigma ^2}{\rm{)}}}^{1/2}}}} \le 2 ~\Rightarrow~t \ge \sqrt {\frac{8}{3}} {\sigma }.$ By the definition,  $\|X\|_\psitwo=\sqrt {\frac{8}{3}} \sigma > \sqrt {2} \sigma$.
\end{example}
%The $\|X\|_\psitwo$ is the smallest possible number that makes each of these inequalities valid.
However, the neat notation for defining sub-Gaussian norm sometime leads to unknown constants in the CIs as shown next.
% for the weighted sum of $n$ independent sub-Gaussian r.vs contains unknown constant, if we merely use the sub-Gaussian norm condition.
\begin{corollary}[Theorem 2.6.3, \cite{Vershynin18}]\label{coro:ghd}
 Let $\{ {X_i}\} _{i = 1}^n$ be independent mean-zero sub-Gaussian, $\forall~t \geq 0,~{P}\{|\frac{1}{n}\sum_{i=1}^{n} X_{i}| \geq t\} \leq 2 e^{-{C(n t)^{2}}/{\sum_{i=1}^{n}\|X_{i}\|_{\psi_{2}}^{2}}},~\forall~t \geq 0$ for a constant $C$.
\end{corollary}
The unknown constant $C$ makes the above CIs cannot be used in constructing confidence bands for $\mu$. To obtain more specific bounds (data dependent bounds as a statistics), we adopt the follow three propositions under sub-Gaussian. %  it specifics the constant in Corollary \ref{Sub-gaussian distribution}.

\begin{proposition}[Sub-Gaussian properties]\label{Sub-gaussianproperties}
Let $X \sim \operatorname{subG}(\sigma^{2})$, then for any $t>0,$
\begin{enumerate}[\rm{(}a\rm{)}]
\item \label{p: sub-gaussiantail}
the tail satisfies $P(|X|>t) \leq 2e^{-{t^{2}}/{2 \sigma^{2}}} $;

\item (a) implies that moments $\mathrm{E}|X|^{k} \leq(2 \sigma^{2})^{k / 2} k \Gamma(\frac{k}{2})$ and
$(\mathrm{E}(|X|^{k}))^{1 / k} \leq \sigma e^{1 / e} \sqrt{k},~k \geq 2$;

\item If \eqref{p: sub-gaussiantail} holds and $\E X = 0$, then $\mathrm{E}e^{s X}\leq e^{4 \sigma^{2} s^{2}}$ for any $s>0$;

\item If $X \sim \operatorname{subG}\left(\sigma^{2}\right)$, then $\|X\|_\psitwo \le \frac{{2\sqrt 2}}{{\sqrt {\log 2} }}\sigma$; conversely, if $\|X\|_\psitwo = \sigma$ then $X \sim \operatorname{subG}(4\sigma^{2})$.
\end{enumerate}
\end{proposition}
\begin{proof}
The proofs of (a)-(c) are in Lemma 1.4 and 1.5 in \cite{Rigollet19}). The proofs of (a, b) is similar to Proposition \ref{pro-moment}(a, b) below. For (d), note that %consider the Taylor's expansion of the exponential function,
\begin{align}\label{eq:evenmoment}
\mathrm{E} \exp \left(s^{2} X^{2}\right)=1+\sum_{k=1}^{\infty} \frac{s^{2 k} \mathrm{E}X^{2 k}}{k !}&\stackrel{\rm (b)}{\le} 1+\sum_{k=1}^{\infty} \frac{2s^{2 k} (2 \sigma^{2})^{k} k \Gamma(k)}{k !}=1+4s^{2}\sigma^{2}\sum_{k=0}^{\infty} (2s^{2}\sigma^{2})^{k}\nonumber\\
&\xlongequal{\forall~|2s^{2}\sigma^{2}|<1}1+ \frac{4s^{2}\sigma^{2}}{1-2s^{2}\sigma^{2}}
\stackrel{\forall~|s| \le 1/({2}\sigma)}{\le} 1+ 8s^{2}\sigma^{2} \le e^{8 \sigma^{2} s^{2}}.
\end{align}
By \eqref{eq:evenmoment}, set $\mathrm{E} e^{s_0^{2} X^{2}}\le {e^{8{{s_0}^{\rm{2}}}\sigma^2}} \le {\rm{2}}$ for some $s_0$. Then $s_0^2 \le \frac{{\log 2}}{{8\sigma^2}}$ and
$|s_0| \le \frac{{\sqrt {\log 2} }}{{2\sqrt 2\sigma}} \le \frac{1}{2\sigma}$. Put $|{s_0}| =\frac{{\sqrt {\log 2} }}{{2\sqrt 2\sigma}} $ and the sub-Gaussian norm gives
$ {\rm{E}}{e^{ {X^2}/(\frac{{2\sqrt 2\sigma}}{{\sqrt {\log 2} }})^2}} \le {\rm{2}} \Rightarrow {\left\| X \right\|_{{\varphi _2}}} \le \frac{{2\sqrt 2\sigma}}{{\sqrt {\log 2} }}.$

Conversely, if $\|X\|_\psitwo= \sigma$ then \eqref{eq: sub-gaussian tail} gives
$$P(|X| > t) \le 2{e^{ - {t^2}/{\sigma ^2}}}{\rm{ = }}2{e^{ - \frac{{{t^2}}}{{2{{(\sigma /\sqrt 2 )}^2}}}}}.$$
Then Proposition \ref{Sub-gaussianproperties}(c) concludes $\mathrm{E}e^{s X}\leq {e^{4{{(\sigma /\sqrt 2 )}^2}{s^2}}} = {e^{4{\sigma ^2}{s^2}/2}}$ for any $s>0$, and we have $X \sim \operatorname{subG}(4\sigma^{2})$.
\end{proof}

Let $\{ {Y_i}\} _{i = 1}^n$ be a sequence of \textit{exponential family} (EF) r.vs, its density  % to canonical exponential family
\begin{equation}\label{eq:E-P}
f(y_i;\theta_i ) = c(y_i)\exp \{y_i\theta_i  - b(\theta_i )\}
\end{equation}
with ${\rm{E}}Y_i=\dot{b}(\theta_i)$ and ${\rm{Var}}Y_i=\ddot{b}(\theta_i)$. We next introduce the sub-Gaussian CIs for the non-random weighted sum of EF r.vs with compact parameter space, adapted from Lemma 6.1 in \cite{Rigollet12} with more specific constants.
\begin{proposition}[Concentration for weighted E-F summation] \label{pro-moment}
 We assume \eqref{eq:E-P} and
\begin{itemize}
\item [\textbullet]\emph{(E.1)}: Uniformly bounded variances condition: there exist a compact set $\Omega$ and some constant $C_{b}$ such that $\mathop {\sup }_{\theta_i  \in \Omega } \ddot b (\theta_i ) \le {C_{ b}^2}~\text{for all}~i$.
\end{itemize}
Let $\bm w := ({w_1}, \cdots ,{w_n})^T \in {\mathbb{R}^n}$ be a non-random vector and define $S_n^w = :\sum_{i = 1}^n {{w_i}{Y_i}} $. Then
\begin{enumerate}[\rm{(}a\rm{)}]
\item Closed under addition: $S_{n}^{w}-{\rm{E}}S_{n}^{w} \sim {\rm{subG}}({C_b^{2}\|\bm w\|_{2}^{2}})$ and $P\{|S_{n}^{w}-{\rm{E}}S_{n}^{w}|>t\}\leq2 e^{-{t^{2}}/(2C_b^{2}\|\bm w\|_{2}^{2})}$;
%Exponential moments bound: $\EB\exp \left(s\left(S_{n}^{\omega}-\EB\left(S_{n}^{\omega}\right)\right)\right)\le \exp \left(\frac{s^{2} C_b^{2}\left\| \bm w \right\|_{2}^{2}}{2}\right)$ for $s \in \mathbb{R}^n$;
\item Let $C_n:=S_{n}^{w}-{\rm{E}}S_{n}^{w}$ and $\Gamma (t):=\int_{0}^\infty x^{t-1} e^{-x} dx$ be the Gamma function. For all integer $k \ge 1$, we have moments bound:
    ${\rm{E}}|C_n|^k \le k{(2{C_b^2})^{k/2}}\Gamma (k/2)\left\|\bm w \right\|_2^k.$
\item The MGF of centralized $|C_n|^2$:
${\rm{E}}e^{s[|C_n|^2-{\rm{E}}|C_n|^2]} \le  e^{{s^{\rm{2}}}{{({\rm{8}}\sqrt {\rm{2}} C_b^2\left\| w \right\|_2^2)}^2}/2},~\forall ~ | s |  \le ({{\rm{8}}C_b^2\|\bm w \|_2^2})^{-1}.$

\item \emph{In this case, we do not assume \eqref{eq:E-P} and (E.1)}. Suppose $\{ Y_{i}-{\rm{E}}Y_{i}\} _{i = 1}^n $ are independent distributed as $\{\operatorname{subG}(\sigma_i^{2})\} _{i = 1}^n$ with ${C_b^2}=:\mathop {\max }_{1 \le i \le n} \sigma_i^{2}>0$, then $\rm{(a)-(c)}$ also hold.
\end{enumerate}
\end{proposition}
\begin{proof}
Based on the MGF and uniformly bounded variances condition, the proof of (a) can be found in Lemma 6.1 of \cite{Rigollet12}. In the proof of (c), we update the constant, and (d) is similar for the sub-Gaussian case.

\item (b) The proof relies on expectation formula for positive r.v. (in terms of integral of tail probability) which transforms tail bound to moment bound. For any integer $k \ge 1$,
\begin{align}\label{eq:EFmgf}
\EB\left| S_{n}^{w}-\EB S_{n}^{w} \right|^k& = \int_{0}^\infty P(  \left| S_{n}^{w}-\EB S_{n}^{w} \right|^k \ge s ) ds\xlongequal{t =s^{1/k}} \int_{0}^\infty k t^{k-1}P\left(  \left| S_{n}^{w}-\EB S_{n}^{w} \right| \ge t \right) dt.
\end{align}
Applying tail bound in (a), we have by letting ${D_{k,C}} = k{(2{C_b^2})^{k/2}}\Gamma (k/2)$
\begin{align*}
\EB\left| S_{n}^{w}-\EB S_{n}^{w} \right|^k &\le 2k  \int_{0}^\infty t^{k-1} e^{-\frac{t^{2}}{2C_b^{2}\|\bm w\|_{2}^{2}}} dt \xlongequal{x ={t^{2}}/{2C_b^{2}\|\bm w\|_{2}^{2}}} k (2C_b^2)^{\frac{k}{2}}  \|\bm w\|_{2}^{k} \int_{0}^\infty x^{{\frac{k}{2}}-1} e^{-x} dx={D_{k,C}}\left\| \bm w \right\|_2^k.
\end{align*}

\item (c) The proof will resort to $(\E|Z|)^{{k}} \leq \E|Z|^{k}$ and Jensen's inequality \begin{align}\label{eq:Jensen}
    {(\frac{{|a| + |b|}}{2})^k} \le \frac{1}{2}{|a|^k} + \frac{1}{2}{|b|^k},~\text{for integer}~k \ge 1.
 \end{align}
From Taylor's expansion, \eqref{eq:Jensen} gives
\begin{align*}\label{eq:EFmgf}
{\rm{E}}e^{s[|C_n|^2-{\rm{E}}|C_n|^2]}& =  1 + \sum_{k=2}^\infty \frac{s^{k} \E [|C_n|^2-\E|C_n|^2]^{k}}{k!}\le 1 + \sum\limits_{k = 2}^\infty  {\frac{{{s^k}{{\rm{2}}^{k - 1}}{\rm{E}}\{ {|C_n|^{2k} + {{\left( {{\rm{E}}|C_n|{^2}} \right)}^k}} \}}}{{k!}}}\\
[\text{By}~(\E|Z|)^{{k}} \leq \E(|Z|^{k})]~~&\le1 + \sum\limits_{k = 2}^\infty  {\frac{{{s^k}{{\rm{2}}^{k - 1}}{\rm{E}}\left\{ {|C_n|^{2k}+ {\rm{E}}|C_n|^{2k}} \right\}}}{{k!}}}\le1 + \sum\limits_{k = 2}^\infty  {\frac{{{s^k}{{\rm{2}}^k} \cdot 2k{{(2C_b^2\left\|\bm w \right\|_2^2)}^k}\Gamma (k)}}{{k!}}},
\end{align*}
where the last inequality is by Proposition~\ref{pro-moment}(b). Then, under $| {4sC_b^2\left\| \bm w \right\|_2^2} | < {\rm{1}}$, we have
\begin{align*}
{\rm{E}}e^{s[|C_n|^2-{\rm{E}}|C_n|^2]}&= 1 + 2\sum\limits_{k = 2}^\infty  {{{(4sC_b^2\left\|\bm w \right\|_2^2)}^k}}  = 1 + \frac{{2{{(4sC_b^2\left\| \bm w \right\|_2^2)}^2}}}{{1 - 4|s|C_b^2\left\| \bm w \right\|_2^2}}\\
[| {4sC_b^2\| \bm w \|_2^2} | \le \textstyle{{\rm{1}}\over{{\rm{2}}}} \Leftrightarrow | s |~ \le \textstyle{{{\rm{1}}}\over{{{\rm{8}}C_b^2\|\bm w \|_2^2}}}]~~&\le 1 + \frac{{{s^{\rm{2}}}{{({\rm{8}}\sqrt {\rm{2}} C_b^2\left\|
\bm w \right\|_2^2)}^2}}}{{\rm{2}}} \le {{\rm{e}}^{{s^{\rm{2}}}{{({\rm{8}}\sqrt {\rm{2}} C_b^2\left\| \bm w \right\|_2^2)}^2}/2}}.
\end{align*}

\item (d) It follows by defining ${C_b^2}=:\mathop {\max }\limits_{1 \le i \le n} \sigma_i^{2}>0$ as the common variance proxy for $\{ {Y_i}\} _{i = 1}^n$. For $i=1,2,\cdots,n$, we have:
$\mathrm{E}e^{sw_i(Y_{i}-{\rm{E}}Y_{i})}\leq e^{\sigma^{2} s^{2}w_i^{2}/{2}}$, $\forall s \in \mathbb{R}$.
\end{proof}
Proposition~\ref{pro-moment}(a) yields the following results (The first result is in Corollary 1.7 of \cite{Rigollet19}). The second sub-Gaussian CI below specifies the unknown constant in Theorem 2.6.2 of \cite{Vershynin18}.
\begin{proposition}\label{Sub-gaussianConcentration}
Let $\{X_{i}\}_{i=1}^n$ be $n$ independent $\operatorname{subG}(\sigma_i^{2})$. Define ${\sigma ^2} = \mathop {\max }_{1 \le i \le n} \sigma _i^2$,
\begin{center}
$P( {| {\sum\limits_{i = 1}^n {{w_i}} {X_i}}| > t} ) \le 2e^{{ - {{{t^2}}}/({2{\sigma ^2}\|\bm w\|_2^2})}}$ and ${P}( {| {\sum\limits_{i = 1}^n {{w_i}} {X_i}}| > t} ) \le 2 e^{-{ t^{2}}/(8\sum_{i=1}^{n}\|{w_i}X_{i}\|_{\psi_{2}}^{2})},~\forall~t \geq 0$
\end{center}
for any non-random vector $\bm w := ({w_1}, \cdots ,{w_n})^T$.
\end{proposition}
\begin{proof}
To see the second CI, just use the Proposition \ref{Sub-gaussianproperties}(d) and the Proposition~\ref{pro-moment}(d), by noticing that if $\|X_i\|_\psitwo < \infty$ then ${w_i} X_i \sim \operatorname{subG}(4\|{w_i} X_i\|_\psitwo^2)$.
\end{proof}
\subsection{Randomly weighted sum of independent sub-Gaussian variables}\label{se;weights}
In this part, we outline the sub-Gaussian type CIs for the randomly weighted sum of exponential family r.vs:
$S_n^{W } = :\sum_{i = 1}^n {{W_i}{Y_i}}$
where $\{W_i\}_{i = 1}^n$ are called the multipliers (or random weights) which are independent from $\{Y_i\}_{i = 1}^n$. The normalized sum $\frac{1}{{\sqrt n }}(S_n^W - {\rm{E}}S_n^W)$ is also call \emph{multiplier empirical processes}, and it serves for the multiplier Bootstrap inference where the multipliers $\{W_i\}$ are r.vs independent from $\{Y_i\}_{i = 1}^n$, see Chapter 2.9 of \cite{van96}. To get sub-Gaussian concentration, some regularity conditions for the parameter space are required.
\begin{itemize}
\item [\textbullet](E.2): Let $\bm W := ({W_1}, \cdots ,{W_n})^T \in {\mathbb{R}^n}$ be a random vector with some bounded components, i.e. $| W_{i}| \leq w_i <\infty$ for a non-random vector $\bm w := ({w_1}, \cdots ,{w_n})^T \in {\mathbb{R}^n}$.
\end{itemize}

\begin{theorem}[Concentration inequalities for randomly weighted sum] \label{lem-weighted}
 Let $\{ {Y_i}\} _{i = 1}^n$ belong to the canonical exponential family \eqref{eq:E-P}, and let $\{ {W_i}\} _{i = 1}^n$ be independent of $\{ {Y_i}\} _{i = 1}^n$. Define the randomly weighted sum $S_n^{W } = :\sum_{i = 1}^n {{W_i}{Y_i}} $, then under \rm{(E.1)} and \rm{(E.2)}
\[P( {|S_n^W - {\rm{E}}S_n^W| \ge t} ) \le  2e^{-{ t^{2}}/({2C_{ b}^{2}\|\bm w\|_{2}^{2}})}.\]
\end{theorem}
\begin{proof}
Let $Y_i = \dot{b}(\theta_i)+ Z_i$, where $\{ {Z_i}\} _{i = 1}^n$ are centralized and independent E-F r.vs. From ${\rm{E}}Y_i=\dot{b}(\theta_i)$ and the identity \eqref{eq: identity} for a dominating measure $\mu(\cdot)$
\begin{align}\label{eq: identity}
\smallint {dF_{Y_i}(y)}  = 1 \Leftrightarrow \smallint  {c({y}){e^{{y}{\theta _i}}}\mu (dy)}  =e^{b({\theta _i})}.
\end{align}
Let ${\rm{E}}_{\cdot|\bm W}(\cdot):={\rm{E}}(\cdot |\bm W)$ and $s$ be in $(-\delta,\delta )$ (a neighbourhood of zero). Then,
\begin{align*}
{\rm{E}}_{\cdot|\bm W}[{\rm{e}}^{s{W_i}{Y_i}}] & =  \int {\rm{e}}^{s{W_i}{Y_i}}dF_{Y_i|\bm W}(y)=  \int {\rm{e}}^{s{W_i}{Y_i}}dF_{Y_i}(y)~~[\text{by}~\{W_i\}_{i = 1}^n \bot \{Y_i\}_{i = 1}^n] \\
& =  \int {c(y)}e^{y{\theta _i} - b ({\theta _i})}e^{s{W_i}y}\mu(dy)\xlongequal{\eqref{eq: identity}} e^{ b ({\theta _i} + s{W_i}) - b ({\theta _i})}.
\end{align*}
It can be easily derived from (E.2) and Taylor's expansion,
\begin{align}\label{eq:mgf}
{\rm{E}}_{\cdot|\bm W}[{\rm{e}}^{s({W_i}{Y_i}-{\rm{E}}_{\cdot|\bm W}({W_i}{Y_i}))}] &= e^{ b ({\theta _i} + s{W_i}) - b ({\theta _i}) - \dot b ({\theta _i}){W_i}s}\xlongequal{\exists {\tilde \eta _{i}} \in [{\theta _i},{\theta _i} + s{W_i}]}
e^{\frac{{{s^2}W_i^2}}{2}\ddot b (\tilde \eta _i)}\le e^{\frac{{{s^2}{C_{b}^2}W_i^2}}{2}}.
\end{align}
By the conditional independence for $\{{W_i}{Z_i}|\bm W\}$ and \eqref{eq:mgf}, it follows that when $s \in (-\delta,\delta )$
\begin{align}\label{eq:Cmgf}
{{\rm{E}}_{ \cdot |\bm W}}[e^{s\sum_{i = 1}^n [ {W_i}{Z_i} - {\rm{E}}_{\cdot|\bm W}({W_i}{Z_i})]} ]&=\prod\limits_{i = 1}^n {{{\rm{E}}_{ \cdot |\bm W}}e^{ s[{W_i}{Z_i} - {\rm{E}}_{\cdot|\bm W}({W_i}{Z_i})]} } \le \prod_{i=1}^n e^{\frac{ s^2C_{b}^2W_i^2}{2}} \le 2e^{\frac{ s^{2} C_{b}^{2}\left\| \bm w \right\|_{2}^{2}}{2} },
\end{align}
where the last inequality is from $\{| W_{i}| \leq w_i\}$ for a non-random vector $\bm w := ({w_1}, \cdots ,{w_n})^T$.

By the conditional Markov's inequality and symmetry of $Z_i$, we have, as $s \in (-\delta,\delta )$
\begin{align}\label{eq:P2}
P(|\sum{}_{i = 1}^n [ {W_i}{Z_i} - {\rm{E}}_{\cdot|\bm W}({W_i}{Z_i})]| \ge   t|\bm W )&\le
\mathop {\inf }\limits_{s > 0}[ {e^{-s t}}{ {{\rm{E}}_{ \cdot |\bm W}}e^{ s(\tilde S_{n}^{W}-{\rm{E}}_{\cdot|\bm W}\tilde S_{n}^{W})}}+{e^{-s t}}{ {{\rm{E}}_{ \cdot |\bm W}}e^{ -s(\tilde S_{n}^{W}-{\rm{E}}_{\cdot|\bm W}\tilde S_{n}^{W})}}]\nonumber\\
&\le 2\mathop {\inf }\limits_{s > 0}e^{\frac{ s^{2} C_{b}^{2}\left\| \bm w \right\|_{2}^{2}}{2} - s t }=2e^{-\frac{t^{2}}{2C_{b}^{2}\|\bm w\|_{2}^{2}}}
\end{align}
where the last equality is minimized by setting $s = t/(C_{b}^2\|\bm w\|_2^2)$.
\end{proof}

Note that Lemma 6.1 in \cite{Rigollet12} is about the concentration for the non-random weighted sum of exponential family r.vs. The assumption of compact parameter space for exponential family is vital for obtaining the sub-Gaussian type concentration. If we do not impose condition (E.2) and the assumption that $\{W_i\}_{i=1}^n$ and $\{Y_i\}_{i=1}^n$ are dependent, a counterexample for sub-Gaussian concentration is $W_i=Y_i$. Thus, $S_n^W$ is a quadratic form, and $S_n^W-{\rm{E}} S_n^W$ is sub-exponential by Lemma~\ref{lem: sub-exponential squared} below. If $\{W_i\}_{i=1}^n$ and $\{Y_i\}_{i=1}^n$ are dependent but $\{W_i\}_{i=1}^n$ are still bounded, another counterexample is $W_i=\rm{sign}(Y_i)$. Therefore, $S_n^{W } =\sum_{i = 1}^n {|Y_i|} $ is not zero-mean, and the concentration of $\sum_{i = 1}^n {|Y_i|} $ fails.

\subsection{Concentration for Lipschitz functions of random vectors}\label{Lipschitz}
In the analyses of high-dimensional statistics by empirical processes, researches often resort to the CIs of Lipschitz functions for either bounded or strongly log-concave random vectors \citep{Wainwright19}.

\begin{lemma}[Theorem 2.26, \cite{Wainwright19}]\label{lem:caussiancon}
Let $\bm N \sim N(\bm 0, \mathrm{I}_p)$. Let $f:\mathbb{R}^n \rightarrow \mathbb{R}$ be L-Lipschitz with respect to (w.r.t.) the Euclidean norm: $| f ( \bm a ) - f ( \bm b ) | \leq L \| \bm a - \bm b \| _ { 2 }$ for any $\bm a , \bm b \in \mathbb { R } ^ { n }$. Then, $P\left(|f(\bm N)-\mathrm{E} f(\bm N) |\geq t\right) \leq 2e^{-t^2 /(2L^2)},~\forall~t>0$.
\end{lemma}

A non-negative function $f(\bm x):\mathbb{R}^n \rightarrow \mathbb{R}$ is \emph{log-concave} if
\begin{align}\label{eq:log-concave}
\log f(\lambda \bm x+(1-\lambda) \bm y) \geq {\lambda}\log f(\bm x)+ ({1-\lambda})\log f(\bm y),~\forall~\lambda \in[0,1]~\text{and}~\bm x, \bm y \in \mathbb{R}^{n}.
\end{align}
A function $\psi(\bm x)
: \mathbb{R}^{n} \rightarrow \mathbb{R}$ is $\gamma$-\emph{strongly concave} if there is some $\gamma>0$ s.t.
\begin{center}
$
\lambda \psi(\bm x)+(1-\lambda) \psi(\bm y)-\psi(\lambda \bm x+(1-\lambda) \bm y) \le \frac{\gamma}{2} \lambda(1-\lambda)\|\bm x-\bm y\|_{2}^{2},~\forall~\lambda \in[0,1]~\text{and}~\bm x, \bm y \in \mathbb{R}^{n},
$
\end{center}
A continuous probability density $f(\bm x)$ and the corresponding r.v. are log-concave (or strongly log-concave) if $f(\bm x)$ is a log-concave function (or strongly log-concave function), see \cite{Saumard14} for a review of the log-concavity in statistics.

\begin{lemma}[Theorem 3.16, \cite{Wainwright19}]\label{3.16}
Let $\mathbb{P}$ be any $\gamma-$strongly log-concave distribution on $\mathbb{R}^n$ with parameter $\gamma>0$. Then for any function $f: \mathbb{R}^{n} \rightarrow \mathbb{R}$ that is $L$-Lipschitz w.r.t. the Euclidean norm, we have
\begin{center}
${P}[f(X)-\mathrm{E}f(X)  \geq t] \leq  e^{-\frac{\gamma t^{2}}{4 L^{2}}}$ for $X \sim \mathbb{P}$ and $t\ge 0$.
\end{center}
\end{lemma}

The standard Gaussian random vector is $1-$strongly log-concave distributed. However, Lemma~\ref{lem:caussiancon} has the sharper constant ${2L^{2}}$ than the Gaussian case of Lemma \ref{3.16} with constant ${4 L^{2}}$. Beyond Gaussian and strongly log-concave, it is possible to establish concentration for distributions involving bounded r.vs. A function $f(\bm x): \mathbb{R}^{n} \rightarrow \mathbb{R}$ is said to be \emph{separately convex} if, the univariate function $y_{k} \mapsto f\left(x_{1}, x_{2}, \ldots, x_{k-1}, y_{k}, x_{k+1}, \ldots, x_{n}\right)$ for each index $k \in\{1,2, \ldots, n\},
$ is convex for each fixed vector $\left(x_{1}, x_{2}, \ldots, x_{k-1}, x_{k+1}, \ldots, x_{n}\right) \in \mathbb{R}^{n-1} .$
\begin{lemma}[Theorem 3.4, \cite{Wainwright19}]\label{3.17}
 Let $\left\{X_{i}\right\}_{i=1}^{n}$ be independent r.vs, each supported on the interval $[a, b]$. Let $f: \mathbb{R}^{n} \rightarrow \mathbb{R}$ be separately convex, and $L$-Lipschitz w.r.t. the Euclidean norm. Then,
$
{P}[f(X)- \mathrm{E}f(X) \ge t] \leq e^{-{t^{2}}/{4 L^{2}(b-a)^{2}}}.
$  for $X \sim \mathbb{P}$ and $t\ge 0$.
\end{lemma}
%Connecting to Wassersten distance, \cite{Wainwright19} extends Lemma~\label{3.17} to
\begin{example}[Order Statistics]
From Lemma \ref{3.16} and Lemma \ref{3.17}, suppose that  $\left\{X_{i}\right\}_{i=1}^{n}$ are independent r.vs which are $\gamma-$strongly log-concave distributed satisfying ${P}[f(X)-\mathrm{E}f(X)  \geq t] \leq  e^{-\frac{\gamma t^{2}}{4 L^{2}}}$ for any function $f: \mathbb{R}^{n} \rightarrow \mathbb{R}$ that is $L$-Lipschitz w.r.t. the Euclidean norm. Let $X_{(k)}$ be the $k$-th order statistic of $X_{1}, \ldots, X_{n}$, it can be shown that
\begin{center}
${P}(|X_{(k)}-\mathrm{E}X_{(k)}|\ge t) \leq 2  e^{-\frac{t^{2}}{4/\gamma }}$ by checking $|X_{(k)}-Y_{(k)}| \leq\|X-Y\|_{2}$, i.e. $L=1$.
\end{center}
Indeed, we have
\begin{center}
$X_{(k)}-Y_{(k)} \le \left|X_{l}-Y_{l}\right| \leq\|X-Y\|_{2}$ for some $l \in \{1,2,\cdots,n\}$.
\end{center}
More results of the tail bounds for the order statistics of IID r.vs are reported in \cite{Boucheron12}.
\end{example}
\section{Sub-exponential Distributions}\label{sec-se}
\subsection{Characterizations}

The requirement in definition of sub-Gaussian r.v. $\mathrm{E}e^{s X}\le e^{\frac{\sigma^{2} s^{2}}{2}},~\forall s \in \mathbb{R}$ is too strong. For example, we consider the MGF of exponential distributions:

\begin{example} [MGF of exponential distributions]\label{eg:Exponential}
Consider the exponential r.v. $X \sim {\rm{Exp}}(\mu)$ with $\mathrm{E}X=\mu>0$.  The MGF of $X-\mu$ satisfies
\begin{align}\label{eq:Exponentialmgf}
 {\rm{E}}{e^{s(X - \mu )}} = \frac{e^{ - s\mu }}{{1 - {\rm{s}}\mu } } = {({e^{ - s\mu/2 }}{\left( {1 - {\rm{s}}\mu } \right)^{ - 1/2}})^2} \le {e^{4{{(s\mu /2)}^2}}} < {e^{{s^2}(2\mu)^2/2}},
\forall ~ |s| \le (2\mu)^{-1}
\end{align}
where the second last inequality is by ${e^{-t}}/{\sqrt{1-2 t}} \leq e^{2 t^{2}}$ for $|t|\le 1/4$.
 (This is from the property of $f(t): = (1 - 2t){e^{4{t^2} + 2t}}$ with $f(0)=1$: (a). $f'(t) > 0,~0 < t < 1/4$; (b). $f(t) \ge 1, ~- 1/4 < t < 0$).
\end{example}

In \eqref{eq:Exponentialmgf}, the MGF of the exponential r.v. is divergent on $s=1/\mu$ and it cannot be bounded by a Gaussian MGF of $s$ in $\mathbb{R}$, and the exponential MGF is bounded by Gaussian MGF for $|s| \le \frac{1}{{2\mu }}$ via inequality \eqref{eq:Exponentialmgf}. Motivated by Example \ref{eg:Exponential}, the first definition of sub-exponential distribution \eqref{eq:sube1} below is exactly the locally sub-Gaussian property.

\begin{definition}[Sub-exponential distributions, \cite{Cramer1938}]\label{def:Sub-exponential}
 A r.v. $X \in \mathbb{R}$ with $\E X = 0$ is sub-exponential with
variance parameter $\lambda$ (denoted $X \sim \operatorname{subE}(\lambda))$ if its MGF satisfies
\begin{equation}\label{eq:sube1}
 {\mathrm{{E}}}e^{s X} \leq e^{\frac{s^{2}\lambda ^{2}}{2}} \quad
\text{for all } |s| <{1}/{\lambda}.
\end{equation}
\end{definition}
In \cite{Wainwright19}, sub-exponential distributions are generally defined by two positive parameters $(\lambda, \alpha)$ (denoted $X \sim \operatorname{subE}(\lambda, \alpha)$):
${\mathrm{{E}}}e^{s X} \leq e^{\frac{s^{2}\lambda ^{2}}{2}}~\text { for all }|s|<{1}/{\alpha}.$ The minimal variance variance
\begin{center}
$\tau^2_{a,opt}(X):= \inf \big\{ \tau^2 > 0 : {\mathbb{{E}}}\exp(tX) \leq \exp\{\tau^2 t^2 / 2\}, \, \forall \, |t|<a^{-1} \big\}$
\end{center}
 is called as \emph{optimal variance parameter}; see \cite{Buldygin2000}.

The $\lambda^2$ in \eqref{eq:sube1} is treated as a \emph{variance proxy} and ${\alpha}$ is seen as \emph{locally sub-Gaussian factor}, see Remark \ref{re:locally sub-Gaussian} later. Specifically, $\operatorname{subE}(\lambda)=\operatorname{subE}(\lambda, \lambda)$. Sub-Gaussian r.vs are sub-exponential by definition, but not vice verse. In Corollary \ref{Sub-gaussian distribution}, one equivalence of sub-Gaussian r.vs is that the survival function is bounded by the Gaussian-like survival function up to a constant. Similarly,  the sub-exponential r.v. has a characterization that the survival function is bounded by that of a exponential distribution. Similar to sub-Gaussian characterizations, there are at least six equivalent forms for sub-exponential distributions which are useful for checking the sub-exponential distribution.
\begin{corollary}[Characterizations of sub-exponential]\label{prop: sub-exponential properties}
Let $X$ be a r.v. in $\R$ with $\E X = 0$. Then the following properties are equivalent, where $\{K_i\}_{i=1}^6$ are  positive constants.
\begin{enumerate}[\rm{(}1\rm{)}]
    \item \label{p: exponential tail}
      The tails of $X$ satisfy
      $
      P \{ |X| \ge t \} \le 2 e^{-t/K_1}~\text{for all } t \ge 0;
      $
   \item \label{p: exponential MGF}
    The MGF of $X$ satisfies
      $
      \E e^{\l X} \le e^{K_2^2 \l^2}~\text{for all}~|\l| \le \frac{1}{K_2};
      $
    \item \label{p: exponential moments}
      The moments of $X$ satisfy
      $
     (\E |X|^p)^{1/p} \le K_3 p~\text{for  integer}~p \ge 1;
      $
    \item \label{p: exponential MGF abs}
      The MGF of $|X|$ satisfies
      $
      \E e^{\l |X|} \le e^{K_4 \l}~\text{for all } 0 \le \l \le \frac{1}{K_4};
      $
    \item \label{p: exponential MGF finite}
      The MGF of $|X|$ is bounded at some point:
      $\E e^{|X|/K_5} \le 2$;
      \item \label{p: exponential MGF finite2}
  Bounded  MGF of $X$ in a compact set:
      $
      \E e^{tX} < \infty,~\forall |t|<  1/K_6.
      $

  \end{enumerate}

\end{corollary}

The zero mean is only used in the proof of \eqref{p: exponential MGF} of Corollary \ref{prop: sub-exponential properties}. The equivalence among \eqref{p: exponential tail}--\eqref{p: exponential MGF finite} is proved in \cite{Vershynin18} and that between \eqref{p: exponential MGF finite} and \eqref{p: exponential MGF finite2} can be found in Lemma 5 of \cite{Petrov1995}. The \eqref{p: exponential MGF finite2} is the called \emph{Cramer's condition} which is an essential characterization, it signifies that: \emph{All r.vs. are sub-exponential if their MGF exist in a neighborhood of zero.} \cite{Pistone1999} names the property \eqref{p: exponential MGF finite2} as the \emph{exponentially integrable} r.v.

\begin{example}[Moment of exponential distributions]\label{eg:conExponential}
The $P(X - \mu  \ge t) = {e^{ - (t + \mu )/\mu }} \le {e^{ - t/\mu }}$  and the symmetry of $X - \mu$ implies $K_1=\mu$ in Corollary \ref{prop: sub-exponential properties}. Continue to Example \ref{eg:Exponential}, the ``$\le$'' in \eqref{eq:Exponentialmgf} implies $ {\rm{E}}{e^{s(X - \mu )}} \le {e^{{{(s\mu /{\sqrt 2})}^2}}}\le {e^{{{(2s\mu )}^2}}},
\forall ~ |s| <(2\mu)^{-1}$. So $ K_2=\mu/\sqrt 2 $ and $ K_6=2\mu$ in Corollary \ref{prop: sub-exponential properties}. Next, we evaluate the moment of $X$ for any $p \geq 1,$
$
\mathrm{E}|X|^{p}=\int_{0}^{\infty} x^{p} \cdot \mu^{-1} e^{-\mu^{-1} x} d x\xlongequal{y=\mu^{-1} x} {\mu^{p}} \int_{0}^{\infty} y^{p} e^{-y} d y={\Gamma(p+1)}{\mu^{p}}.
$ By $\Gamma(p+1) \leq p^{p}$ for $p \geq 1,$ it gives: $(\mathrm{E}|X|^{p})^{1 / p}={(\Gamma(p+1))^{1 / p}}{\mu} \leq {p}\mu$. Via \eqref{eq:Minkowski} shows that $(\mathrm{E}|X-\mu|^{p})^{1 / p}\leq 2{p}\mu$ and thus $ K_3=2\mu$ in Corollary \ref{prop: sub-exponential properties}. Assume ${\E}X=0$, then by Stirling's approximation $p ! \geq (p / e)^{p}$
\begin{align}\label{eq:subE6}
\E{e^{\lambda |X|}} = 1 + \sum\limits_{p = 2}^\infty  {\frac{{{\lambda ^p}\E|X{|^p}}}{{p!}}} & \le 1 + \sum\limits_{p = 2}^\infty  {\frac{{{{(\lambda {K_3}p)}^p}}}{{{{(p/e)}^p}}}}  = 1 + \sum\limits_{p = 2}^\infty  {{{(e{K_3}\lambda )}^p}}  = 1 + \frac{{{{(e{K_3}\lambda )}^2}}}{{1 - e{K_3}\lambda }},~\forall~|e{K_3}\lambda | < 1\nonumber\\
(\text{Restrict}~e{K_3}\lambda  \le 1/2) ~~& \le 1 + 2{(e{K_2}\lambda )^2} \le {e^{2{{(e{K_3}\lambda )}^2}}} \le  {e^{e{K_3}\lambda }} \le{e^{2e{K_3}\lambda }},~\forall~\lambda  \le 1/(2e{K_3}).
\end{align}
Thus $K_4=2e{K_3}=4e\mu$. That $\E{e^{\lambda |X|}} \le {e^{e{K_3}\lambda }}~\text{for}~0<\lambda  \le 1/(2e{K_2})$ in \eqref{eq:subE6} implies $\E{e^{|X|/(2e{K_3})}} < {e^{1/2}} < 2$. Hence $K_5={K_3}$.
\end{example}

\begin{example} [Geometric distributions]
{\color{black}{The \emph{geometric distribution} $X\sim \mathrm{Geo} (q) $ for r.v. $X$ is defined by:
$P(X=k) ={(1 - q)} {q^{k-1}},(q \in (0,1),~k=1,2,\cdots)$. The mean and the variance of ${\rm{Geo}} (q) $ are ${(1 - q)}/{{q}}~\text{and}~{(1 - q)}/{{{{q}^2}}}$ respectively. Apply
Lemma 4.3 in \cite{Hillar2013}, we get ${({\rm{E}}|X{|^k})^{1/k}} <{{-2k}}/{{ \log (1 - q)}}$. It follows from the Minkowski's inequality and Jensen's inequality $(\E|Z|)^{{k}} \leq \E|Z|^{k}$ for integer $k \ge 1$ that
\begin{equation}\label{eq:Minkowski}
(\mathrm{E}|X-\mathrm{E}X|^{k})^{1 / k} \le (\mathrm{E}|X|^{k})^{1 / k}+|\mathrm{E}X| \leq 2 (\mathrm{E}|X|^{k})^{1 / k}\le {{-4k}}/{{ \log (1 - q)}}
\end{equation}
and Corollary \ref{prop: sub-exponential properties}(3) implies the \textit{centralized $\mathrm{Geo} (q)$} is sub-exponential with $K_3={{-4}}/{{ \log (1 - q)}}$.}}
\end{example}
\begin{example}[Discrete Laplace r.vs]
A r.v. $X \sim \mathrm{DL}(q)$ obeys the discrete Laplace distribution if
$f_{q}(k)=\mathbb{P}(X=k)=\frac{1-q}{1+q} q^{|k|},~k \in \mathbb{Z}=\{0,\pm 1,\pm 2, \ldots\}$ with parameter $q \in(0,1).$ The discrete Laplace r.v. is the difference of two IID $\mathrm{Geo} (q) $. The geometric distribution is sub-exponential, thus Corollary \ref{sub-exponentialConcentration}(a) mentioned later implies that the discrete Laplace is also sub-exponential distributed. In differential privacy of network models, the noises are assumed following the discrete Laplace distribution, see \cite{Fan20} and references therein.
\end{example}

The next result shows that a sum of independent sub-exponential r.vs has two tails with difference convergence rate, which is slightly different from Hoeffding's inequality. Deviating from the mean, it tells us that the tail of the sum of sub-exponential r.vs behaves like a combination of a Gaussian tail and a exponential tail.
\begin{corollary}[Concentration for weighted sub-exponential sums]\label{sub-exponentialConcentration}
Let $\{ X_{i}\} _{i = 1}^n $ be independent $\{\operatorname{subE}(\lambda_i)\} _{i = 1}^n$ distributed with zero mean. Define $\lambda= \mathop {\max }_{1 \le i \le n} \lambda_i>0$ and the non-random vector $\bm w := ({w_1}, \cdots ,{w_n})^T \in {\mathbb{R}^n}$ with $w= \mathop {\max }_{1 \le i \le n} |w_i|>0$, we have
\begin{enumerate}[\rm{(}a\rm{)}]
\item Closed under addition: $\sum_{i = 1}^n {{w_i}{X_i}}  \sim {\rm{subE}}(\| \bm w \|_2{\lambda} )$;
\item
$P( {| {\sum\limits_{i = 1}^n {{w_i}{X_i}} }| \ge t} ) \le 2e^{ { - \frac{1}{2}( {\frac{{{t^2}}}{{\left\|\bm w \right\|_2^2{\lambda ^2}}} \wedge \frac{t}{{w\lambda }}})} } = \left\{ {\begin{array}{*{20}{c}}
{2{e^{ - {t^2}/2\left\|\bm w \right\|_2^2{\lambda ^2}}},~~~~0 \le t \le \left\|\bm w \right\|_2^2\lambda /w}\\
{2{e^{ - t/2w\lambda }},~~~~~~~t > \left\|\bm w \right\|_2^2\lambda /w}
\end{array}} \right..$
\item \emph{Let $\{ X_{i}\} _{i = 1}^n $ be independent zero-mean $\{\operatorname{subE}(\lambda_i,\alpha_i )\} _{i = 1}^n$ distributed}. Define $\alpha:= \mathop {\max }_{1 \le i \le n} \alpha_i>0$, $\| \bm\lambda \|_2:= (\sum_{i = 1}^n {\lambda _i^2} )^{1/2}$ and $\bar \lambda : = (\frac{1}{n}\sum_{i = 1}^n {\lambda _i^2} )^{1/2}$. Then $\sum_{i = 1}^n {{X_i}}  \sim {\rm{subE}}(\| \bm \lambda \|_2,\alpha)$ and
\begin{align}\label{eq:subEE}
P(|\frac{1}{n}\sum\limits_{i = 1}^n {{X_i}} | \ge t) \le 2{e^{ - \frac{1}{2}(\frac{{n{t^2}}}{{{{\bar \lambda }^2}}} \wedge \frac{{nt}}{\alpha })}}= \left\{ {\begin{array}{*{20}{c}}
2{e^{ - \frac{{n{t^2}}}{{2{{\bar \lambda }^2}}}}},~0 \le t \le \frac{{{{\bar \lambda }^2}}}{\alpha }\\
2{e^{ - \frac{{nt}}{{2\alpha }}}},~~~~t > \frac{{{{\bar \lambda }^2}}}{\alpha }
\end{array}} \right.,~\forall t \ge 0.
\end{align}
\end{enumerate}
\end{corollary}
\begin{remark}\label{re:locally sub-Gaussian}
The $(\frac{{n{t^2}}}{{{{\bar \lambda }^2}}} \wedge \frac{{nt}}{\alpha })$ in \eqref{eq:subEE} reveals that the smaller $\alpha$ (locally sub-Gaussian factor) leads to sharper sub-exponential concentration. The sub-exponential concentration tends to the sub-Gaussian concentration with variance proxy ${{\bar \lambda }^2}$ when ${\alpha } \to 0$, which coincides the locally sub-Gaussian definition for sub-exponential distribution in Definition \ref{def:Sub-exponential}.
\end{remark}
\begin{proof}
{\rm(a)} By definition of sub-exponential r.vs, ${\rm{E}}{e ^{s{w_i}{X_i}}} \le {e ^{{s^2}w_i^2\lambda _i^2/2}}~\forall~| s| \le \frac{{\rm{1}}}{{|w_i|{\lambda _i}}},\;i = 1,2, \cdots ,n$, and it implies ${\rm{E}}{e ^{s{w_i}{X_i}}} \le {e ^{{s^2}w_i^2\lambda _i^2/2}},~| s| \le \frac{{\rm{1}}}{{{w}{\lambda}}}$ for all $i$. By the independence among $\{ X_{i}\} _{i = 1}^n $,
\begin{center}
${\rm{E}}\exp \{ s\sum\limits_{i = 1}^n {{w_i}{X_i}} \}  = \prod\limits_{i = 1}^n {\rm{E}} {e^{s{w_i}{X_i}}} \le \exp \{ {s^2}\sum\limits_{i = 1}^n {w_i^2\lambda _i^2/2} \}  \le {e^{{s^2}\left\| \bm w \right\|_2^2{\lambda ^2}/2}},~| s| \le \frac{{\rm{1}}}{{{w}{\lambda}}}.$
\end{center}
{\rm(b)} Applying Chernoff's inequality, we get for $0<| s| \le \frac{{\rm{1}}}{{{w}{\lambda}}}$
\begin{equation}\label{eq:minimise}
P( {\sum\limits_{i = 1}^n {{w_i}{X_i}}  \ge t} ) \le {{\mathop{\rm e}\nolimits} ^{ - st}}{\rm{E}}e^{ s\sum\limits_{i = 1}^n {{w_i}{X_i}}}  \le e^ {\frac{{{s^2}\left\|\bm w \right\|_2^2{\lambda ^2}}}{2} - st}  = :{e^{g(s,t)}}.
\end{equation}
We minimise the upper bound in \eqref{eq:minimise} on the constraint $|s| <\frac{{\rm{1}}}{{{w}{\lambda}}}$. Note that $\frac{\partial }{\partial s}g(s,t) = s\left\| w \right\|_2^2{\lambda ^2} - t$ implies the stationary point $s_0 ={\frac{{{t}}}{{\left\| \bm w \right\|_2^2{\lambda ^2}}}}$. So we have ${\left. {g(s,t)} \right|_{s = {s_0}}} =  -\frac{{\rm{1}}}{{2{w}{\lambda}}}$.
\begin{itemize}
\item  If $s_0 \le \frac{t}{{w\lambda }}$, then the minimizer is $s=s_0$. Actually, the upper bound is $e^ { - \frac{{{t^2}}}{2{\left\| \bm w \right\|_2^2{\lambda ^2}}}} $.
\item If $s_0=\frac{{{t}}}{{\left\|\bm w \right\|_2^2{\lambda ^2}}} > \frac{1}{{w\lambda }}$, so the minimizer is obtained at the boundary $s=\frac{1}{{w\lambda }}$ with condition $\left\|\bm w \right\|_2^2 < \frac{{tw}}{\lambda }$. Hence the upper bound is bounded by
    $e^{ { - \frac{t}{{w\lambda }} + \frac{{\left\|\bm w \right\|_2^2}}{{2{w^2}}}} } \le e^{ - \frac{t}{{w\lambda }} + \frac{t}{{2w\lambda }}} = e^ { - \frac{t}{{2w\lambda }}} .$
\end{itemize}
Summarizing this two cases, we can get
\[P( {\sum\limits_{i = 1}^n {{w_i}{X_i}}  \ge t}) \le e^ { - \frac{1}{2}( {\frac{{{t^2}}}{{\|\bm w \|_2^2{\lambda ^2}}} \wedge \frac{t}{{w\lambda }}} )}\]
By the same argument, we will have $P({\sum\limits_{i = 1}^n {{w_i}{X_i}} }  \le -t ) \le \exp\{- \frac{1}{2}( {\frac{{{t^2}}}{{\|\bm w \|_2^2{\lambda ^2}}} \wedge \frac{t}{{w\lambda }}} )\}$ due to the symmetric property of mean zero random variables.

\noindent{\rm(c)}. The proof is similar to {\rm(b)}, see page 29 of \cite{Wainwright19}.
\end{proof}

Corollary \ref{sub-exponentialConcentration}(b) is due to Petrov, and it is also called Petrov's exponential inequalities, see \cite{Lin11}. Although Corollary \ref{sub-exponentialConcentration}(b,c) are non-asymptotically valid for any number of summands. Nevertheless, it also has asymptotical merit, which implies: \emph{Strong Law of Large Numbers} (SLNN), \emph{Central Limit Theorem} (CLT), and \emph{Law of the Iterated Logarithm} (LIL) for sub-exponential sums, as discussed below.

\begin{enumerate}[\rm{(}1\rm{)}]
\item \textbf{SLNN}. Let $w_i=1/n$. Consider the sample mean ${{\bar X}_n}=\frac{1}{n} \sum_{i=1}^n X$ for IID $\{\operatorname{subE}(\lambda_i)\} _{i = 1}^n$ data $\{X_i\}_{i=1}^n$ with population mean $\mu$, and we can use Corollary~\ref{sub-exponentialConcentration}(b) to prove that ${{\bar X}_n} \overset{a.s.}{\rightarrow}\mu$. We verify the Borel-Cantelli lemma by observing that
$\sum_{n=1}^{\infty}P(|{{\bar X}_n}-\mu|>\varepsilon)\leq\sum_{n=1}^{\infty}2e^{ { - \frac{n}{2}( {\frac{{{\varepsilon^2}}}{{{\lambda ^2}}} \wedge \frac{\varepsilon}{\lambda }} )} }<\infty,$ which shows the strong convergence: ${{\bar X}_n} \overset{a.s.}{\rightarrow}\mu$. Corollary~\ref{sub-exponentialConcentration}(b) also implies the rate of convergence for sample mean for all $n$ with a high probability. It is easy to see that the sample mean ${{\bar X}_n}$ has the non-asymptotic error bounds by
\begin{equation}\label{eq:subEclt}
| {{\bar X}_n}-\mu| \le \sqrt {\frac{{2{\lambda ^{\rm{2}}}t}}{n}}  \vee \frac{{2\lambda t}}{n}=\begin{cases}
  	 \sqrt {\frac{{2{\lambda ^{\rm{2}}}t}}{n}} , & n \ge 2t~(\text{slow global rate})  \\
	\frac{{2\lambda t}}{n},  & n < 2t~(\text{fast local rate})
	\end{cases}
\end{equation}
$\forall t> 0$ with the probability at least $1-2e^{-t}$.

\item \textbf{CLT}. To study the convergence rate of CLT, we standardize the sum by letting $w_i=1/\sqrt n$ and apply Corollary \ref{sub-exponentialConcentration}(b) to
$$
P (|\sqrt n {{\bar X}_n}| \ge t )
  \le 2\exp \left\{ { - \frac{1}{2}\left( {\frac{{{t^2}}}{{{\lambda ^2}}} \wedge \frac{t}{{\lambda /\sqrt n }}} \right)} \right\}=\begin{cases}
  	2 e^{-ct^2/{\lambda ^2}}, & t \le  \lambda \sqrt n;  \\
	2 e^{-t \sqrt{n}/{\lambda }},  & t > \lambda \sqrt{n}.
	\end{cases}
$$
The above deviation inequality is powerful as it indicates the phase transition about the tail behavior of $\sqrt n {{\bar X}_n}$: {\textbf{Small Deviation Regime}.} In the regime $t \le \lambda \sqrt n$, we have
a sub-Gaussian tail bound with variance proxy ${\lambda ^2}$ as if the sum had the {\em normal distribution} with a constant variance. Note that the domain $t \le  \lambda \sqrt n$ widens as $n$ increases and then the central limit theorem becomes more powerful. \textbf{Large Deviation Regime}. In the regime $t \ge \lambda \sqrt n$, the sum has a heavier tail. The sub-exponential tail bound is affected from \emph{the extreme variable among $\{\operatorname{subE}(\lambda_i)\} _{i = 1}^n$ with parameter $\lambda /\sqrt n$}.

\item \textbf{LIL}. Let $w_i=1/n$ and $t = \frac{{R\sqrt {\log \log n} }}{{\sqrt n }} \le \left\|\bm w \right\|_2^2\lambda /w = \lambda $ for some positive constant $R$. Corollary \ref{sub-exponentialConcentration}(b) claims
\begin{align*}
P(|{{\bar X}_n}|  \ge \frac{{R\sqrt {\log \log n} }}{{\sqrt n }})& \le 2{e^{ - {t^2}/2\left\|\bm w \right\|_2^2{\lambda ^2}}}  =  2\exp \{  - \frac{n}{{2{\lambda ^2}}} \cdot \frac{{{R^2}\log \log n}}{n}\} \\
& = 2\exp \{ \log {(\log n)^{ - {R^2}/2{\lambda ^2}}}\}  = {2}/{{{{(\log n)}^{{R^2}/2{\lambda ^2}}}}}.
\end{align*}
Therefore, with probability $1 - {2}/{{{{(\log n)}^{{R^2}/2{\lambda ^2}}}}}$,
$|{{\bar X}_n}|  \le \frac{{R\sqrt {\log \log n} }}{{\sqrt n }}.$ Although some researchers claims that LIL is useless, we clarify that there are still some meaningful applications of LIL, see \cite{Jamieson14} and \cite{Yang19} for the statistical and machine learning applications of the LIL.
\end{enumerate}

\subsection{Sub-exponential norm}
Recall the Corollary \ref{prop: sub-exponential properties}(\ref{p: exponential MGF finite}): The absolute value of sub-exponential r.v. $|X|$ has a bound MGF at point $K_5^{-1}$: $\phi_{|X|}(K_5^{-1}):=\E e^{|X|/K_5} \le 2$. Similar to the definition of sub-Gaussian norm, we define the sub-exponential norm.
\begin{definition}[sub-exponential norm]\label{def: sub-exponential}
The sub-exponential norm of $X$
  is defined as
  \begin{equation}\label{eq: psione}
  \|X\|_\psione = \inf \left\{ t>0 :\; \E \exp(|X|/t) \le 2 \right\}.
  \end{equation}
\end{definition}
An alternative definition of the sub-exponential norm is
$\|X\|_\psione := \sup_{p \ge 1} p^{-1} ({\rm{E}} |X|^p)^{1/p}$ as in \cite{Vershynin2010}. The sub-exponential r.v. $X$ satisfies the equivalent properties in Corollary~\ref{prop: sub-exponential properties} (Characterizations of sub-exponential). Next, we present a useful lemma below which is to determine the  sub-exponential parameter in the Definition \ref{def:Sub-exponential} by its MGF if we adopt Definition \ref{def: sub-exponential} of the sub-exponential norm.
\begin{proposition}[Properties of sub-exponential norm]\label{prop:Psub-E}
If $\E \exp(|X|/\|X\|_{{\psi _1}} ) \le 2$, then
\begin{enumerate}[\rm{(}a\rm{)}]
\item Tail bounds $P( |X| > t  ) \le 2e^{-t/\|X\|_{{\psi _1}}}~\text{for all } t \ge 0$;
\item  Moment bounds
$\E |X|^k \le 2{\|X\|_{{\psi _1}}^k}k! \quad \text{for all integer}~k \ge 1;$
\item If $\E X=0$, the MGF bounds
${\mathrm{{E}}}e^{s X} \leq e^{({\rm{2}} \left\| X \right\|_{{\psi _1}})^{2}s^{2}/2}~\text { for all }|{s}|<1/(2\left\| X \right\|_{{\psi _1}})$, i.e.$X  \sim \operatorname{subE}(2\left\| X \right\|_{{\psi _1}})$.
\end{enumerate}
\end{proposition}
\begin{proof}
\rm{(a)}. To verified (a), using exponential Markov's inequality, we have
$
P(|X| \geq t)={P}(e^{|X / \|X\|_{{\psi _1}}|} \geq e^{t / \|X\|_{{\psi _1}}}) \leq e^{-t / \|X\|_{{\psi _1}}|} \mathrm{E} e^{|X / \|X\|_{{\psi _1}}|} \leq 2 e^{-t / \|X\|_{{\psi _1}}}
$
by Definition \ref{def: sub-exponential}.

\noindent\rm{(b)}. Similar to the proof of Theorem~\ref{lem-moment} (b), we get from (a)
\begin{align*}
\E|X{|^k} &= \int_0^\infty  P (|X| \ge t)k{t^{k - 1}}dt \le 2k\int_0^\infty  {{e^{ - t/\|X\|_{{\psi _1}}}}} {t^{k - 1}}dt\nonumber\\
[\text{let}~s = {t/\|X\|_{{\psi _1}}}]~~&={{2k}}\int_0^\infty  {{e^{ - s}}} {({s}\|X\|_{{\psi _1}})^{k - 1}}\|X\|_{{\psi _1}}ds= 2{\|X\|_{{\psi _1}}^k}{k}\Gamma( {{k-1}})= 2{\|X\|_{{\psi _1}}^k}k!.
\end{align*}
\rm{(c)}. Applying Taylor's expansion to MGF, we have
\begin{align*}
\mathrm{E} \exp (s X)&=1+\sum_{k=2}^{\infty} \frac{s^{k} \mathrm{E} X^{k}}{k !} \stackrel{\rm (b)}{\le} 1 + {\rm{2}}\sum\limits_{k = 2}^\infty  {{{(s\left\| X \right\|_{{\psi _1}})}^k}}=1+\frac{2(s\left\| X \right\|_{{\psi _1}})^{2}}{1-{s\left\| X \right\|_{{\psi _1}}}},~(|{s\left\| X \right\|_{{\psi _1}}}|<1)\\
[\text{if}~|{s}|<1/(2\left\| X \right\|_{{\psi _1}})]~&\le 1+{4(s\left\| X \right\|_{{\psi _1}})^{2}} \leq e^{({\rm{2}} \left\| X \right\|_{{\psi _1}})^{2}s^{2}/2}.
\end{align*}
Therefore, $X  \sim \operatorname{subE}(2\left\| X \right\|_{{\psi _1}})$.
\end{proof}
Proposition \ref{prop:Psub-E}(c) implies the following user-friendly concentration inequality which contains all known constant. One should note that Theorem 2.8.1 of \cite{Vershynin18} includes an un-specific constant, which makes it is inefficacious when constructing non-asymptotic confident intervals for sub-exponential sample mean.
\begin{proposition}[Concentration for r.v. with sub-exponential sum]\label{Sub-expConcentration}
Let $\{ X_{i}\} _{i = 1}^n $ be zero mean independent sub-exponential distributed with $\|X_i\|_{\psi_1}\le \infty$. Then for every $t \ge 0$,
\begin{center}
$P( {| {\sum\limits_{i = 1}^n {{X_i}} }| \ge t} ) \le 2 \exp\{  { - \frac{1}{4}( {\frac{{{t^2}}}{\sum\nolimits_{i = 1}^n {2\|X_i\|_{\psi_1}^2}} } \wedge \frac{t}{\mathop {\max }\limits_{1 \le i \le n}\|X_i\|_{\psi_1}}}) \}.$
\end{center}
\end{proposition}
\begin{proof}
If $\mathrm{E} \exp(|X|/\|X\|_\psione ) \le 2$, then $X  \sim \operatorname{subE}(2\left\| X \right\|_{{\psi _1}})$ by using Proposition \ref{prop:Psub-E}(c). The result follows by employing Corollary \ref{sub-exponentialConcentration}(b).
\end{proof}

\cite{Gotze2019} mentions an explicitly calculation the sub-exponential norm with example of Poisson distributions. Therefore, it is convenient to apply Proposition~\ref{Sub-expConcentration} to get the concentration of sub-exponential summation.
\begin{lemma}
If $\|X\|_{\psi_{1}}$ exists, then $\|X\|_{\psi_{1}}=1 / \phi_{|X|}^{-1}(2)$ for the MGF $\phi_{X}(t):=\E e^{tX}$.
\end{lemma}
\begin{proof}
Note that $\|\cdot\|_{\psi_{1}}$ is the smallest $t$ such that $\mathrm{E}e^{|X| / t}=\phi_{|X|}({t^{-1}})\le 2,$ so ${t^{-1}} \le \phi_{|X|}^{-1}(2)$ and $t \ge 1 / \phi_{|X|}^{-1}(2)$. By the definition of $\|\cdot\|_{\psi_{1}}$ again, we have $\|X\|_{\psi_{1}}=1 / \phi_{|X|}^{-1}(2)$.
\end{proof}
\begin{example}[The sub-exponential norm of bounded r.v.]\label{eg:boundedr.v.}
Consider a r.v. $|X| \le M<\infty$. Set $\E e^{|X|/t}\le e^{M/t}\le 2 $ and $t\ge M/\log 2$, By the definition of $\|X\|_{\psi_{1}}$, we have $\|X\|_{\psi_{1}}= M/\log 2$.
\end{example}

\begin{example}[The sub-exponential norm of Poisson r.v.]
Poisson r.v. $X$ has the probability mass function
$P({X=k}) = \frac{{\lambda ^{k}}}{{k!}}{e^{ - k}},(k = 1,2, \cdots ,n;\lambda>0).$ We denote it as $X\sim {\rm{Poisson}}(\lambda)$. The MGF of the $ {\rm{Poisson}}(\lambda)$ is ${\phi _X}(t): = {e^{\lambda ({e^t} - 1)}}$. We have $\left\|X\right\|_{\psi_{1}}=[\log (\log (2) \lambda^{-1}+1)]^{-1}$, and the triangle inequality shows $\left\|X-\E X \right\|_{\psi_{1}}\le \left\|X\right\|_{\psi_{1}}+\left\|\E X \right\|_{\psi_{1}}=   \left\|X\right\|_{\psi_{1}}+\frac{\lambda}{{ {\log 2} }} \le[\log (\log (2) \lambda^{-1}+1)]^{-1}+\frac{\lambda}{{ {\log 2} }} \propto \lambda$, where we use inequality $\left\|\E X \right\|_{\psi_{1}} = \frac{|\E X |}{{ {\log 2} }}$ by Example \ref{eg:boundedr.v.}.
\end{example}

Corollary \ref{Sub-expConcentration} is useful in the next subsection for the concentration for quadratic forms.

\subsection{Concentration for quadratic forms and norm of random vectors}\label{quadratic}

All concentration results in the above sections are about the mean. The inference for the variance and covariance in high-dimensional models is an important problem, see Section 6 of \cite{Wainwright19}. It is connected with squares of r.vs. The sample variance is a quadratic form (with shift term) of the data. The data are often postulated as sub-Gaussian. For the square of a sub-Gaussian r.v., it is natural to ask what is the behavior of the tail (or the exponential moment). The answer is sub-exponential by using \eqref{p: sub-gaussian MGF square} in Corollary \ref{Sub-gaussian distribution}.

A simple example that the quadratic form of Gaussian is $\chi^2$ distributed, and the $\chi^2$-distribution of 2 degrees of freedom is exponentially distributed with mean 2. Let us look the $\chi^2$-concentration below:
\begin{example}[Chi-squared r.vs]
If $\{ X_{i}\} _{i = 1}^n\stackrel{\text {IID}}{\sim}N(0,1)$, then we say
$Y_n:=\sum_{i=1}^n X_i^2$
follows $\chi^2$-distribution with $n$-degree of freedom, denoted as $Y_n \sim \chi^2(n)$. The density function is $
    f(y)={\Gamma^{-1}(\frac{n}{2})}{(\frac{1}{2})^{\frac{n}{2}}}y^{\frac{n}{2}-1}e^{-\frac{y}{2}}\cdot 1(y>0)$. As $s<1/2$, the MGF of $X_i^2-1$ is
\begin{center}
$
\mathrm{E}e^{s(X_i^2-1)}=\frac{1}{\sqrt{2 \pi}} \int_{-\infty}^{+\infty} e^{s\left(x^{2}-1\right)} e^{-x^{2} / 2} d x =\frac{e^{-s}}{\sqrt{1-2 s}}\le e^{2s^{2}}= e^{(2s)^{2} / 2}, \quad \text { for all }|s|<\frac{1}{4}
$
\end{center}
where the second last inequality is due to $\frac{e^{-t}}{\sqrt{1-2 t}} \leq e^{2 t^{2}}$ for $|t|<1/4$. Then $X_i^2 \sim \operatorname{subE}(2,4)$. Applying Corollary \ref{sub-exponentialConcentration}(c), we have $Y_n \sim \operatorname{subE}(2\sqrt n ,4)$, therefore $P(|\frac{{{Y_n} - n}}{n}| \ge t) \le 2{e^{ - \frac{n}{8}({t^2} \wedge t)}}.$
\end{example}
Similar sub-exponential results also hold for independent sum of square of sub-Gaussian r.vs. The following two lemmas in Page31 of \cite{Vershynin18} confirm this simple example to the general situation.

\begin{lemma}[Square and product of sub-Gaussian are sub-exponential]\label{lem: sub-exponential squared}
  \rm{(a)}. A r.v. $X$ is sub-Gaussian if and only if $X^2$ is sub-exponential. Moreover,
  $\|X^2\|_\psione = \|X\|_\psitwo^2$; \rm{(b)}. Let $X$ and $Y$ be sub-Gaussian r.vs. Then $XY$ is sub-exponential and
  $\|XY\|_\psione \le \|X\|_\psitwo \, \|Y\|_\psitwo. $
\end{lemma}

For Lemma \ref{lem: sub-exponential squared}(a), it follows from $\|X^2\|_\psione = \|X\|_\psitwo^2$ and Lemma \ref{prop:Psub-E} that Corollary \ref{Sub-expConcentration} coincides Proposition \ref{Sub-gaussianConcentration} as ${ {\max }_{1 \le i \le n}\|X_i\|_{\psi_1}}\to 0$, i.e. the sub-exponential r.v. degenerates to the sub-Gaussian r.v.  the next proposition gives the accurately sub-exponential parameter for the square of sub-Gaussian r.v. in Definition \ref{def:Sub-exponential}, and it improves the constant in Lemma 1.12 of \cite{Rigollet19} (from $\operatorname{subE}\left(16 \sigma^{2}\right)$ to $\operatorname{subE}({{\rm{8}}\sqrt {\rm{2}} } \sigma^{2})$).

\begin{proposition}\label{lem:quadraticsub-Gaussian}
Let $X \sim \operatorname{subG}(\sigma^{2})$, then $Z:=X^{2}-\mathrm{E}X^{2} \sim \operatorname{subE}({{\rm{8}}\sqrt {\rm{2}} } \sigma^{2})$ or $\sim \operatorname{subE}({{\rm{8}}\sqrt {\rm{2}} } \sigma^{2}, 8\sigma^{2})$.
\end{proposition}
\begin{proof}
The proof is immediately from Proposition \ref{pro-moment}(c) by letting $\bm w := (1, 0,\cdots ,0)^T $.
\end{proof}

In below, we deal with a sharper Hanson-Wright inequality  in \cite{Bellec2019}. The Hanson-Wright (HW) inequality is a general concentration result for quadratic forms of sub-Gaussian r.vs, which was first studied in \cite{Hanson71}. Let $\mathbf{A}=(a_{ij})\in\R^{n\times n}$ be a real matrix and the $\vxi=(\xi_1,...,\xi_n)^T$ be a centered random vector with independent components. Define the \emph{Frobenius norm} (Hilbert-Schmidt norm) $\|\mathbf{A}\|_{\mathrm{F}}:=\sqrt{\operatorname{tr}\left(\mathbf{A}^{T} \mathbf{A} \right)}=\sqrt{\sum_{i, j} A_{i, j}^{2}}$ and the \emph{spectral norm} (operator norm) $\|\mathbf{A} \|_{2}:=\sup _{\|\bm u\|_{2} \leq 1}\|\mathbf{A}\bm u\|_{2}$. As an extension of $\chi^2$ r.vs, it is of interest to study the concentration behavior of $\vxi^T \mathbf{A} \vxi - \E[ \vxi^T \mathbf{A} \vxi]$. Under the setting above, Example 2.12 in \cite{Boucheron13} gives the Gaussian chaos concentration.

\begin{corollary}[Gaussian chaos of order 2]
    \label{prop:gaussian-chaos}
Let $\xi_1,...,\xi_n$ be zero-mean Gaussian with $\E \xi_i^2=\sigma_i^2$, Define $\mathbf{D}_\sigma = \diag(\sigma_1,...,\sigma_n)$, then for any $x>0$
    \begin{equation}
       P(\vxi^T \mathbf{A} \vxi - \E[\vxi^T \mathbf{A} \vxi ] \ge 2 \hsnorm{\mathbf{D}_\sigma \mathbf{A} \mathbf{D}_\sigma} \sqrt{x} + 2 \opnorm{\mathbf{D}_\sigma \mathbf{A} \mathbf{D}_\sigma}x )
        \le
        e^{-x}.
        \label{eq:gaussian-chaos}
    \end{equation}
\end{corollary}

The similar concentration phenomenon is also available for sub-Gaussian r.vs. which is named as the HW inequality. \cite{Rudelson13} gives a modern proof by the so-called \emph{decoupling argument} attributed to \cite{Bourgain1996}.
\begin{corollary}[R-V's HW inequality]
    \label{prop:hanson}

    Let $n\ge 1$ and $\vxi:=(\xi_1,...,\xi_n)^T$ be an independent zero-mean sub-Gaussian r.vs with
    $\max_{i=1,...,n} \psinorm{\xi_i} \le K$ for $K>0$.
    Let $A$ be any $n\times n$ real matrix. Then there exists a constant $c>0$ such that
    \begin{equation}
        P( \vxi^T \mathbf{A} \vxi - \E [ \vxi^T \mathbf{A} \vxi ]  > t) \le e^{ - c ( \frac{t^2}{K^4 \hsnorm{\mathbf{A}}^2}\wedge \frac{t}{K^2 \opnorm{\mathbf{A}}} )},~t \ge 0
        \label{eq:hanson-t}
    \end{equation}
  Furthermore, for any $x > 0$, $P(\vxi^T \mathbf{A} \vxi - \E[\vxi^T \mathbf{A} \vxi ] \le c K^2( \opnorm{\mathbf{A}}x +  \hsnorm{\mathbf{A}} \sqrt{x}))\ge 1-e^{-x}$.
\end{corollary}
Intuitively, the term $K^2 \hsnorm{\mathbf{A}}$ is seen as the ``variance term''. When $\mathbf{A}$ is diagonal-free (i.e. the $\mathbf{A}$ matrix has zeros down its diagonal: $a_{ii}=0$ ), the r.v. $\vxi^T \mathbf{A} \vxi$ is zero-mean. \cite{Pollard2015} shortens the proof without unknown constant.

\begin{corollary}[Diagonal-free Hanson-Wright inequality]
Let $\xi_{1}, \ldots, \xi_{n}$ be independent, centered sub-Gaussian r.vs with $\max_{i=1,...,n} \psinorm{\xi_i} \le K<\infty$. Let $\mathbf{A}$ be an $n \times n$ matrix of real numbers
with $a_{i i}=0$ for each $i$. Then ${P}(\vxi^T \mathbf{A} \vxi \geq t)\leq e^{- (\frac{t^{2}}{64 K^{4}\|\mathbf{A}\|_{\mathrm{F}}}\wedge\frac{t}{8 \sqrt{2} K^{2}\|\mathbf{A}\|_{2}})}~ \text { for } t \geq 0.$
\end{corollary}

Under assumptions on the moments of $\xi_1,...,\xi_n$ (do not need sub-Gaussian assumption), the next corollary provides a concentration inequality for quadratic forms of independent r.vs satisfying Bernstein's moment condition (discussed in the next subsection).
\begin{corollary}[Quadratic forms concentration with moment conditions]\label{thm:hanson-moment}
    Assume that the r.v. $\vxi=(\xi_1,...,\xi_n)^T$
    satisfies the condition on independent variables $\xi_1^2,...,\xi_n^2$:
$\E |\xi_i|^{2p} \le \tfrac{1}{2} \; p! \; \sigma_i^2 \; \kappa^{2p-2}~\forall p \ge 1$ for some $\kappa>0$.
 Let $A$ be any $n\times n$ real matrix.
    Then for all $t\ge 0$,
    \begin{equation}\label{eq:hanson-moment-t}
       P(\vxi^T \mathbf{A} \vxi - \E[\vxi^T \mathbf{A} \vxi]  > t)
        \le e^{-(\frac{t^2}{192 \kappa^2 \hsnorm{\mathbf{A} \mathbf{D}_\sigma}^2}\wedge
                \frac{t}{  256 \kappa^2 \opnorm{\mathbf{A}}})},
    \end{equation}
    where $\mathbf{D}_\sigma := \diag(\sigma_1,...,\sigma_n)$.
    Furthermore, with probability greater than $1-e^{-x}$,
    \begin{equation}
        \vxi^T \mathbf{A} \vxi - \E[\vxi^T \mathbf{A} \vxi ]
        \le
        256  \kappa^2 \opnorm{\mathbf{A}} x
        + {8\sqrt{3}} \kappa \hsnorm{\mathbf{A} \mathbf{D}_\sigma} \sqrt{x},~\forall x \ge 0
        .
        \label{eq:hanson-moment-x}
    \end{equation}
\end{corollary}

The bound in (\ref{eq:hanson-moment-t}) is exactly $
    \exp(
        - \frac{t^2}{192 \kappa^2 \hsnorm{ \mathbf{A} \mathbf{D}_\sigma}^2}
   )$ if $t$ is small, and while the $\exp(- c \frac{t^2}{K^4 \hsnorm{\mathbf{A}}^2})$ in right hand side of the R-V's HW inequality \eqref{eq:hanson-t} has an unspecific constant $c>0$.

{\color{black}{We finish this subsection with an exponential inequality for
quadratic forms of a sub-Gaussian random vector. Consider the $n$-dimensional unit sphere ${S^{n - 1}} := \{ \boldsymbol{x} \in {\mathbb{R}^n}:{\left\| \boldsymbol{x} \right\|_2} = 1\}.$ Early in \cite{Fukuda1990}, a random vector $\boldsymbol{X}$ in $\mathbb{R}^n$ is called sub-Gaussian (sub-exponential) if the one-dimensional marginals $\langle {\boldsymbol{X},\boldsymbol{x}} \rangle $ are sub-Gaussian (sub-exponential) r.vs for all $\boldsymbol{x} \in \mathbb{R}^n$. Naturally, the sub-Gaussian (sub-exponential) norm of $\boldsymbol{X}$ is defined as
$\|\boldsymbol{X}\|_{\psi _2} := \sup_{\boldsymbol{x} \in S^{n-1}} \|\langle {\boldsymbol{X},\boldsymbol{x}} \rangle \|_{\psi _2}$ ( $\|\boldsymbol{X}\|_{\psi _1} := \sup_{\boldsymbol{x} \in S^{n-1}} \|\langle {\boldsymbol{X},\boldsymbol{x}} \rangle \|_{\psi _1}$). For the sub-Gaussian, \cite{Fukuda1990}'s definition is equivalent to Chapter 6.3 of \cite{Wainwright19}, a random vector $\boldsymbol{X} \in \R^{d}$ with parameter $\sigma \in \R$ is sub-Gaussian (denote $\mathrm{subGV}(\sigma^2)$) so that:
\begin{equation}\label{eq:SGV}
\mathrm{E} e^{\lambda(\bm{v}, \boldsymbol{X}-\mathrm{E} \boldsymbol{X})} \leq e^{\lambda\sigma^{2}/{2}}, ~ \forall ~ \lambda \in \mathbb{R}^{n}~\text{and}~\bm{v} \in {S^{n - 1}}~\Leftrightarrow~\mathrm{E}e^{\boldsymbol{\alpha}^{T}(\boldsymbol{X}-\mathrm{E} \boldsymbol{X})} \le e^{\|\boldsymbol{\alpha}\|^{2} \sigma^{2} / 2},~\forall~\boldsymbol{\alpha} \in \mathbb{R}^{n}
\end{equation}
In \cite{Pan2020}, the subG random vector with parameter $v_{0} \ge 1$ is defined by ${P}\left(|\langle\boldsymbol{u}, \boldsymbol{X}\rangle| \ge v_{0}\|\boldsymbol{u}\|_{\boldsymbol{\Sigma}} \cdot t\right) \le 2 e^{-t^{2} / 2}$ for all $\boldsymbol{u} \in \mathbb{R}^{n}$ and $t \ge 0,$ where $\boldsymbol{\Sigma}=\mathrm{E}\left(\boldsymbol{X} \boldsymbol{X}^T\right)$ and $\|\boldsymbol{u}\|_{\mathbf{A}}=\left\|\mathbf{A}^{1 / 2} \boldsymbol{u}\right\|_{2}$ is the norm indexed by ${\mathbf{A}}$. For Definition \eqref{eq:SGV}, \cite{Hsu2012} obtains a tail bound for subG random vectors:
\begin{corollary}[Tail inequality for quadratic forms of sub-Gaussian vectors]\label{thm:zhangt}
Let $\bm{\Sigma}=\mathbf{A}^{T} \mathbf{A}$ for $p \times n$ matrix $\mathbf{A}$. Consider a sub-Gaussian random vector $\vxi=(\xi_1,...,\xi_n)^T\sim \mathrm{subGV}(\sigma^2)$ with independent components for $\boldsymbol{\mu}=\mathrm{E}\vxi$. Then, for any $t\ge 0$
\begin{center}
${P}\{\|\mathbf{A} \vxi\|^{2}>\sigma^{2}[\operatorname{tr}(\bm{\Sigma})+2 \operatorname{tr}(\bm{\Sigma}^{2} t)^{1 / 2}+2\|\bm{\Sigma}\|_2 t]+\operatorname{tr}(\bm{\Sigma} \boldsymbol{\mu} \boldsymbol{\mu}^{T})(1+2 \sqrt{\textstyle{{\|\bm{\Sigma}\|_2^{2}}\over{\operatorname{tr}(\bm{\Sigma}^{2})}} t})\} \leq e^{-t}.$
\end{center}
\end{corollary}
Conditioning on a divergence number of non-random covariates, an application of Corollary \ref{thm:zhangt} for the prediction error \eqref{olsR} in regressions with sub-Gaussian noise is given in Section \ref{se;lm}. The concentration bounds of sub-Gaussian random vectors depend on the parameter $\sigma$: the smaller $\sigma$, the tighter concentration bounds. Equation \eqref{eq:SGV} requires the distribution of subG random vectors to be isotropic, and the random vectors have an exponential tail, but the sub-Gaussian parameter $\sigma$ may be large, which leads to loose bounds for constructing confidence bands. To establish tighter bounds, \cite{Jin19} define a different and general class of sub-Gaussian distributions in $\mathbb{R}^{n}$, called norm-subGaussian random vectors as follows.
\begin{definition}[Norm-subGaussian]\label{def:norm-subGaussian}
A random vector $\boldsymbol{X} \in \mathbb{R}^{n}$ is norm-subG (denoted $\mathrm{nsubG}(\sigma^2)$), if $\exists \sigma$ so that:
$
{P}(\|\boldsymbol{X}-\mathrm{E} \boldsymbol{X}\| \geq t) \leq 2 e^{-\frac{t^{2}}{2 \sigma^{2}}}, \quad \forall t \in \mathbb{R}^{+}.
$
\end{definition}
The Definition \ref{def:norm-subGaussian} only requires the tail probability estimate has sub-Gaussian tail under $l_2$-norm, avoiding the uniform condition in \eqref{eq:SGV}. If $\mathrm{E} \boldsymbol{X}=\bm 0$ and  $2 e^{-t^{2} / 2\sigma^2} \ge {P}\left(\|\boldsymbol{X}\| \ge t\right)$ from $\mathrm{nsubG}(\sigma^2)$, we get ${P}\left(|\langle\boldsymbol{u}, \boldsymbol{X}\rangle| \ge t\right) \le {P}\left(\|\boldsymbol{X}\| \ge t\right) \le 2 e^{-t^{2} / 2\sigma^2},~\bm{u} \in {S^{n - 1}}$ by Cauchy's inequality. Thus, $\mathrm{nsubG}(\sigma^2)$ implies \eqref{eq:SGV}, and this verifies that the norm-subG is more general. \cite{Jin19} show that if $\boldsymbol{X} \in \mathbb{R}^{n}$  is $\mathrm{subGV}(\sigma^2/{n})$, then $\boldsymbol{X} \sim \mathrm{nsubG}(8\sigma^2)$.

}}

\section{Sub-Gamma Distributions and Bernstein's Inequality}
\subsection{Sub-Gamma distributions}
Comparing to the classical Chebyshev's inequality, Bernstein-type inequalities have more precise concentration, it originally is an extension of the Hoeffding's inequality with bounded assumption [see \cite{Bernstein1924}, \cite{Bennett1962}]. As mentioned by \cite{Pollard2015}, the proof of Hoeffding's inequality with endpoints of the interval $[a,b]$ in Lemma~\ref{lm:Hoeffding} (with $n=1$) crudely depends on the variance bound:
\begin{equation}\label{eq:VarX}
{\mathop{\rm Var}\nolimits} X = \E{(X - \E X)^2} \le \E[X - ({{b - a}})/{2}]^2 \le [({{b - a}})/{2}]^2~\text { if } a \leq X \leq b
\end{equation}
without any other variance information.

If $X$ takes values near the endpoints of the interval $[a,b]$ with a small probability, it is guessed that the sharper concentration could be improved by adding variance condition. The following tail bound for the sum $S_n:=\sum_{i = 1}^n {{X_i}}$ needs extra variance information.

\begin{corollary}[Bernstein's inequality with the bounded condition]\label{lm-Bernsteinbd}
Let $X_{1},\ldots ,X_{n}$ be centralized independent variables such that $|X_{i}|\leq M$ a.s. for all $i$. Then, $\forall~t>0$
\begin{center}
${P} (|S_n|\geq t)\leq 2 e^{-{\frac {t^{2}/2}{\sum _{i=1}^{n}\mathrm{Var}X_{i}+Mt/3}}},~P\{|S_n|\geq(2t \sum_{i=1}^{n}\mathrm{Var}X_{i})^{1/2}+\frac{{Mt}}{3}\} \leq 2e^{-t}.$
\end{center}
\end{corollary}
The next example illustrates a sharp confidence interval for sample mean if we known that the variance is sufficient small.
\begin{example}[Non-asymptotic confidence intervals]
Let $\{X_i\}_{i=1}^n \stackrel{\rm{IID}}{\sim} X$ with the support $[-c,c]$ and the mean $\mu$. Hoeffding's and Bernstein's inequalities show for ${\bar X}:={n}^{-1} \sum{}_{i=1}^{n} X_{i}$
$$
\begin{array}{c}
{P}(|\bar X-\mu | \le \sqrt{\frac{2 c^{2} \log (2 / \delta)}{n}}) \geq 1-\delta \quad \text { Hoeffding};\\
{P}(|\bar X-\mu | \le \frac{c}{3 n} \log (2 / \delta)+\sqrt{\frac{2(\operatorname{Var} X) \log (2 / \delta)}{n}}) \geq 1-\delta \quad \text { Bernstein}.
\end{array}
$$
 For large $n$, the Bernstein's confidence interval is substantially shorter if ${X_i}$ has relatively small variance, i.e. ${\rm{Var}}{X} \ll {c^2}$ (the factor $\sqrt{\frac{\log (2 / \delta)}{n}}$ is a dominated term). The Hoeffding's confidence is shorter as ${\rm{Var}}{X}= {c^2}$ (This extreme case attains the upper bound ${\rm{Var}}{X} \le {c^2}$ in \eqref{eq:VarX} due to $b-a=2c$). But, for the case ${\rm{Var}}{X}< {c^2}$,  if  $n$ is sufficient small s.t. $\frac{c}{{3n}}\log (2/\delta ) \ge (c - \sqrt {\rm{Var}X} )\sqrt {\frac{{2\log (2/\delta )}}{n}} $, i.e. we need restrictions $n \le \frac{1}{{18}}{\left( {\frac{c}{{c - \sqrt {{\rm{Var}}X} }}} \right)^2}\log (\frac{2}{\delta })\ge 1$ to ensure Hoeffding's confidence interval is more accurate when $\delta  \le 2\exp \{  - \frac{1}{{18}}{(\frac{{c - \sqrt {{\rm{Var}}X} }}{c})^2}\}$.
\end{example}
%\fn{this and the next sentences should be in the introduction on the roles of the CI}

To prove Corollary \ref{lm-Bernsteinbd}, we need get the sharp bounds of the MGF of the single variable and then do aggregation for the summation. By the Taylor expansion, we have
\begin{center}
${\rm{E}}{e^{s{X_i}}} = 1 + \sum\limits_{k = 2}^\infty  {{s^k}} \frac{{{\rm{E}}{X_i^k}}}{{k!}} \le 1 + \sum\limits_{k = 2}^\infty  {{s^k}} \frac{{{M^{k - 2}}{\rm{Var}}{X_i}}}{{k!}} \le 1 + {s^2}{\rm{Var}}{X_i}\sum\limits_{k = 2}^\infty  {\frac{{{{(|s|M)}^{k - 2}}}}{{k!}}},~1 \leq i \leq n.$
\end{center}
Applying the inequality $k ! / 2 \geq 3^{k-2}$ for any $k \geq 2$, it implies
\begin{equation}\label{eq:GammaGMFbd}
{\rm{E}}{e^{s{X_i}}} \le 1 + \frac{{{s^2}{\rm{Var}}{X_i}}}{2}\sum\limits_{k = 2}^\infty  {{{\left( {\frac{{|s|M}}{3}} \right)}^{k - 2}}}  = 1 + \frac{{{s^2}{\rm{Var}}{X_i}/2}}{{1 - |s|M/3}} \le \exp \left( {\frac{{{s^2}{\rm{Var}}{X_i}/2}}{{1 - |s|M/3}}} \right).
\end{equation}
The upper bounds of MGF essentially have the same form in comparison with Gamma distribution below whose MGF is bounded by \eqref{eq:GammaGMF} in following example.
\begin{example}[Gamma r.vs]\label{ex:Gamma}
The Gamma distribution with density $f(x)=\frac{x^{a-1} e^{-x / b}}{\Gamma(a) b^{a}},~x \geq 0$ is denote as $\Gamma(a,b)$. We have $\mathrm{{E}} X=a b$ and $\operatorname{Var}X=a b^{2}$ for $X\sim\Gamma(a,b)$. The Page28 of \cite{Boucheron13} illustrates that the log-MGF of a centered $\Gamma(a,b)$ is bounded by
\begin{equation}\label{eq:GammaGMF}
\log (\mathrm{{E}}{e^{s({X}-\mathrm{{E}} X)}})=a(-\log (1-sb)-sb)  \leq s^{2}a b^{2}/[2(1- bs)], \quad \forall~0<s<b^{-1}.
\end{equation}
\end{example}
Motivated by the MGF bounds in \eqref{eq:GammaGMF}, \cite{Boucheron13} defines the sub-Gamma r.v. based on the right tail and left tail with variance factor $v$ and scale factor $b$.
\begin{definition}[Sub-Gamma r.v.]\label{def: sub-Gamma}
A centralized r.v. $X$ is {\em sub-Gamma} with the \emph{variance factor} $\upsilon>0$ and the \emph{scale parameter} $c>0$ (denoted by $X \sim \mathrm{sub}\Gamma(\upsilon,c)$) if
\begin{equation}\label{eq:sub-gamma}
\log (\mathrm{{E}}{e^{s{X}}}) \leq {s^{2}}{\upsilon}/[2({1-c|s|})], \quad \forall~0<|s|<c^{-1}.
\end{equation}
\end{definition}
If the restriction $0<|s|<b^{-1}$ is replaced by one side conditions $0<s<b^{-1}$ (or $0<-s<b^{-1}$), we call it \emph{sub-Gamma on the right tail} (or \emph{sub-Gamma on the left tail}), denoted as $\mathrm{sub}\Gamma_{+}(\upsilon,c)$ (or $\mathrm{sub}\Gamma_{-}(\upsilon,c)$). In Example \ref{ex:Gamma},  the Gamma r.v. $X \sim \mathrm{sub}\Gamma_{+}(a b^{2},b)$. The \eqref{eq:sub-gamma} is called \emph{two-sided Bernstein's condition}.

\begin{example}[Sub-exponential r.vs]
The sub-exponential distribution implies the sub-Gamma condition:
$\log (\mathrm{{E}}{e^{s{X}}}) \le {\frac{s^{2}\lambda ^{2}}{2}} \le \frac{s^{2}\lambda ^{2}}{
{2(1 - \lambda|s|)}},~\forall~|s| <\frac{1}{\lambda}.$ This shows that $X \sim \operatorname{subE}(\lambda)$ implies $X \sim \mathrm{sub}\Gamma(\lambda^{2},{\lambda})$.
\end{example}
The sub-Gamma condition \eqref{eq:sub-gamma} leads to the useful tail bounds and moment bounds.
\begin{lemma}[Sub-gamma properties, \cite{Boucheron13}]\label{sub-gammaproperties}
%\begin{enumerate}[\rm{(}a\rm{)}]
%\item \label{p: sub-gamma}
If $X \sim \mathrm{sub}\Gamma(\upsilon,c)$, then
\begin{equation}\label{eq:sub-gammasingle}
P(|X|>t)\leq 2 e^{-\frac{\upsilon}{c^{2}} h\left(\frac{c t}{\upsilon}\right)}\leq 2e^{-\frac{t^2/ 2}{v+ct}},
\end{equation}
where $h(u)=1+u-\sqrt{1+2|u|}$. Moreover, we have
$
P\{|X|>\sqrt{2 v t}+c t\} \leq e^{-t}.
$
\end{lemma}
The tail bound in Lemma \ref{sub-gammaproperties} verifies that, the sub-Gamma
variable has sub-Gaussian tail behavior with parameter $\upsilon$ for suitably small $t$, and it has exponential tail behavior for larger $t$. The proof is originated from \cite{Bennett1962}.
\begin{proof}
By Chernoff's inequality, $
P\left( {X - {\rm{E}}X \ge t} \right) \le \mathop {\inf }\limits_{s > 0} {e^{ - st}}{\rm{E}}{e^{s(X - {\rm{E}}X)}}$.
It remains to bound $\log({e^{ - st}}{\rm{E}}{e^{s(X - {\rm{E}}X)}})$ by definition of sub-Gamma
variable for all $0<|s|<c^{-1}$
\begin{center}
$\mathop {\inf }\limits_{{c^{ - 1}} \ge s > 0} \log ({e^{ - st}}{\rm{E}}{e^{s(X - {\rm{E}}X)}}) \le \mathop {\inf }\limits_{{c^{ - 1}} \ge u > 0} \left( {\frac{{{u^2}}}{2}\frac{v}{{1 - cu}} - ut} \right) =  - \frac{v}{{{c^2}}}h(\frac{{ct}}{v}) \le  - \frac{{{t^2}/2}}{{v + ct}},$
\end{center}
where the last inequality is from $h(u) = 1 + u - \sqrt {1 + 2|u|}  \ge \frac{{{u^2}/2}}{{1 + u}}$. So we conclude \eqref{eq:sub-gammasingle}.
\end{proof}
\begin{proposition}[Concentration for sub-Gamma sum]\label{sub-GammaConcentration}
Let $\{ X_{i}\} _{i = 1}^n $ be independent\\
 $\{\mathrm{sub}\Gamma(\upsilon_i,c_i)\} _{i = 1}^n$ distributed with zero mean. Define $c= {\max }_{1 \le i \le n} c_i$, we have
\begin{enumerate}[\rm{(}a\rm{)}]
\item Closed under addition: $ S_n:=\sum_{i = 1}^n {{X_i}}  \sim {\rm{sub}}\Gamma({\sum_{i = 1}^n\upsilon_i},c)$;
\item For every $t \ge 0$: ${P} (|S_n|\geq t)\leq 2 \exp \left(-{\frac {t^{2}/2}{\sum _{i=1}^{n}\upsilon_i+ct}}\right)$ and $P\{|S_n|\ge (2t \sum_{i=1}^{n}\upsilon_i)^{1/2}+c t\} \leq 2 e^{-t}$;
\item If $X \sim \mathrm{sub}\Gamma(\upsilon,c)$, the moments bounds satisfy for any integer $k\ge 1$:
\begin{center}
$\mathrm{E}{X^{k}} \le k{2^{k - 2}}[2{(\sqrt {2v} )^k}\Gamma (\frac{k}{2}) + c{(\sqrt {2v} )^{k - 1}}\Gamma (\frac{{k + 1}}{2}) + 3{c^k}\Gamma ({k})].$
\end{center}
\item If $X \sim \mathrm{sub}\Gamma(\upsilon,c)$, the even moments bounds satisfy
$\mathrm{E}{X^{2k}} \le k!{({8v} )^k}+(2k)!{({4c} )^{2k}},~k\ge 1.$

\item If $P\{|X|>(2t \upsilon)^{1/2}+c t\} \leq 2 e^{-t}$, then $X \sim \mathrm{sub}\Gamma(32(\upsilon+2c^2),8c)$.
\end{enumerate}
\end{proposition}
\begin{proof}
{\rm(a)} By definition of $\{\mathrm{sub}\Gamma(\upsilon_i,c_i)\} _{i = 1}^n$, we have $\log (\mathrm{{E}}{e^{s{X_i}}}) \leq \frac{s^{2}}{2} \frac{\upsilon_i}{1-c_i|s|},~\forall~0<|s|<c^{-1}$, from which and the independence among $\{ X_{i}\} _{i = 1}^n $, thus
\begin{center}
$\log (\mathrm{{E}}{e^{sS_n}}) \leq \frac{s^{2}}{2}\sum_{i = 1}^n \frac{\upsilon_i}{1-c_i|s|}\leq \frac{s^{2}}{2} \frac{\sum_{i = 1}^n\upsilon_i}{1-c|s|},\quad {\rm{for~all~}}0<|s|<c^{-1}.$
\end{center}

\noindent{\rm(b)} Employing Lemma \ref{sub-gammaproperties}, we immediately obtain {\rm(b)} due to {\rm(a)}.

\noindent{\rm(c)} Applying the integration form of the expectation formula, it yields
\begin{align*}
{\rm{E}}{X^k}&\le {\rm{E}}{|X|^k} = k\int_0^\infty  {{x^{k - 1}}} P\{ |X| > x\} dx= k\int_0^\infty  {{x^{k - 1}}} P\{ |X| > \sqrt {2vt}  + ct\} (\frac{{\sqrt {2v} }}{{2\sqrt t }} + c)dt\\
& \le 2k\int_0^\infty  {{{(\sqrt {2vt}  + ct)}^{k - 1}}} (\frac{{\sqrt {2vt}  + 2ct}}{{2t}}){e^{ - t}}dt= k\int_0^\infty  {[{{(\sqrt {2vt}  + ct)}^k} + ct{{(\sqrt {2vt}  + ct)}^{k - 1}}} ]\frac{{{e^{ - t}}}}{t}dt.
\end{align*}
From {\rm(b)} and inequality \eqref{eq:Jensen},
\begin{align*}
{\rm{E}}{X^k}& \le k\int_0^\infty  {\left\{ {{2^{k - 1}}[{{(\sqrt {2vt} )}^k} + {{(ct)}^k}] + ct{2^{k - 2}}[{{(\sqrt {2vt} )}^{k - 1}} + {{(ct)}^{k - 1}}]} \right\}} \frac{{{e^{ - t}}}}{t}dt\\
& = k{2^{k - 2}}\int_0^\infty  {[2{{(\sqrt {2v} )}^k}{t^{(k/2) - 1}} + c{{(\sqrt {2v} )}^{k - 1}}{t^{(k + 1)/2-1}} + 3{c^k}{t^{k - 1}}]} {e^{ - t}}dt\\
& = k{2^{k - 2}}[2{(\sqrt {2v} )^k}\Gamma (\frac{k}{2}) + c{(\sqrt {2v} )^{k - 1}}\Gamma (\frac{{k - 1}}{2}) + 3{c^k}(k - 1)!].
\end{align*}
\noindent{\rm(d,e)} The proofs are in Theorem 2.3 of \cite{Boucheron13}.
\end{proof}

Having obtained Proposition \ref{sub-gammaproperties}(b), from the upper bound in \eqref{eq:GammaGMFbd}, we finish the proof of Proposition \ref{lm-Bernsteinbd} by treating $X_i \sim \mathrm{sub}\Gamma({\rm{Var}}{X_i}/2,M/3)$ for $i=1,2,\cdots,n$.

\subsection{Bernstein's growth of moments condition}

In some settings, one can not assume the r.vs being bounded. Bernstein's inequality for the sum of independent r.vs allows us to estimate the tail probability by a weaker version of an exponential condition on the growth of the $k$-moment without the boundedness.
\begin{corollary}[Bernstein's inequality with the growth of moment condition]\label{lm-Bernsteingm}
If the centred independent r.vs $X_1,\ldots, X_n$ satisfy \emph{the growth of moments condition}
\begin{equation}\label{eq-Bernsteinmoments}
{\rm{E}}{\left| {{X_i}} \right|^k} \le {2}^{-1}v_i^2{\kappa_i^{k - 2}}k!,~(i = 1,2, \cdots ,n),~\text{for all}~k\ge 2
\end{equation}
where $\{\kappa_i\}_{i=1}^n, \{v_i\}_{i=1}^n$ are constants independent of $k$. Let ${\nu _n^2} = \sum_{i = 1}^n {v _i^2} $ (the fluctuation of sums) and $\kappa={\max }_{1 \le i \le n} \kappa_i$. Then, we have $X_{i}\sim \mathrm{sub}\Gamma(v_i,\kappa_i)$ and for $t>0$
\begin{equation}\label{eq-Bernstein}
P\left( {\left| {{S_n}} \right| \ge t} \right) \le 2e^{ - \frac{{{t^2}}}{{2{\nu _n^2} + 2\kappa t}}},~~P( {\left| {{S_n}} \right| \ge \sqrt {2\nu _n^2t}  + \kappa t}) \le 2{e^{ - t}}.
\end{equation}
\end{corollary}
\begin{proof}
Given that $\kappa_i|s|<1$ for all $i$, \eqref{eq-Bernsteinmoments} implies that $X_{i}\sim \mathrm{sub}\Gamma(v_i,\kappa_i)$ for $1 \leq i \leq n$
\begin{center}
$
{\rm{E}} e^{s X_{i}} \leq 1+\frac{v_{i}^{2}}{2} \sum_{k=2}^{\infty}|s|^{k} \kappa_i^{k-2}=1 +\frac{s^{2} v_{i}^{2}}{2(1-|s| \kappa_i)} \leq e^{s^{2} v_{i}^{2} /(2-2 \kappa_i|s|)}.
$
\end{center}
The independence among $\{X_i\}_{i=1}^n$ and Proposition \ref{sub-GammaConcentration}(a,b) implies \eqref{eq-Bernstein}.
\end{proof}

The \eqref{eq-Bernsteinmoments} is also called \emph{Bernstein's moment condition}. Corollary~\ref{lm-Bernsteingm} slightly extends Lemma 2.2.11 in \cite{van96} for the case ${\kappa _i} \equiv \kappa$ (a fixed number). It should be noted that \eqref{eq-Bernsteinmoments} can be replaced by $\frac{1}{n} \sum_{i=1}^{n} {\rm{E}}{\left| {{X_i}} \right|^k} \le \frac{1}{2}v^2{\kappa^{k - 2}}k!, k=3,4, \ldots, \forall i$, where the $v^2$ is a variance-depending constant such that  $\frac{1}{n} \sum_{i=1}^{n} {\rm{E}}{\left| {{X_i}} \right|^2} \le v^2$. Then \eqref{eq-Bernstein} still holds with $\nu _n^2=nv$, see Theorem 2.10 in \cite{Boucheron13}.

\begin{example}[Normal r.v.]\label{ex:bounded}
Applying the relation between MGF and moment, the $k$-th moment of $X \sim N(0,\sigma^{2})$ is
${\rm{E}}{ X ^{2k - 1}} = 0;~{\rm{E}}{\left| X \right|^{2k}} = {\sigma ^{2k}}(2k - 1)(2k - 3) \cdots 3 \cdot 1 \le {2}^{-1}(2{\sigma ^2}){\sigma ^{2k - 2}}(2k)!$,
 which satisfies \eqref{eq-Bernsteinmoments} with $v^2 = 2{\sigma ^2},\kappa  = {\sigma ^2}$.
\end{example}

%\begin{corollary}\label{col-Bernstein1}
If \eqref{eq-Bernsteinmoments} is changed to $
{\rm{E}}{\left| {{X_i}} \right|^k} \le \tilde v_i^2{(2\tilde\kappa)^{k - 2}},~{\tilde\nu _n^2} = \sum_{i = 1}^n {\tilde v _i^2} ~(1 \leq i \leq n;~\forall~k\ge 2)$, in this case, then we have a looser moment bound ${\rm{E}}{\left| {{X_i}} \right|^k} \le \tilde\nu_i^2{(2\tilde\kappa)^{k - 2}}\le\frac{{\rm{1}}}{{\rm{2}}} \cdot 2\tilde\nu_i^2{(2\tilde\kappa )^{k - 2}}k!$, which gives $P\left( {\left| {{S_n}} \right| \ge t} \right) \le 2e^{ - \frac{{{t^2}}}{{4{\tilde\nu _n^2} + 4\tilde\kappa t}}}$ by \eqref{eq-Bernstein} in Corollary~\ref{lm-Bernsteingm}  ($v_i^2=2\tilde v^2,~\kappa=2\tilde\kappa$).
%\end{proof}
Without using the accurate sub-Gamma conditions, Lemma 19.32 in \cite{Vaart1998} shows a
looser CI $P\left( {\left| {{S_n}} \right| \ge t} \right) \le 2e^{ - \frac{{{t^2}}}{{4{\tilde\nu _n^2} + 4\tilde\kappa t}}}$ by a different proof.

\subsection{Concentration of exponential family without compact space}

Theory and statistical applications of natural exponential family \eqref{eq:E-P} have attracted renewed attention in the past years \citep{Lehmann06}.  in Lasso penalized \emph{generalized linear models} (GLMs), the results of oracle inequalities lie on CIs of a quantity that can be represent as Karush-Kuhn-Tucker conditions (see \eqref{eq:kkt}) related to the centralized exponential family empirical process:
$\sum_{i = 1}^n {{w_i}({Y_i}}  - {\rm{E}}{Y_i}) $ for no-random weights $\{w_i\}_{i=1}^n$ depending on the fixed design. \cite{Kakade10} has studied the sub-exponential growth of the cumulants of an exponential family distribution and studied oracle inequalities of Lasso regularized GLMs, but the constant in their result is not specific.

In this part, we obtain cental moments bounds with a specific constant, which gives the Bernstein's inequality for the general exponential family, and the proof is based on the Cauchy formula of higher-order derivatives for complex functions [Corollary 4.3 in \cite{Shakarchi10}].
\begin{lemma}[Cauchy's derivative inequalities]\label{lem:Cauchy}
If $f$ is analytic in an open set that contains the closure of a disk $D$ centered at $z_0$ of radius $0<r<\infty$, then
$
|f^{(n)}(z_{0})| \leq \frac{n !}{r^{n}}\sup _{z : |z-z_0|=r}|f(z)|.
$
\end{lemma}
\noindent \cite{Zhang14} adopts a similar approach for recovering the probability mass function (p.m.f.) from the characteristic function.

It is well-known that exponential families on the natural parameter space,
$\Theta :=\{ \theta \in \mathbb { R } ^ { k } :{e^{b({\theta})}} := \int_{} {c(y){e^{y\theta }}\mu (dy)}< \infty \}$ have finite analytic (standardized) moments and cumulants, see Lemma 3.3 in \cite{Kakade10}. The natural parameter space of an exponential family is convex, see \cite{Lehmann06}. A nice property in Lehmann's measure-theoretical statistical inference book is that:

\begin{lemma}[Analytic property of MGF in the exponential family]\label{le:Analytic}
The MGF $m_{\theta_i}(s):={{\rm{E}}_{\theta_i} }{e^{s{X_i} }}$ on $s\in \mathbb{C}$ of exponential family r.vs indexed by $\theta_i$, is analytic on $\Theta$ [see Theorem 2.7.1 in \cite{Lehmann06} or Theorem 2 in
\cite{Pistone1999}].
\end{lemma}

First, let us check the following lemma which is deduced by Cauchy's inequalities for the Taylor's series coefficients of a complex analytic function.
\begin{proposition}
The $s \mapsto \bar{m}_{\theta _i}(s):={{\rm{E}}_{\theta _i}}{e^{s|{Y_i} - \dot b({\theta _i}){\rm{|}}}}$ is analytic on the natural parameter space $\Theta$ with radius ${\rm{r}}(\Theta)$, and the $k$-th absolute central moment of $\{w_i{Y_i}\}_{i=1}^n$ is bounded by
\begin{center}
${\rm{E}}_{\theta _i}{|{{w_i}({Y_i}}  - {\rm{E}}{Y_i}) |^k}\le \frac{k!}{2}{({w}\sqrt {2C_{{\theta _i}}} )^2}({w}C_{{\theta _i}})^{k - 2},~k=2,3,\cdots$
\end{center}
 where $\{w_i\}_{i=1}^n$ are non-random with $w:= \mathop {\max }\limits_{1 \le i \le n} |w_i|>0$, and ${C_{\theta_i}}: =  \mathop {\inf }\limits_{0<r \le {\rm{r}}(\Theta)} {r}^{-1}{{{{{{\rm{E}}_{{\theta _i}}}{e^{r|{X_i} - \dot b({\theta _i})|}}}}}} $.
\end{proposition}
 \begin{proof}
 Let $s \in \mathbb{R}{\mathrm{i}}:=\{b{\mathrm{i}}:b \in \mathbb{R}\}$ be a given complex number on imaginary axis.
\begin{align}\label{eq:barm}
{{\bar m}_{\theta _i}}(s) &= {{\rm{E}}_{{\theta _i}}}( {{e^{s[{Y_i} - \dot b({\theta _i})]}}1\{ {Y_i} \ge \dot b({\theta _i})\} } ) + {{\rm{E}}_{{\theta _i}}}( {{e^{s[\dot b({\theta _i}) - {Y_i}]}}1\{ {Y_i} < \dot b({\theta _i})\} } )\nonumber\\
 &= \smallint{}_{x \ge \dot b({\theta _i})} {c(x){e^{x({\theta _i} + s)}}{e^{ - \dot b({\theta _i})}}\mu (dx)}  + \smallint{}_{x < \dot b({\theta _i})} {c(y){e^{x({\theta _i} - s)}}{e^{ - \dot b({\theta _i})}}\mu (dx)} \nonumber\\
 &= {e^{ - \dot b({\theta _i})}}[ {\smallint{}_{x \ge \dot b({\theta _i})} {c(x){e^{x({\theta _i} + s)}}\mu (dx)}  + \smallint{}_{x < \dot b({\theta _i})} {c(x){e^{x({\theta _i} - s)}}\mu (dx)} } ].
\end{align}
The natural parameter space implies $\int {c(x){e^{x{\theta _i}}}\mu (dx)}$ is finite and analytic for ${\theta _i} \in \Theta$, so
\begin{center}
$\smallint{1\{ x \ge \dot b({\theta _i})\} c(x){e^{x({\theta _i} +s)}}\mu (dx)}$
and ${\int {1\{ x < \dot b({\theta _i})\} c(y){e^{x({\theta _i} - s)}}\mu (dx)} }$
\end{center}
are finite and analytic for $s \in \mathbb{r}{\rm{i}}, {\theta _i}\in\Theta$. By Lemma \ref{le:Analytic}, $\bar{m}_{\theta _i}(s)$ in \eqref{eq:barm} is analytic on
\begin{center}
$D_{{\theta _i}}:=\left\lbrace s \in \mathbb{C}: \textrm{Re}({\theta _i}+s) \in \textrm{Int}(\Theta)~\textrm{and} ~\textrm{Re}({\theta _i}-s) \in \textrm{Int}(\Theta) \right\rbrace .$
\end{center}
by using analytic continuation [i.e. the $\bar{m}_{\theta _i}(s)$ has an analytic continuation from ${{\bar m}_{\theta _i}}(s)$ on $s\in D_{{\theta _i}}$ to ${{\bar m}_{\theta _i}}(s)$ on $s \in \mathbb{C}$, see Corollary 4.9 in \cite{Shakarchi10}].

Since $0+{\theta _i}=\theta_i \in D_{{\theta _i}} \subset \textrm{Int}(\Theta)$, ${{\bar m}_{\theta _i}}(s)$ is analytic at the point 0 and hence the function is also analytic in a neighborhood of 0.
%\vee 1
By the analyticity of the functions $\{\bar{m}_{\theta _i}(s)\}_{{\theta _i}\in\Theta}$ on $s\in \textrm{Int}(\Theta)$, and Cauchy's derivative inequality with $z_0=0$, we have
\begin{align}\label{eq:Cauchy}
{\rm{E}}_{\theta _i}{| {{Y_i} - \dot b({\theta _i})} |^k} &= {{\bar m}_{\theta _i}}^{(k)}(0) \le k!{{{r^{-k}}}}\mathop {\sup }{}_{|s|= r} |{{\rm{E}}_{\theta _i}}{e^{s|{Y_i} - \dot b({\theta _i}){\rm{|}}}}|,~~{0 < r \le {\rm{r}}(\Theta)}.
\end{align}
Let $s = r(\cos \omega  + {\rm{i}}\sin \omega ),\omega  \in [0,2\pi ]$. Then, we get
${{\rm{E}}_{{\theta _i}}}{e^{s|{Y_i} - \dot b({\theta _i}){\rm{|}}}} = {{\rm{E}}_{{\theta _i}}}{e^{r(\cos \omega  + {\rm{i}}\sin \omega )|{Y_i} - \dot b({\theta _i}){\rm{|}}}}$\\$ = {{\rm{E}}_{{\theta _i}}}[ {{e^{r\cos \omega |{Y_i} - \dot b({\theta _i})|}}{e^{{\rm{i}}r\sin \omega |{Y_i} - \dot b({\theta _i})|}}} ].$
Hence, \eqref{eq:Cauchy} gives
\begin{align*}
k!\frac{1}{{{r^k}}}\mathop {\sup }\limits_{|s| = r}| {{{\rm{E}}_{{\theta _i}}}{e^{s|{Y_i} - \dot b({\theta _i})|}}} |&\le k!\frac{1}{{{r^k}}}\mathop {\sup }\limits_{\omega  \in [0,2\pi ]}{{{\rm{E}}_{{\theta _i}}}{e^{r\cos \omega |{Y_i} - \dot b({\theta _i})|}}} \le k!\frac{{{\rm{E}}_{{\theta _i}}}{e^{r|{Y_i} - \dot b({\theta _i})|}}}{{{r^k}}},\\
(\text{Due to}~{{\rm{E}}_{{\theta _i}}}{e^{r|{Y_i} - \dot b({\theta _i})|}} \ge 1)~~& = k!{{\rm{\{ }}\frac{{{{{\rm{[}}{{\rm{E}}_{{\theta _i}}}{e^{r|{Y_i} - \dot b({\theta _i})|}}]}^{1/k}}}}{r}\} ^k} \le k!{{\rm{\{ }} \frac{{{{{\rm{[}}{{\rm{E}}_{{\theta _i}}}{e^{r|{Y_i} - \dot b({\theta _i})|}}]}}}}{r}\} ^k}.
\end{align*}
From \eqref{eq:Cauchy}, it shows that by take infimum over ${0 < r \le{\rm{r}}(\Theta)}$,
\begin{center}
${\rm{E}}_{\theta _i}{| {{Y_i} - \dot b({\theta _i})} |^k} \le k!{{\rm{\{ }}\mathop {\inf }\limits_{0 < r \le {\rm{r}}(\Theta) } {r^{-1}}{{{{{\rm{[}}{{\rm{E}}_{{\theta _i}}}{e^{r|{Y_i} - \dot b({\theta _i})|}}]}}}}\} ^k}\le k!C_{\theta_i}^k =\frac{k!}{2}{(\sqrt {2C_{{\theta _i}}} )^2}C_{{\theta _i}}^{k - 2}$
\end{center}
where ${C_{\theta_i}}: =  \mathop {\inf }\limits_{0 < r \le {\rm{r}}(\Theta)} {r^{-1}}{{{{{\rm{[}}{{\rm{E}}_{{\theta _i}}}{e^{r|{Y_i} - \dot b({\theta _i})|}}]}}}} $. Then for $\{{{w_i}({Y_i}}  - {\rm{E}}{Y_i})\}_{i=1}^n$, we have $
{\rm{E}}_{\theta _i}{|{{w_i}({Y_i}}  - {\rm{E}}{Y_i}) |^k}\le \frac{1}{2}k!{(\sqrt {2C_{{\theta _i}}} )^2}C_{{\theta _i}}^{k - 2}{w}^k=\frac{1}{2}k!{({w}\sqrt {2C_{{\theta _i}}} )^2}({w}C_{{\theta _i}})^{k - 2},~~k=2,3,\cdots.$
\end{proof}

Therefore, ${w_i}X_{i}\sim \mathrm{sub}\Gamma(w\sqrt {2C_{{\theta _i}}},wC_{{\theta _i}})$ by Proposition \ref{lm-Bernsteingm} and we can apply the Bernstein's inequality with the growth of moments condition to get the following concentration of exponential family on a natural parameter space.
\begin{theorem}[Concentration of exponential family] \label{lem-moment}
 Let $\{ {Y_i}\} _{i = 1}^n$ be a sequence of independent r.vs with their densities $\{f(y_i;\theta_i )\}_{i = 1}^n$ belong to canonical exponential family \eqref{eq:E-P} on the natural parameter space ${\theta _i} \in\Theta$. Given non-random weights $\{w_i\}_{i=1}^n$ with $w= \mathop {\max }_{1 \le i \le n} |w_i|>0$, then
 \begin{equation}\label{eq-BernsteinEP}
P( {| \sum\limits_{i = 1}^n {{w_i}({Y_i} - {\rm{E}}{Y_i})}| \ge t} ) \le 2\exp( - \frac{{{t^2}}}{{4w^2 \sum_{i = 1}^n {C_{{\theta _i}}} + 2w\mathop {\max }\limits_{1 \le i \le n} C_{{\theta _i}} t}}).
\end{equation}
\end{theorem}
Theorem \ref{lem-moment} has  no compact space assumption. If we impose the compact space assumption {\rm{(E.1)}} in Proposition \ref{pro-moment}, it leads to the sub-Gaussian concentration as presented in Proposition \ref{lem-moment}. The constant $C_{{\theta _i}}$ in Theorem \ref{lem-moment} is hard to determine in general exponential family with infinite support. However, if the exponential family is Poisson, the $C_{{\theta _i}}$ can be obtained as an explicit form.

\begin{theorem}[Concentration for weighted Poisson summation]\label{col:Poisson}
 Let $\{ {Y_i}\} _{i = 1}^n$ be independent $\{{\rm{Poisson}}(\lambda_i)\} _{i = 1}^n$ distributed. For non-random weights $\{w_i\}_{i=1}^n$ with $w= \mathop {\max }_{1 \le i \le n} |w_i|>0$, put $S_n^w:=\sum_{i = 1}^n {{w_i}({Y_i} - {\rm{E}}{Y_i})} $, then for all $t \ge 0$
\begin{equation}\label{eq:Poisson}
{P} (|S_n^w|\geq t)\leq 2 \exp(-{\frac {t^{2}/2}{{{w^2}\sum_{i = i}^n\lambda_i}+ w t/3}}),~~P\{|S_n^w|> w [(2t \sum\limits_{i = 1}^n\lambda_i)^{1/2}+\frac{t}{3}]\} \leq e^{-t}.
\end{equation}
\end{theorem}
\begin{proof}
We evaluate the log-MGF of centered Poisson r.vs $\{Y_{i} - {\mathrm{{E}}}Y_{i}\}_{i=1}^n$
\begin{center}
$\log {\mathrm{{E}}} {e^{s{w_i}({Y_{i}} - {\mathrm{{E}}}{Y_{i}})}} = - s{w_i}{\mathrm{{E}}}{Y_{i}} + \log
\mathrm{E} {e^{s{w_i}{Y_{i}}}} = - {\lambda_{i}}s{w_i} + \log
{e^{\lambda_i ({e^{s{w_i}}}- 1)}}
={\lambda_i ({e^{s{w_i}}}-s{w_i}- 1)}.$
\end{center}
Note that, for $s$ in a small neighbourhood of zero,
\begin{align}\label{eq:lessq}
{\lambda_i ({e^{s{w_i}}}-s{w_i}- 1)}= \lambda_i\sum\limits_{k = 2}^\infty  {\frac{{{{(s{w_i})}^k}}}{{k!}}} & \le \lambda_i\sum\limits_{k = 2}^\infty  {\frac{{{{(\left| {sw} \right|)}^k}}}{{k!}}}  = \frac{{\lambda_i{s^2}{w^2}}}{2}\sum\limits_{k = 2}^\infty  {\frac{{{{ w }^{k - 2}}}}{{k(k - 1) \cdots 3}}}  \nonumber\\
&\le \frac{\lambda_i{{s^2}{w^2}}}{2}\sum\limits_{k = 2}^\infty  {{{\left( {\frac{{\left| {sw} \right|}}{3}} \right)}^{k - 2}}} = \frac{{{s^2}}}{2}\frac{{{w^2}\lambda_i}}{{1 - w\left| s \right|/3}},
\end{align}
for $\left| s \right| \le 3/w$, which implies $w_i({Y_{i}} - {\mathrm{{E}}}{Y_{i}})\sim \mathrm{sub}\Gamma({{w^2}\lambda_i},w/3)$.
By Proposition \ref{sub-GammaConcentration}(a), we have
$ S_n^w \sim \mathrm{sub}\Gamma({{w^2}\sum_{k = i}^n\lambda_i},w/3).$
Then applying Proposition \ref{sub-GammaConcentration}(b), we get \eqref{eq:Poisson}.
\end{proof}

Before ending this section, we show a result for checking  Bernstein's moment condition by the moment recurrence condition of log-concave distributions.
\begin{definition}[Moment recurrence condition] \label{ex:log-concave}
A r.v. $Z$ is called \emph{moment bounded}  with parameter $L>0$ if it has recurrence condition
$\E |Z|^p \le \; p \; L \cdot \E|Z|^{p-1}$ for any integer $p \ge 1$.

\end{definition}
By the recursion relation,  Definition \ref{ex:log-concave} implies that any moment bounded r.v. $Z$ satisfies $\E |Z|^p \le p!L^P$. Hence, the tails of its moment bounded r.vs decay as the Bernstein's  growth of moment condition. So the constant $C_{{\theta _i}}$ in Theorem \ref{lem-moment} is relatively easy to find. Lemmas 7.2, 7.3, 7.6 and 7.7 in \cite{Schudy11} showed that any \emph{log-concave continuous distribution}(see Section 3.4) and \emph{log-concave discrete distribution} $X$ with density $f$ is moment bounded with parameter $L  \propto \mathrm{E}|X|$.

\begin{example}[Log-concave continuous distributions, \cite{Bag2005}]
Many continuous distributions, such as \emph{normal distribution}, \emph{exponential distribution}, \emph{uniform distribution over any convex set}, \emph{logistic distribution}, \emph{extreme value distribution}, \emph{chi-square distribution}, \emph{chi distribution}, \emph{hyperbolic secant distribution}, \emph{Laplace distribution}, \emph{Weibull distribution} (the shape parameter $\theta \ge 1$), \emph{Gamma distribution} (the shape parameter $a\ge 1$) and \emph{Beta distribution} (both shape parameters are $\ge 1$) have log-concave continuous densities.
\end{example}

Analogous to the log-concave continuous function in \eqref{eq:log-concave}, we can define log-concave sequence for the p.m.f. of discrete r.v., which also has Bernstein-type concentrations.

\begin{definition}[Log-concave discrete distributions] \label{ex:log-concavede}
A sequence $\{p_i\}_{i \in \mathbb{Z}}$ (or $\{p_i\}_{i \in \mathbb{N}}$) is said to be \emph{log-concave} if
$p_{i+1}^{2} \geq$ $p_{i} p_{i+2}$ for all $i \in \mathbb{Z}$ (or $i \in \mathbb{N}$).
An integer-valued r.v. $X$ is log-concave if its probability mass function (p.m.f.) $p_{i}:=P(X=i)$ is log-concave sequence.
%\fn{This should be reported earlier ba}
\end{definition}
\begin{example}[Log-concave discrete distributions]
%The \emph{binomial coefficients along any row of Pascal's triangle}, and \emph{the elementary symmetric means of a finite sequence of real numbers} are log-concave sequences.
 \emph{Bernoulli and binomial distributions}, \emph{Poisson distribution}, \emph{geometric distribution}, and \emph{negative binomial distribution} (with number of success $>1$) and \emph{hypergeometric distribution} have log-concave integer-valued p.m.f., see \cite{johnson05}.
\end{example}

\section{Sub-Weibull distributions}\label{Sub-Weibull}
%\subsection{Motivation, definition and properties}
\subsection{Sub-Weibull r.vs and $\psi_{\theta}$-norm}

%As called in ,
A r.v. is heavy-tailed if its distribution function $F(\cdot)$ fails to be bounded by a decreasing exponential function \citep{Foss2011},
\begin{center}
$\int e^{\lambda x} d F(x)=\infty, \forall \lambda>0$ (the tail decays slower than some exponential r.v.s).
\end{center}
We first give a simple example of the heavy-tailed distributions {arisen by} multiplying sub-Gaussian r.vs. The proof is motivated by Lemmas 2.7.7 of \cite{Vershynin18}.
\begin{lemma}[The product of sub-Gaussians]\label{lem: dproduct sub-gaussian}
Suppose $\{X^{(m)}\}_{m=1}^{d}$ are sub-Gaussian (may be dependent). Then $\prod\limits_{m = 1}^d {|{X^{(m)}}{|^{2/d}}} $ is sub-exponential and $
\|\prod\limits_{m = 1}^d {{{[{X^{(m)}}]}^{2/d}}} \|_\psione \le \prod\limits_{m = 1}^d\|{X^{(m)}}\|_\psitwo^{2/d}.$
\end{lemma}
\begin{proof}
By the definition of sub-Gaussian norm,
${\rm{E}}e^{ |{X^{(m)}}/{\| {{X^{(m)}}}\|_{{\psi _2}}}|^2}  \le 2,~m = 1,2, \cdots ,d.$
Applying the elementary inequality $\prod_{m = 1}^d {{a_m}}  \le \frac{1}{d}\sum_{m = 1}^d {a_m^d} $, we get by Jensen's inequality
\begin{align}\label{eq: X Y sub-w}
{\rm{E}}e^{ \prod\limits_{m = 1}^d [|{X^{(m)}}{|^{2/d}}/{{\| {{X^{(m)}}}\|}_{{\psi _2}}^{2/d}}] }  \le {\rm{E}}e^{ \frac{1}{d}\sum\limits_{m = 1}^d {{{[{X^{(m)}}/{{\| {{X^{(m)}}}\|}_{{\psi _2}}}]}^2}} } \le \frac{1}{d}\sum\limits_{m = 1}^d {{\rm{E}}e^{{[{X^{(m)}}/{{\| {{X^{(m)}}} \|}_{{\psi _2}}}]}^2}}  \le {\rm{2}}.
\end{align}
The proof is finished by the definition of the sub-exponential norm.
\end{proof}

In probability, Weibull r.vs are generated from the power of the exponential r.vs.
\begin{example}[Weibull r.vs]\label{eq:Weibullb}
The Weibull r.v. $X \in \mathbb{R}^+$ is defined by its survival function
\begin{center}
${P}(X \ge x) = e^{-bx^{\theta}}~(x\ge 0)~\text{for the scale parameter}\,b>0~\text{and the shape parameter}~\theta>0.$
\end{center}
\end{example}
Sub-Weibull distribution is characterized by the right tail of the Weibull distribution and  is a generalization of both sub-Gaussian and sub-exponential distributions.
%\fn{Can we drop this paragraph ?:Not drop}
\begin{definition}[Sub-Weibull distributions]
\label{def:subweibull}
A r.v. $X$ satisfying
$ {P}(|X| \ge x) \le ae^{-bx^{\theta}}$ for given $a, b, \theta>0$, is called a sub-Weibull r.v. with tail parameter $\theta$~(denoted by $X \sim \operatorname{subW}(\theta)$).
\end{definition}

A $\operatorname{subW}(\theta)$'s tail is no heavier than that of a Weibull r.v. with tail parameter
$\theta$.  It is emphasized that $X \sim \operatorname{subW}(\theta)$ r.vs with $\theta< 1$ belongs to heavy-tailed r.vs. Recently, the Weibull-like tail condition is also studied in high-dimensional statistics and random matrix theory [see \cite{Tao13}, \cite{Kuchibhotla18} and \cite{Wong17}]. \cite{Gotze2019} names $\operatorname{subW}(\theta)$ as $\theta$-sub-exponential r.v. There are 4 equivalent conditions to reveal the sub-Weibull tail condition which is useful in applications.
\begin{corollary}[Characterizations of sub-Weibull condition]\label{th:subWeibull}
  Let $X$ be a r.v.. Then the following properties are equivalent.
\begin{enumerate}[\rm{(}1\rm{)}]
    \item The tails of $X$ satisfy ${P}(|X| \ge x) \le e^{ - (x/ K_1)^{\theta}},~\text{for all } x \ge 0$.

    \item The moments of $X$ satisfy $
      \|X\|_k:=(\mathrm{E}|X|^{k})^{1 / k} \leq K_{2} k^{1 /\theta}~\text{for all } k \ge 1 \wedge \theta$;

    \item The MGF of $|X|^{1/\theta}$ satisfies $
      \mathrm{E}e^{ \lambda^{1/\theta} |X|^{1/\theta} }\le e^{\lambda^{1/\theta} K_3^{1/\theta} }$ for $|\lambda| \le \frac1{K_3}$;

     \item The MGF of $|X|^{1/\theta}$ is bounded at some point: ${\rm{E}}e^ {|X/K_4|^{1/\theta}}  \le 2.$

\end{enumerate}
\end{corollary}
The proof can be founded in \cite{Wong17}, \cite{Vladimirova2019} by mimicking the proof of Proposition 2.5.2 in \cite{Vershynin18}. It follows from Corollary~\ref{th:subWeibull}(4)  that $X$ is sub-Weibull with tail parameter $\theta$ if and only if $|X|^{1 / \theta}$ is sub-exponential.

 Let $\theta_{1}$ and $\theta_{2}$ be two sub-Weibull parameters. Corollary \ref{th:subWeibull} implies
$
\operatorname{subW}\left(\theta_{1}\right) \subset \operatorname{subW}\left(\theta_{2}\right).
$ for $0<\theta_{2} \leq \theta_{1}$. The following Orlicz-type norms play crucial roles in deriving tail and maximal inequality for sub-Weibull r.vs without the zero-mean assumption.
\begin{definition}[Sub-Weibull norm or $\psi_{\theta}$-norm]
	Let $\psi_{{\theta}}(x)=e^{x^{{\theta}}}-1$. The sub-Weibull norm of $X$ for any $\theta>0$ is defined as
$
		\|X\|_{\psi_{\theta}}:=\inf \{C\in(0, \infty): ~ \mathrm Ee^{|X|^{\theta}/C^{\theta}}\leq 2\}.
$
\end{definition}
From Corollary \ref{th:subWeibull}(4), a second useful definition of sub-Weibull r.vs is the r.vs with finite $\psi_{\theta}$-norm. Sub-Weibull norm is a special case of the Orlicz norm \citep{Well17}.
\begin{definition}[Orlicz norms]\label{def:IncOrliczNorm}
Let $g:\,[0, \infty) \to [0, \infty)$ be a non-decreasing convex function with $g(0) = 0$. The ``$g$-Orlicz norm'' of a r.v. $X$ is
$\|X\|_{g}:=\inf \{\eta>0: \mathrm{E}[g(|X| / \eta)] \leq 1\}.$
\end{definition}
Let $g(x)=e^{x^{{\theta}}}-1$ and $ \mathrm{E}[g(|X| / \eta)] \leq 1$ implies $\mathrm E[\exp(|X|^{\theta}/ \eta^{\theta})]\leq 2$, which is the definition of sub-Weibull norm. Similar to sub-exponential, \cite{Zajkowski19}, \cite{Wong17}, \cite{Vladimirova2019} attained %that the sub-Weibull r.v. $X$ satisfies
 the following. %  properties with sub-Weibull norm up to a constant.
\begin{corollary}[Properties of sub-Weibull norm]\label{prop:Psub-W}
\text{If}~$\E e^{{|X / \|X\|_{\psi_{\theta}}|^{\theta}}} \le 2$, then \rm{(}a\rm{)}. $P (|X| > t) \le 2 e^{-(t / \|X\|_{\psi_{\theta}})^{\theta}}~\text{for all } t \ge 0$; \rm{(}b\rm{)}. Moment bounds: $\E |X|^k \le 2 \|X\|_{\psi_{\theta}}^{k} \Gamma(\frac{k}{\theta}+1)$.
\end{corollary}

A sharpen version of Corollary \ref{prop:Psub-W} is presented as follows.
\begin{corollary}[Sharper moment bounds, Corollary 3 in \cite{Zhang20}]\label{prop:Psub-WW}
Let ${C_\theta } := \mathop {\max }\limits_{k \ge 1} {\left( {\frac{{2\sqrt {2\pi } }}{\theta }} \right)^{1/k}}{\left( {\frac{k}{\theta }} \right)^{1/(2k)}}$, we have $(\E |X|^k)^{1/k} \le  {C_\theta }{( {\theta {e^{11/12}}} )^{ - 1/\theta }} \| X\|_{\psi_{\theta}}{k^{1/\theta }}$ for all $k \geq 1$.
\end{corollary}

Via a power transform of sub-Weibull r.v., the next corollary explains the relation of sub-Weibull norm with parameter $\theta$ and $r\theta$, which is similar to Lemma \ref{lem: dproduct sub-gaussian} for sub-exponential norm.

\begin{corollary}[Corollary 4 in \cite{Zhang20}]\label{lem:2theta}
    For any $\theta,r \in(0, \infty),$ if  $X \sim \operatorname{subW}(\theta)$, then $|X|^r \sim \operatorname{subW}(\theta/r)$. Moreover,
    \begin{equation}\label{eq:2theta}
         \left\||X|^{r}\right\|_{\psi_{\theta/r}}=\|X\|^r_{\psi_{\theta}}.
    \end{equation}
    Conversely, if  $X \sim \operatorname{subW}(r\theta)$, then $X^r \sim \operatorname{subW}(\theta)$ with $\left\|X^{r}\right\|_{\psi_{\theta}} =\|X\|_{\psi_{r\theta}}^{r}$.
\end{corollary}

Lemma \ref{lem: dproduct sub-gaussian} and Corollary \ref{lem:2theta} can be extended to product of r.vs, and we state it as the following corollary.
\begin{corollary}[Proposition D.2 in \cite{Kuchibhotla18}]\label{co: dproduct subw}
 If $\{W_{i}\}_{i=1}^d$ are (possibly dependent) r.vs satisfying $\left\|W_{i}\right\|_{\psi_{\alpha_{i}}}<$ $\infty$ for some $\alpha_{i}>0,$ then
$$
\|\prod_{i=1}^{d} W_{i}\|_{\psi_{\beta}} \leq \prod_{i=1}^{d}\|W_{i}\|_{\psi_{\alpha_{i}}} \text { where } \frac{1}{\beta}:=\sum_{i=1}^{d} \frac{1}{\alpha_{i}}.
$$
\end{corollary}
\subsection{Concentrations for sub-Weibull summation}

The Chernoff inequality tricks in the derivation of Corollary \ref{sub-exponentialConcentration} for sub-exponential concentration is not valid for sub-Weibull distributions, since the exponential moment equivalent conditions of sub-Weibull are on the absolute value $|X|$. However, Bernstein's moment condition  is the exponential moment of the absolute value. An alternative method is given by \cite{Kuchibhotla18}, who defines the so-called Generalized Bernstein-Orlicz (GBO) norm. Fixed $\alpha > 0$ and $L \ge 0$, define a function $\Psi_{\theta, L}(\cdot)$ with its inverse function $
\Psi_{\theta, L}^{-1}(t) := \sqrt{\log (t+1)} + L[\log (t+1)]^{1/\theta}~\forall~t\ge 0.$ A promising development is that the following GBO norm helps us derive tail behaviors for sub-Weibull r.vs.
\begin{definition}[Generalized Bernstein-Orlicz Norm]\label{def:GBOnorm}
The \emph{generalized Bernstein-Orlicz} (GBO) norm of a r.v. $X$ is then given by: $\|X\|_{\Psi_{\theta, L}}:=\inf \{\eta>0: \mathrm{E}[{\Psi_{\theta, L}}(|X| / \eta)] \le 1\}.$
\end{definition}
The monotone function $\Psi_{\theta, L}(\cdot)$ is motivated by the classical Bernstein's inequality for sub-exponential r.vs. Like the sub-Weibull norm properties Corollary \ref{prop:Psub-W}(a), the following proposition in \cite{Kuchibhotla18} allows us to get the concentration inequality for r.vs with finite GBO norms.
\begin{corollary}[GBO norm concentration]\label{prop:GBO norm}
For any r.v. $X$ with $\norm{X}_{\Psi_{\theta, L}}< \infty$, we have $
{P}(|X| \ge \norm{X}_{\Psi_{\theta, L}}\{\sqrt{t} + Lt^{1/\theta}\}) \le 2e^{-t},~\mbox{for all}~ t\ge 0.$
%Conversely, if the tail bound~\eqref{eq:TailForm} holds for  $L > 0$, then $$\norm{X}_{\Psi_{\theta, c(\theta)L}} \le \sqrt{3}\norm{X}_{\Psi_{\theta, L}},\quad\mbox{where}\quad c(\theta) := 3^{1/\theta}/\sqrt{3}.$$
\end{corollary}

From Corollary \ref{prop:GBO norm}, it is easy to derive the concentration inequality for a single sub-Weibull r.v. or even the sum of independent sub-Weibull r.vs. Theorem 3.1 in \cite{Kuchibhotla18} obtains an upper bound for the GBO norm of the summation.
\begin{corollary}[Concentration for sub-Weibull summation]\label{thm:SumNewOrliczVex}
If $\{X_i\}_{i=1}^n$ are independent centralized r.vs %with mean zero
such that $\norm{X_i}_{\psi_{\theta}} < \infty$ for all $1\le i\le n$ and some $\theta > 0$, then for any weight vector $\bm w= (w_1, \ldots, w_n)\in\mathbb{R}^n$, we have $\|\sum_{i=1}^n w_iX_i\|_{\Psi_{\theta, L_n(\theta)}} \le 2eC(\theta)\norm{ \bm b}_2$ and
\begin{equation}\label{eq:sWc}
P(|\sum{}_{i = 1}^n {{w_i}{X_i}} | \ge 2eC(\theta ){\left\|\bm b \right\|_2}\{ \sqrt t  + {L_n}(\theta ){t^{1/\theta }}\} ) \le 2{e^{ - t}}
\end{equation}
where $\bm b = (w_1\norm{X_1}_{\psi_{\theta}}, \ldots, w_n\norm{X_i}_{\psi_{\theta}})^T\in\mathbb{R}^n$, $L_n(\theta) := \frac{4^{1/\theta}}{\sqrt{2}\norm{\bm b}_2}\times\begin{cases}\norm{\bm b}_{\infty},&\mbox{if }\theta < 1,\\
{4e\norm{\bm b}_{\frac{\theta }{{{\rm{1}} - \theta }}}}/{C(\theta)},&\mbox{if }\theta \ge 1\end{cases}$
\begin{center}
$
\text{and}~C(\theta) ~:=~ \max\{ \sqrt{2}, 2^{1/\theta}\} \times \begin{cases}\sqrt{8}e^3(2\pi)^{1/4}e^{1/24}(e^{2/e}/\theta)^{1/\theta},&\mbox{if }\theta < 1,\\
4e + 2(\log 2)^{1/\theta},&\mbox{if }\theta \ge 1.
\end{cases}
$
\end{center}
%where $\beta$ is the \emph{H\"{o}lder conjugate} of $\theta$ satisfying $1/\theta + 1/\beta = 1$ in the case $\theta \ge 1$.
\end{corollary}

The upper bound of sub-Weibull norm for summation provided by Corollary \ref{thm:SumNewOrliczVex} dependents on $\left\|X_{i}\right\|_{\psi_{\theta}}$ and the $\bm w$. \cite{Zhang20} gives a sharper version of  Corollary \ref{thm:SumNewOrliczVex}. The $\theta=1$ is the phrase transition point, and reflect the fact that Weibull r.vs are log-convex for $\theta \leq 1$ and log-concave for $\theta \geq 1$. At last, we mention a generalized Hanson-Wright inequality for sub-Weibull r.vs in Proposition 1.5 of \cite{Gotze2019}. Let
$\max _{i=1, \ldots, n}\|(a_{i j})_{j}\|_{2}:=\|A\|_{2 \rightarrow \infty}$ where $\|A\|_{p \rightarrow q}:=\sup \{\|A x\|_{q}:\|x\|_{p} \leq 1\} .$
\begin{corollary}[Concentration for the quadratic form of sub-Weibull r.vs.]\label{prop:Qsub-Weibull}
Let $q \in \mathbb{N}, A=\left(a_{i j}\right)$ be a symmetric $n \times n$ matrix and let $\{X_i\}_{i=1}^n$ be independent and centered r.vs with $\left\|X_{i}\right\|_{\Psi_{2 / q}} \leq M$ and $\mathrm{E} X_{i}^{2}=\sigma_{i}^{2} .$ We have
${P}(|\sum_{i, j} a_{i j} X_{i} X_{j}-\sum_{i=1}^{n} \sigma_{i}^{2} a_{i i}| \geq t) \leq 2 e^{-\eta(A, q, t / M^{2})/C}$, $\forall t \ge 0$,
where
$\eta(A, q, t):=\min (\frac{t^{2}}{\|A\|_{\mathrm{F}}^{2}}, \frac{t}{\|A\|_{\mathrm{op}}},(\frac{t}{\max_{i=1, \ldots, n}\|(a_{i j})_{j}\|_{2}})^{\frac{2}{q+1}},(\frac{t}{\|A\|_{\infty}})^{\frac{1}{q}})$ and $C$ is a constant.
\end{corollary}

%A random vector $\boldsymbol{X}$ in $\mathbb{R}^n$ is called sub-Weibull if its marginals $\langle {\boldsymbol{X},\boldsymbol{x}} \rangle $ are sub-Weibull r.vs for all $\boldsymbol{x} \in \mathbb{R}^n$. The sub-Weibull norm of $\boldsymbol{X}$ is defined as
%$\|\boldsymbol{X}\|_{\psi _\theta} := \sup_{\boldsymbol{x} \in S^{n-1}} \|\langle {\boldsymbol{X},\boldsymbol{x}} \rangle \|_{\psi _\theta}$.

%\section{Applications of CIs}

\section{Concentration for Extremes}\label{Extremes}

The CIs presented so far only concern with linear combinations of independent r.vs or Lipschitz function of random vectors. In many statistics applications, we have to control the maximum of the $n$ r.vs when deriving the error bounds, while these r.vs may be arbitrarily dependent. This section is developed on advanced proof skills. So we present the proofs even for existing results, which are applications of CIs in a probability aspect.
\subsection{Maximal inequalities}

%\subsubsection{Bounds for expectation of the maximum r.vs}
This section presents the maximal inequalities for r.vs. $\{X_i\}_{i=1}^n$ which may not be independent. In the theory of empirical process, it is of interest to bound $\mathrm{E}{{\max }_{1 \le i \le n} |X_i|}$ [Section 2.2, \cite{van96}]. If $\{X_i\}_{i=1}^n$ are arbitrary sequence of real-valued r.vs and have finite $r$-th moments ($r\ge1$), \cite{Aven85} gives a crude upper bounds for $\mathrm{E}{{\max }_{1 \le i \le n} {X_i}}$ by Jensen's inequality
\begin{align}\label{le:crude}
 \mathrm{E}\{\mathop {\max }\limits_{1 \le i \le n}\left|X_{i}\right|^r\}^{1/r} \le \{\mathrm{E}\mathop {\max }\limits_{1 \le i \le n}\left|X_{i}\right|^r\}^{1/r} \le {\{ \sum\limits_{i = 1}^n{\rm{E}} {\left| {{X_i}} \right|^r}\} ^{1/r}} \le {n^{1/r}} \mathop {\max }\limits_{1 \le i \le n} {{\rm{(E}}{\left| {{X_i}} \right|^r})^{1/r}}
\end{align}
Page314 of \cite{Vaart1998} mentions a sharper version of \eqref{le:crude} without the proof. In below, we introduce the proof by the truncation technique.
%\footnote{HM:VDV didn's give the proof and the truncation technique is important to show max inequality.}
\begin{corollary}[Sharper maximal inequality]\label{le:sharpmax}
  Let $\{X_i\}_{i=1}^n$ be identically distributed but not necessarily independent and assume $\mathrm{E}(|X_1|^p) < \infty,(p \ge 1)$. Then,
$
\mathrm{E}{\mathop {\max }_{1 \le i \le n} |X_i|}=o(n^{1 / p}).
$
\end{corollary}
\begin{proof}
Let $M_{n}:={\mathop {\max }_{1 \le i \le n} |X_i|}$. For any $\epsilon>0,$ we truncate $M_{n}$ by ${\epsilon n^{1 / p}}$,
$$
\begin{aligned}
 \mathrm{E}M_{n} &=\int_{0}^{\epsilon n^{1 / p}} {P}\left(M_{n}>t\right) d t+\int_{\epsilon n^{1 / p}}^{\infty}{P}\left(M_{n}>t\right) d t \leq \int_{0}^{\epsilon n^{1 / p}} 1 d t+\int_{_{\epsilon n^{1 / p}}}^{\infty} n {P}\left(|X_{1}|>t\right) d t \\ &=\epsilon n^{1 / p}+n^{1 / p} \int_{\epsilon n^{1 / p}}^{\infty} n^{(p-1) / p}{P}\left(|X_{1}|>t\right) d t\leq \epsilon n^{1 / p}+\frac{n^{1 / p}}{\epsilon^{p-1}} \int_{\epsilon n^{1 / p}}^{\infty} t^{p-1} {P}\left(|X_{1}|>t\right) d t.
 \end{aligned}
$$
Thus, by dividing $n^{1 / p}$ we have $\frac{\mathrm{E}M_{n}}{n^{1 / p}} \leq \epsilon+\frac{1}{\epsilon^{p-1}} \int_{\epsilon n^{1 / p}}^{\infty} t^{p-1}{P}\left(|X_{1}|>t\right) d t= \epsilon+o(1)$, where we adopt the fact $\int_{\epsilon n^{1 / p}}^{\infty} t^{p-1}{P}\left(|X_{1}|>t\right) d t=o(1)$ from moment condition: $\mathrm{E}\left|X_{1}\right|^{p}<\infty$. Finally, it implies that ${\limsup}_{n \rightarrow \infty} \frac{\mathrm{E}M_{n}}{n^{1 / p}} \leq \epsilon$, which gives $\mathrm{E}M_{n}=o(n^{1 / p})$ by letting $\epsilon \to 0$.
\end{proof}

\begin{example}[Pareto distribution]
For instance, let $\{Z_i\}_{i=1}^n$ be IID Pareto r.v.s with the density function $f(x)=\nu a^{\nu} x^{-\nu-1}\cdot {\rm{1}}_{\{x \ge k\}},~(k>0, \nu>0)$, and $E Z_i^{\nu-\epsilon}<\infty$ for small $\epsilon \in (0,\nu)$. Corollary \ref{le:sharpmax} gives $\E{\max_{1 \le i \le n} Z_i}=o(n^{1 /({\nu-\epsilon})})$. \cite{Malik66} showed that $\E{\max_{1 \le i \le n} Z_i}=a^r \frac{\Gamma(n+1) \Gamma\left(1-\frac{r}{v}\right)}{\Gamma\left(n+1-\frac{r}{v}\right)}=O(n^{1/\nu})$ by using Stirling's approximation of gamma function.
\end{example}

Corollary \ref{le:sharpmax} reveals that  ${{\max }_{1 \le i \le n} |X_i|}$ diverges at rate slower than $n^{1 /r}$ under the $r$-th moment condition. As $r$ increases, it will slow down the divergence rate of the maxima. If we have arbitrary finite $r$-th moment conditions (such as Gaussian distribution), it means that the divergence rate of maxima is slower than any polynomial rate $n^{1 /r}$. This suggests that the rate may be logarithmic. With the sub-Gaussian assumptions,  the logarithmic divergence rate is possible and the proof is based on controlling the expectation of the supremum of variables, from the argument in \cite{Pisier1983}.
\begin{corollary}[Sub-Gaussian maximal inequality, \cite{Rigollet19}]\label{eq:Sub-Gaussianmaximal}
Let $\{X_{i}\}_{i=1}^n$ be r.vs (without independence assumption) such that $X_{i} \sim \operatorname{subG}\left(\sigma^{2}\right).$ Then,

(a) $\mathrm{E}[\max\limits _{1 \leq i \leq n} X_{i}] \leq \sigma \sqrt{2 \log n}$ and $ \mathrm{E}[\max\limits _{1 \leq i \leq n}\left|X_{i}\right|] \leq \sigma \sqrt{2 \log (2 n)}$.

(b)
${P}(\max \limits_{1 \leq i \leq n} X_{i}>t) \leq n e^{-\frac{t^{2}}{2 \sigma^{2}}} ~ \text { and }~{P}(\max \limits_{1 \leq i \leq n}\left|X_{i}\right|>t) \leq 2 n e^{-\frac{t^{2}}{2 \sigma^{2}}}.$
\end{corollary}
\begin{proof}
(a) By the property of maximum, sub-Gaussian MGF and Jensen's inequality,
\begin{align*}
\EB \max_{1 \le i \le n} X_i& =\mathop {\inf }\limits_{s > 0}  s^{-1} \EB  \log e^{s\max\limits_{1 \le i \le n} X_i }  \le \mathop {\inf }\limits_{s > 0} s^{-1} \log \EB e^{s\max\limits_{1 \le i \le n} X_i}\le \mathop {\inf }\limits_{s > 0} s^{-1}\log\sum_{i=1}^n  \EB  e^ {s X_i}\le \mathop {\inf }\limits_{s > 0} s^{-1}   \log \sum_{i=1}^n e^{\frac{\sigma^2s^2}{2}}\\
&=\mathop {\inf }\limits_{s > 0} \left( {\frac{{\log n}}{s} + \frac{{{\sigma ^2}s}}{2}} \right) =\sigma \sqrt{2 \log n}~~[\text{Setting}~s = \sqrt {\frac{{{\rm{2}}\log n}}{{{\sigma ^2}}}}~\text{as the optimal bound.}]
\end{align*}
Let $Y_{2i-1} = X_i$ and $Y_{2i}=-X_i (1 \le i \le n)$. It gives
$\EB \max_{1 \le i \le n} |X_i| = \EB \max_{1 \le i \le n} \max\{X_i, -X_i\}  = \EB\max_{1 \le i \le 2n} Y_i.$ The previous result for sample size $2n$ finishes the proof of the second part.

(b) By Chernoff inequality and the sub-Gaussian MGF, we have
${P}(  \max_{1 \le i \le n} X_i >  t )\le \mathop {\inf }\limits_{s > 0} e^{-s t} \mathrm{E} e^{s \max\limits_{1 \le i \le n} X_i}\le \mathop {\inf }\limits_{s > 0} e^{-s t}\sum_{i=1}^n \mathrm{E} e^{s X_i} \le \mathop {\inf }\limits_{s > 0} n e^{-s t + \frac{\sigma^2s^2}{2}}\xlongequal{s={t}/{\sigma^2}} n e^{- \frac{t^2}{2 \sigma^2}}.$
For the second part, note that
${P}(  \max_{1 \le i \le n} |X_i| > t ) = {P}(  \max_{1 \le i \le n} X_i > t,\max_{1 \le i \le n} -X_i \ge t)  \le 2 {P}(  \max_{1 \le i \le n} X_i  > t ). $
\end{proof}

The Corollary 7.2 (b) implies a tail bound for maximum of sub-Gaussian random variables
 for $t>0$
\[
    P( \max_{1 \leq i \leq n} X_i \leq \sigma\sqrt{2(\log n + t )}) \geq 1- \exp \left\{ -t \right\}.
\]

By a similar proof, Corollary \ref{eq:Sub-Gaussianmaximal} can be extended to other r.vs, such as sub-Gamma r.vs and r.vs characterized by sub-Weibull norm (or Orlicz norm) as presented before.
\begin{corollary}[Concentration for maximum of sub-Gamma r.vs]\label{sub-GammaMAX}
Let $\{ X_{i}\} _{i = 1}^n $ be independent zero-mean $\{\mathrm{sub}\Gamma(\upsilon_i,c_i)\} _{i = 1}^n$. Then, for $\max_{i=1,...,n}\upsilon_{i} =: \upsilon$ and $\max_{i=1,...,n}c_{i} =: c$,
\begin{center}
$\E (\max\limits_{i=1,\cdots,n}|X_i|)  \leq [2\upsilon \log (2n)]^{1/2} + c  \log (2n)$.
\end{center}
\end{corollary}
\noindent(See Theorem 3.1.10 in \cite{Gine15} for the proof of Corollary \ref{sub-GammaMAX}.)

In bellow, based on the sub-Weibull norm condition, a fundamental result due to \cite{Pisier1983} is given for obtaining the divergence rate of the maxima of sub-Weibull r.vs.
\begin{corollary}[Maximal inequality for sub-Weibull r.vs]\label{pp-MaximalSW}
For $\theta \ge 1$, consider the sub-Weibull norm $\|X\|_{\psi_{\theta}}:=\inf_{C\in(0, \infty)}\{ ~ \mathrm E e^{|X|^{\theta}/C^{\theta}}\leq 2\}$ for $\psi_{\theta}(x)=e^{x^{\theta}}-1$. For any r.vs $\{X_{i}\}_{i=1}^n$,
\begin{equation}\label{eq-MaximalSW}
\mathrm{E}({\mathop {\max }\limits_{1 \le i \le n} |X_i|})\le \psi^{-1}_{\theta}(n)\mathop {\max }\limits_{1 \le i \le n}\|X_i\|_{\psi_{\theta}}=(\log (1+n))^{1 / \theta}\mathop {\max }\limits_{1 \le i \le n}\|X_i\|_{\psi_{\theta}}.
\end{equation}
If the function ${\psi_{\theta}}(x)$ is replaced by any non-decreasing convex function $g(x)$ with $g(0) = 0$ in the definition of Orlicz norm: $\|X\|_{g}:=\inf \{\eta>0: \mathrm{E}[g(|X| / \eta)] \le 1\}$, then
\begin{center}
$\mathrm{E}({\mathop {\max }\limits_{1 \le i \le n} |X_i|})\le g^{-1}(n)\mathop {\max }\limits_{1 \le i \le n}\|X_i\|_g$~~~~for finite $\mathop {\max }\limits_{1 \le i \le n}\|X_i\|_g$.
\end{center}
\end{corollary}
\begin{proof}
From Jensen's inequality, for $C\in(0, \infty)$ and $\psi_{\theta}(x)=e^{x^{\theta}}-1$ we get
\begin{equation}\label{eq-Maximaljs}
\psi_{\theta} [ {\mathrm{E}( { \max\limits_{1 \le i \le n} | {{X_i}}|/C} )} ] \le \mathrm{E}[ {\max\limits_{1 \le i \le n} \psi_{\theta}( {| {{X_i}} |/C} )} ] \le  {\sum_{i = 1}^n \mathrm{E}{\psi_{\theta} ( {| {{X_i}} |/C} )} } \le n
\end{equation}
where the last inequality is by the definition of sub-Weibull norm: $\mathrm{E}{\psi_{\theta} ( {| {{X_i}} |/C} )}\le 1$.

Let $C={\max }_{1 \le i \le n}\|X_i\|_{\psi_{\theta}}$. Applying the non-decreasing property of $\psi_{\theta}(x)$ (so does its inverse $\psi_{\theta}^{ - 1}(x)$), the \eqref{eq-Maximaljs} implies
${\rm{E}}( {{\max }_{1 \le i \le n}| {{X_i}}|/C} )\le \psi _\theta ^{ - 1}(n)$ by operating the map $\psi_{\theta}^{ - 1}$, and so we have \eqref{eq-MaximalSW}. The derivation of Orlicz norm case is the same.
\end{proof}

By Hoeffding's lemma, the following results on the maxima of the sum of independent r.vs, is useful for bounding empirical processes.
\begin{corollary}[Maximal inequality for bounded r.vs, Lemma 14.14 in \cite{Buhlmann11}]\label{pp-Maximalbd}
Let $\{X_{i}\}_{i=1}^n$ be independent r.vs on  $\mathcal{X}$ and $\{f_{i}\}_{i=1}^n$ be real-valued functions on $\mathcal{X}$ which satisfy
${\rm{E}}f_{j}(X_{i})=0,~\vert f_{j}(X_{i})\vert \le a_{ij}$ for all $j=1,...,p$ and all $i=1,...,n$. Then
\begin{center}
${\rm{E}}( \underset{1\le j\le p}{\max}| \sum\limits_{i=1}^{n}f_{j}(X_{i})|) \le [2\log(2p)]^{1 / 2}\underset{1\le j\le p}{\max}(\sum\limits_{i=1}^{n}a_{ij}^{2})^{1 / 2}.$
\end{center}
\end{corollary}
\begin{proof}
Let ${V_j} = \sum_{i = 1}^n {{f_j}({X_i})} $. By Jensen's inequality and Hoeffding's lemma
\begin{align*}
{\rm{E}}\mathop {\max }\limits_{1 \le j \le p} |{V_j}| &= \frac{1}{\lambda }{\rm{E}}\log {{\rm{e}}^{\lambda \mathop {\max }\limits_{1 \le j \le p} |{V_j}|}}  \le \frac{1}{\lambda }\log {\rm{E}}{{\rm{e}}^{\lambda \mathop {\max }\limits_{1 \le j \le p} |{V_j}|}}\le \frac{1}{\lambda }\log \sum\limits_{i = 1}^p{\rm{E}}{{\rm{e}}^{\lambda |{V_j}|}} \\
[\text{Corollary \ref{lm:Hoeffding}}]~&\le \frac{1}{\lambda }\log [\sum\limits_{j = 1}^p {2e^{ \frac{1}{2}{\lambda ^2}\sum\limits_{i = 1}^n {a_{ij}^2} } } ] \le \frac{1}{\lambda }\log {\rm{[2}}pe^{\frac{1}{2}{\lambda ^2}\mathop {\max }\limits_{1 \le j \le p} \sum\limits_{i = 1}^p {a_{ij}^2}}]= \frac{1}{\lambda }\log {\rm{(2}}p) + \frac{1}{2}\lambda \mathop {\max }\limits_{1 \le j \le p} \sum\limits_{i = 1}^n {a_{ij}^2} .
\end{align*}
Then ${\rm{E}}\mathop {\max }\limits_{1 \le j \le p} |{V_j}|\le \mathop {\inf }\limits_{\lambda  > 0}\{\frac{1}{\lambda }\log {\rm{(2}}p) + \frac{1}{2}\lambda \mathop {\max }\limits_{1 \le j \le p} \sum\limits_{i = 1}^n {a_{ij}^2}\}=\sqrt{2\log(2p)}\underset{1\le j\le p}{\max}(\sum_{i=1}^{n}a_{ij}^{2})^{1 / 2}$.
\end{proof}

If Hoeffding's lemma for moment is replaced by Bernstein's  moment conditions, then the maximal inequality for the sum of independent bounded r.v.s in Corollary \ref{pp-Maximalbd} can be extended to Bernstein's moment conditions. We give a modified version of Corollary 14.1 in \cite{Buhlmann11} based on truncated Jensen's inequality.
\begin{proposition}[Maximal inequality with Bernstein's moment conditions]\label{pp-MaximalB}
If $\{X_{ij}\}, j=1, \ldots, p$ are  read-valued independent variables across $i=1, \ldots, n$. Assume $\mathrm{E}X_{ij}=0$ and Bernstein's moment conditions:$
\frac{1}{n} \sum_{i=1}^{n} {\rm{E}}{\left| {{X_{ij}}} \right|^k} \le \frac{1}{2}v^2{\kappa^{k - 2}}k!,~~ k=2,3, \ldots, \forall j,$ Then,
\begin{center}
${\rm{E}}( \underset{1\le j\le p}{\max}| \frac{1}{n}\sum_{i=1}^{n}X_{ij}|^m) \le [{\textstyle{{\kappa \log (2p)} \over n}}+(v^2+1)\sqrt{\textstyle{{\log (2p)} \over n}}]^{m},~\forall~1\le m \leq 1+\log p~\text{and}~p \geq 2.$
\end{center}
\end{proposition}
\begin{proof}
Let $M_{n,m}={\max}_{1\le j\le p}| \frac{1}{n}\sum_{i=1}^{n}X_{ij}|^m$. First, we show for any r.v. $X$ and all $m \geq 1$,
\begin{align} \label{moment1}
\mathrm{E}|X|^{m} \leq \log ^{m}(\mathrm{E} e^{|X|}-1+e^{m-1}).
\end{align}
The function $g(x)=\log ^{m}(x+1), x \geq 0$ is concave for all $x \geq e^{m-1}-1 .$ By the truncated Jensen's inequality in Lemma \ref{lem:Truncated} with $Z:= e^{|X|}-1,~c=e^{m-1}-1$,  we have
\begin{center}
$
\begin{array}{c}
\mathrm{E}|X|^{m}=\mathrm{E} \log ^{m}({e}^{|X|}-1+1) \leq \log ^{m}[\mathrm{E}({e}^{|X|}-1)+1+\left({e}^{m-1}-1\right)]=\log ^{m}[\mathrm{E}( e^{|X|}-1)+e^{m-1}].
\end{array}
$
\end{center}
Then for all $L,~m>0,$
\begin{align} \label{JensenP}
(\frac{L}{n})^{-m}{\rm{E}}M_{n,m}& \le  \log ^{m}[\operatorname{E}e^{ \underset{1\le j\le p}{\max}|\sum\limits_{i=1}^{n}{X_{ij}}/{L}|}-1+e^{m-1}]\le\log ^{m}[\sum_{j=1}^{p}\operatorname{E}(e^{ |\sum\limits_{i=1}^{n}{X_{ij}}/{L}|}-1)+e^{m-1}].
\end{align}
Therefore, it is sufficient to bound $\operatorname{E}e^{ |\sum_{i=1}^{n}X_{ij}|/ L}$ uniformly in $j$.

Second. To bound the MGF in  \eqref{JensenP}, then we show that for any real-valued r.v. $X,$
\begin{align} \label{moment2}
\mathrm{E}e^{X} \leq e^{\mathrm{E} e^{|X|}-1-\mathrm{E}|X|},~\text{with}~\mathrm{E} X=0.
\end{align}
 Indeed, for any $c>0$, we have
$
e^{X-c}-1 \leq \frac{e^{|X|}}{1+c}-1=\frac{{e}^{X}-1-X+X-c}{1+c} \leq \frac{e^{|X|}-1-| X|+X-c}{1+c}.
$
Let
$
c=\mathrm{E} e^{|X|}-1-\mathrm{E}|X|.
$
Note that $\mathrm{E} X=0$, so
$\mathrm{E} e^{X-c}-1 \leq \frac{\mathrm{E} e^{|X|}-1-\mathrm{E}|X|-c}{1+c}=0.$ Hence $\log(\mathrm{E}e^{X}) \leq c$.

Using Taylor's expansion, the  \eqref{moment2} and ${e}^{|x|} \leq {e}^{x}+{e}^{-x}$ give
\begin{align}  \label{moment3}
 \operatorname{E}e^{ | \sum_{i=1}^{n}X_{ij}|/ L}-1\le \operatorname{E}e^{  \sum_{i=1}^{n}X_{ij}/ L}&+\operatorname{E}e^{ -\sum_{i=1}^{n}X_{ij}/ L}-1\le 2e^{\sum_{i=1}^{n} \mathrm{E}(e^{|X_{ij}| / L}-1-|X_{ij}| / L)}-1 = 2e^{\sum\limits_{m=2}^{\infty} \frac{\sum_{i=1}^{n} \mathrm{E}|X_{ij}|^{m}}{L^{m} m !} }-1\nonumber\\
[\text{By moment conditions}]~ &\leq 2e^{\frac{nv^2}{ 2L^{2}} \sum_{m=2}^{\infty}(\kappa / L)^{m-2}}-1=2e^{\frac{nv^2}{ 2L^{2}(1-\kappa  / L)}}-1.
\end{align}
Combining  \eqref{moment3} and \eqref{JensenP}, we obtain for $L >\kappa=:L-\sqrt{{n}/{2 \log \left( p+e^{m-1}\right)}} $
\begin{align*} \label{JensenP1}
&~~~~{\rm{E}}M_{n,m} \leq (\frac{L}{n})^{m} \log ^{m}[p(2e^{\frac{nv^2}{ 2L^{2}(1-\kappa  / L)}}-1)+e^{m-1}] {\leq} (\frac{L}{n})^{m} \log ^{m}[(p+e^{m-1})e^{\frac{nv^2}{2\left(L^{2}-L\kappa\right)}}]\\
&=(\frac{L\log ( p+e^{m-1})}{n}+{\frac{v^2L}{2(L^{2}-L\kappa)}})^{m}=(\frac{\kappa\log ( p+e^{m-1})}{n}+\sqrt{\frac{ \log \left( p+e^{m-1}\right)}{2n}}+{\frac{v^2}{2(L-\kappa)}})^{m}\\
& \le [\frac{\kappa\log ( p+e^{m-1})}{n}+(v^2+1)\sqrt{\frac{ \log \left( p+e^{m-1}\right)}{2n}} ]^{m}\le [\frac{\kappa\log (2p)}{n}+(v^2+1)\sqrt{\frac{ \log \left( 2p\right)}{2n}}]^{m}.
\end{align*}
where the second and last inequality is by ${\frac{v^2}{L-\kappa}} = [\sqrt{\frac{n}{2 \log \left( p+e^{m-1}\right)}}]^{-1}{v^2}$ and $e^{m-1} \le p$.
\end{proof}

In order to derive some error bounds, one has to control the maximum of the norm of $n$ sub-Gaussian random vectors as an application of the tail inequality for quadratic forms of sub-Gaussian vectors in Corollary \ref{thm:zhangt}. The $n$ random vectors may be arbitrarily dependent.
\begin{proposition}[Maximal inequality for quadratic forms of sub-Gaussian vectors]\label{lem:MAXG}
Let $\bm{\Sigma}=\mathbf{A}^{T} \mathbf{A}$ for a $p \times m$ matrix $\mathbf{A}$. Consider a sequence of sub-Gaussian random vectors $\{\vzeta _i:=(\zeta_{i1},...,\zeta_{im})^T\}_{i=1}^{n}$ with independent components for $\boldsymbol{\mu}_i=\mathrm{E}\vzeta _i$ and $\sigma^2$ s.t.
\begin{center}
$\mathrm{E}e^{\boldsymbol{\alpha}^{T}(\vzeta _i-\boldsymbol{\mu}_i)} \le e^{\|\boldsymbol{\alpha}\|^{2} \sigma^{2} / 2},~\forall~\boldsymbol{\alpha} \in \mathbb{R}^{m}$ for all $i$'s.
\end{center}
where $\{\vzeta _i\}_{i=1}^{n}$ may be arbitrary dependence. Then,
\begin{align*}
{P}\{\max\limits_{1 \le i \le n}\|\vzeta _i\|^{2}\le \sigma^{2}[m+2\sqrt{2m \log n }+4\log n]\}\geq 1-\frac{1}{ n}.
\end{align*}
Moreover, we have $\max\limits_{1 \le i \le n}\|\vzeta _i\|^{2}=O_p(m \vee \log n)$.
\end{proposition}
\begin{proof}
By Corollary \ref{thm:zhangt} with $\mathbf{A}:=\mathbf{I}_m$ and $\bm{\Sigma}:=\mathbf{A}^{T} \mathbf{A}=\mathbf{I}_m$, we get
\[\operatorname{tr}(\bm{\Sigma})=\operatorname{tr}(\mathbf{I}_m)=m,~~\operatorname{tr}\left(\bm{\Sigma}^{2} \right)=m,~~\|\bm{\Sigma}\|=1.
\]
Therefore, we have ${P}\{\|\mathbf{I}_m \vzeta_i\|^{2}>\sigma^{2}(m+2\sqrt{mt}+2 t)\} \leq e^{-t}$ for $i=1,2,\cdots,n$. It gives
\begin{equation}\label{eq:maxG}
{P}\{\max\limits_{1 \le i \le n}\|\vzeta_i\|^{2}>\sigma^{2}[m+2\sqrt{mt}+2 t]\}\leq \sum_{i=1}^n{P}\{\|\vzeta_i\|^{2}>\sigma^{2}[m+2\sqrt{mt}+2 t]\} \leq ne^{-t}.
\end{equation}
Choose $t$ such that $p_n:=ne^{-t} \to 0$. For example, let $t := \log n-\log{p_n}=2\log n$ with $p_n=\frac{1}{ n}$. Inequality \eqref{eq:maxG} shows
\begin{align*}\label{eq:maxG}
{P}\{\max\limits_{1 \le i \le n}\|\vzeta_i\|^{2}\le \sigma^{2}[m+2\sqrt{2m \log n }+4\log n]\}&={P}\{\max\limits_{1 \le i \le n}\|\vzeta_i\|^{2}\le \sigma^{2}[m+2\sqrt{mt}+2 t]\}\\
& \geq 1-\frac{1}{ n}.
\end{align*}

Note that $m+2\sqrt{2m \log n }+4\log n \lesssim m \vee \log n$. Let $p_n=\frac{1}{rn}$ for $r>0$, and $t := \log n-\log{p_n}=2\log n+\log{r}$. Check the definition of the stochastic order,
 $${P}\{\max\limits_{1 \le i \le n}\|\vzeta_i\|^{2}\ge \sigma^{2}[m+2\sqrt{m(2\log n+\log{r})}+2 (2\log n+\log{r})]\} \le \frac{1}{rn},~~\forall r>0.$$
Then we get $\max\limits_{1 \le i \le n}\|\vzeta_i\|^{2}=O_p(m \vee \log n)$.
\end{proof}

\subsection{Concentration for suprema of empirical processes}\label{bracketingEP}
%\subsubsection{Sharper bounds for suprema of empirical processes}\label{bracketing}
%\subsubsection{Dvoretzky-Kiefer-Wolfowitz inequality and its concise proof}
Let $\{X_{i}\}_{i=1}^n$ be a random sample from a measure {{${P}$}} on a measurable space $(\mathcal{X},\mathcal{A})$. %The $\mathcal{X}$ may be multivariate space or even functional space.
The empirical distribution $\mathbb{P}_{n}:=n^{-1} \sum_{i=1}^{n} \delta_{X_{i}},$ where $\delta_{x}$ is the probability mass of 1 at $x$. Given a measurable function $f : \mathcal{X} \mapsto \mathbb{R}$, let $\mathbb{P}_{n} f:=\frac{1}{n} \sum_{i=1}^{n} f\left(X_{i}\right)$ be the expectation of $f$ under the empirical measure $\mathbb{P}_{n}$, and ${P} f:=\int f dP$ be the expectation under {{${P}$}}. The $\mathbb{P}_{n} f$ is called the \emph{empirical process} index by $n$.

%Historically,

The study of the empirical processes begins with the uniform limit law of EDF in Example \ref{eg:edf}. The Glivenko-Cantelli theorem extends the LLN for EDF and gives uniform
convergence: $\left\|\mathbb{F}_{n}-F\right\|_{\infty}=\sup _{t \in \mathbb{R}}\left|\mathbb{F}_{n}(t)-F(t)\right| \stackrel{\text { as }}{\rightarrow} 0.$ Moreover, a stronger result than Example \ref{eg:edf} is the Dvoretzky-Kiefer-Wolfowitz (DKW) inequality \citep{Dvoretzky56}
\begin{equation}\label{eq:DKW}
P({\sup{}_{x\in \mathbb{R} }}|\mathbb{F}_{n}(x)-F(x)|>\varepsilon )\leq 2e^{-2n\varepsilon ^{2}}\quad \forall \varepsilon >0.
\end{equation}

The DKW inequality is a uniform version of Hoeffding's inequality, which also strengthens the Glivenko-Cantelli theorem since \eqref{eq:DKW} implies Glivenko-Cantelli: $\|\mathbb{F}_{n}-F\|_{\infty }\xrightarrow{a.s.}0$ by the first Borel-Cantelli lemma:
$\sum_{n=1}^{\infty}P(|Y_n|\geq \epsilon) < \infty~\text{for any}~\varepsilon  > 0\Rightarrow Y_n \xrightarrow{a.s.} 0$ for any sequence of r.v.s $\{X_{i}\}_{i=1}^n$. \cite{Dvoretzky56} proves $P({\sup_{x\in \mathbb{R} }}|\mathbb{F}_{n}(x)-F(x)|>\varepsilon )\leq Ce^{-2n\varepsilon ^{2}}$ with an unspecified constant $C$. \cite{Massart96} attains the sharper constant $C = 2$. 

In some statistical applications, given an estimator $\hat{\theta},$ and $f_{\hat{\theta}}(X_i)$ is a function of $X_i$ and ${\hat{\theta}}$. We want to study its asymptotic properties for sums  of $f_{\hat{\theta}}(X_i)$ that changes with both $n$ and $\hat{\theta}$,
\begin{center}
$\frac{1}{n} \sum_{i=1}^{n}[f_{\hat{\theta}}\left(X_{i}\right)-\mathrm{E} f_{\hat{\theta}}(X_{i})],~(\text{a dependent sum}).$
\end{center}

A possible route to attain results is via the suprema of the empirical process for all possible  the `` true'' parameter $\theta_{0}$ on a set $K$:
\begin{center}
$
\frac{1}{n} \sum_{i=1}^{n}[f_{\hat{\theta}}\left(X_{i}\right)-\mathrm{E} f_{\hat{\theta}}(X_{i})]\le \sup_{\theta_{0} \in K} | \frac{1}{n} \sum_{i=1}^{n}[f_{\theta_{0}}\left(X_{i}\right)-\mathrm{E} f_{\theta_{0}}(X_{i})]|=: \sup_{{\theta_{0}} \in K} |(\mathbb{P}_{n}-P) f_{\theta_{0}}|.
$
\end{center}
Fortunately,  the summation in the sup enjoys independence. So, the study of convergence rate suprema of empirical processes is important if we consider a functional class ${\cal F}$ instead of the set $K$ such that :
${\mathop {\sup }_{f \in {\cal F}} |(\mathbb{P}_{n}-P) f|}=\sup_{\theta_0 \in K} |(\mathbb{P}_{n}-P)f_{\theta_{0}}|.$

 Let $(\mathcal{F},\|\cdot\|)$ be %a subset of
a normed space of real functions $f: \mathcal{X} \rightarrow \mathbb{R}$. For a probability measure $Q$, define the $L_{r}(Q)$-space with $L_{r}(Q)$-norm by $\|f\|_{L_{r}(Q)}=\left(\int|f|^{r} d Q\right)^{1 / r}$. Given two functions $l(\cdot)$ and $u(\cdot),$ the bracket $[l, u]$ is the set of all functions $f \in \mathcal{F}$ with $l(x) \leq f(x) \leq u(x),$ for all $x \in \mathcal{X} .$ An $\varepsilon$-bracket is a bracket $[l, u]$ with $\|l-u\|_{L_{r}(Q)}<\varepsilon$. The \emph{bracketing number} $N_{[~]}\left({\varepsilon}, \mathcal{F}, L_{r}(Q)\right)$ is minimum number of $\varepsilon$-brackets needed to cover $\mathcal{F}$, i.e.
\begin{center}
$N_{[~]}\left({\varepsilon}, \mathcal{F}, L_{r}(Q)\right)=\inf \{n: \exists l_{1}, u_{1}, \ldots, l_{n}, u_{n} \text { s.t. } \cup_{i=1}^{n}\left[l_{i}, u_{i}\right]=\mathcal{F} \text { and } \|l_n-u_n\|_{L_{r}(Q)} < \varepsilon \}.$
\end{center}
The \emph{covering number} $N\left(\varepsilon, \mathcal{F}, L_{r}(Q)\right)$ is the minimal number of $L_{r}(Q)$-balls of radius $\varepsilon$ needed to cover the set $\mathcal{F}$. The uniform covering numbers is $\sup _{Q} N(\varepsilon\|F\|_{ L_{r}(Q)}, \mathcal{F}, L_{r}(Q))$, where the supremum is taken over all probability measures $Q$ for which $\|F\|_{ L_{r}(Q)}>0$. Two conditions to get the convergence of $\sup_{f \in {\cal F}} |(\mathbb{P}_{n}-P)f_{\theta}|$ are the finite \emph{bracketing number} condition with $L_{1}(P)$-norm in Theorem 19.4 of \cite{Vaart1998} (or finite uniform covering numbers in Theorem 19.13 of \cite{Vaart1998}).
\begin{lemma}[Glivenko-Cantell class] \label{NA-GC}
For every class $\mathcal{F}$ of measurable functions, if
\begin{center}
$N_{[~]}\left(\varepsilon, \mathcal{F}, L_{1}(P)\right)<\infty$~(or $\sup _{Q} N(\varepsilon\|F\|_{ L_{1}(Q)}, \mathcal{F}, L_{1}(Q))<\infty$ with $P^{*} F<\infty$) for every $\varepsilon>0$.
\end{center}
Then, $\mathcal{F}$ is $P$-Glivenko-Cantelli, i.e. ${ {\sup }_{f \in {\cal F}} |(\mathbb{P}_{n}- {P})f|}\stackrel{\text { as }}{\rightarrow} 0$.
\end{lemma}
\begin{proof}
Choose the $\varepsilon$-brackets $\{ [l_i, u_i] \}$ containing the minimal number $N :=  N_{[~]}\left(\varepsilon, \mathcal{F}, L_{1}(P)\right) < \infty$ of the intervals, in which the items have $P|u_i - l_i| = \|u_i - l_i\|_{L_{1}(P)} < \varepsilon$, then for every $f \in \mathcal{F}$, there are some $i$ rendering the following inequalities held:
\begin{equation*}
     \begin{aligned}
     & \mathbb{P}_n f - Pf \leq (\mathbb{P}_n - P)u_i + P(u_i - f) < \max_{1 \leq i \leq N} (\mathbb{P}_n - P)u_i + \varepsilon
         \end{aligned}
\end{equation*}
\begin{equation*}
    Pf - \mathbb{P}_n f \leq P(f - l_i) - (\mathbb{P}_n - P )l_i < \varepsilon - \min_{1 \leq i \leq N} (\mathbb{P}_n - P)l_i
\end{equation*}
By SLLN, $\max_{1 \leq i \leq N} (\mathbb{P}_n - P)u_i, \, \min_{1 \leq i \leq N} (\mathbb{P}_n - P)l_i \, \xrightarrow{a.s.} \, 0$, that complementing with any $\varepsilon > 0$ leads the ultimate result $\|\mathbb{P}_n f - Pf \|_{\mathcal{F}} = \sup_{f \in \mathcal{F}} |\mathbb{P}_n f - Pf| \, \xrightarrow{a.s.} \, \varepsilon, \quad \varepsilon \downarrow 0$. So the $\mathcal{F}$ is P-Glivenko-Cantelli.
\end{proof}
\begin{example} [Empirical process with indicator
functions, Example 19.4 of \cite{Vaart1998}]
Let $\mathcal{F}$ be  the collection of all indicator
functions of the form $f_{t}(x)=1_{(-\infty, t]}(x),$ with $t$ ranging over $\mathbb{R}.$  Then,  $\mathcal{F}$ is  \text{P-Glivenko-Cantelli}. Indeed,
Consider brackets of the form $[1_{\left(-\infty, t_{i-1}\right]}, 1_{\left(-\infty, t_{i}\right)}]$ for a grid of points $-\infty=t_{0}<$
$t_{1}<\cdots<t_{k}=\infty$ with the property $F\left(t_{i}-\right)-F\left(t_{i-1}\right)<\varepsilon$ for each $i .$ These brackets
have $L_{1}(P)$-size $\varepsilon,$ namely
$$P |1_{\left(-\infty, t_{i}\right)}-1_{\left(-\infty, t_{i-1}\right]}|\le \E [1_{\left(-\infty, t_{i}\right)}(X)-1_{\left(-\infty, t_{i-1}\right]}(X)]<\varepsilon.$$
Their total number $k$ can be chosen smaller than $\lceil 1 / \varepsilon\rceil \leq 2 / \varepsilon$, where $\lceil r\rceil=\min \{m \in \mathbb{Z} ; m \leq r+1\}$ is the upper integer part of $r \in \mathbb{R}$ (i.e. is the smallest integer that is greater than or equal to $x .$). So
$$
N_{[~]}\left(\varepsilon, \mathcal{F}, L_{1}(P)\right)\le k <2 / \varepsilon<\infty \text { for every } \varepsilon>0.
$$
We conclude that the class $\mathcal{F}$ is  \text{P-Glivenko-Cantelli}, and it shows the Glivenko-Cantelli theorem.
\end{example}

\begin{example} [Weighted empirical process with dependent weights]\label{eg:Weighted}
Suppose we observe a sequence of IID observations $\{\left({X}_{i}, Y_{i}\right)\}_{i=1}^n$ drawn from a random pair $({X}, Y)$. Given some weighted functions $W(\cdot)$ and a bounded estimator $\hat t \in (0,\tau]$, we want to study the stochastic convergence of \textit{dependent weighted empirical processes}
\begin{center}
$|\frac{1}{n} \sum_{i=1}^{n} 1\left(Y_{i} \geq \hat t\right) W\left({X}_{i}\right)-\mu\left(\hat t ; W\right)|~~(\le \sup _{0 \leq t \leq \tau}|\frac{1}{n} \sum_{i=1}^{n} 1\left(Y_{i} \geq t\right) W\left({X}_{i}\right)-\mu\left(t ; W\right)|)$
\end{center}
where $\mu\left(t ; W\right)=\E_{{X}, Y}\{1(Y \geq t)W\left({X}_{i}\right)\}<\infty$ and $W\left({X}_{i}\right) \leq U_{f}<\infty$ and $ \tau <\infty$.

Consider the class of functions indexed by $t$,
\begin{center}
$\mathcal{F}=\left\{1(y \geq t) W\left(x\right)/ U_{f}: t \in[0, \tau], y \in \mathbb{R}, W\left(x\right)\leq U_{f}\right\}.$
\end{center}
It is crucial to evaluate $N_{[~]}\left({\varepsilon}, \mathcal{F}, L_{1}(Q)\right)$. Given $\epsilon \in (0,1),$ let $t_{i}$ be the $i$-th $\lceil 1 / \varepsilon\rceil$ quantile of $Y,$ thus
$
P\left(Y \leq t_{i}\right)=i \varepsilon,~i=1, \ldots,\left\lceil\frac{1}{\varepsilon}\right\rceil-1.
$
Furthermore, take $t_{0}=0$ and $t_{\lceil 1 / \varepsilon\rceil}=+\infty .$ For $i=1,2, \cdots,\lceil 1 / \varepsilon\rceil,$ define brackets $\left[L_{i}, U_{i}\right]$ with
\begin{center}
$L_{i}(x, y)=1\left(y \geq t_{i}\right) \frac{W\left(x\right)}{U_{f}}, \quad U_{i}(x, y)=1\left(y>t_{i-1}\right) \frac{W\left(x\right)}{U_{f}}$
\end{center}
such that $L_{i}(x, y) \leq 1(y \geq t) e^{f_{\theta}(x)} / U_{f} \leq U_{i}(x, y)$ as $t_{i-1}<t \leq t_{i} .$ The Jensen's inequality gives
$\E|U_{i}-L_{i}|\le |E[\frac{W\left(X_i\right)}{U_{f}}\left\{1\left(Y \geq t_{i}\right)-1\left(Y>t_{i-1}\right)\right\}]|  \le |P\left(t_{i-1}<Y \leq t_{i}\right)|=\varepsilon.$ Therefore, $N_{[~]}\left({\varepsilon}, \mathcal{F}, L_{1}(P)\right) \leq\lceil 1 / \varepsilon\rceil <\infty \text { for every } \varepsilon>0$. So the class $\mathcal{F}$ is  P-Glivenko-Cantelli.
\end{example}
If the upper bounds of $N_{[~]}\left(\varepsilon, \mathcal{F}, L_{2}(P)\right)$ and $\sup{}_{Q} N\left(\varepsilon, \mathcal{F}, L_{2}(Q)\right)$ have polynomial rates w.r.t. $O(1/\varepsilon)$, the following tail bound estimate gives the convergence rate of suprema of empirical processes in Lemma~\ref{NA-GC} obtained by \cite{Talagrand94}. It extends DKW inequality to general empirical processes with the bounded function classes.
\begin{lemma}[Sharper bounds for suprema of empirical processes]\label{tm:Talagrand}
Consider a probability space $(\Omega, \Sigma, P),$ and consider $n$ IID r.vs $\{X_{i}\}_{i=1}^n$ valued in $\Omega,$ of law $P$. Let $\mathcal{F}$ be a class of measurable functions $f: \mathcal{X} \mapsto[0,1]$ that satisfies bracketing number conditions:
$$N_{[~]}\left(\varepsilon, \mathcal{F}, L_{2}(P)\right) \leq ({K}/{\varepsilon})^{V}~(\text{or}~\sup{}_{Q} N\left(\varepsilon, \mathcal{F}, L_{2}(Q)\right) \leq({K}/{\varepsilon})^{V}),~\text { for every } 0<\varepsilon<K.$$
Then, for every $t>0$
\begin{center}
$P(\sqrt n  {\mathop {\sup }\limits_{f \in {\cal F}} |(\mathbb{P}_{n}- \mathbb{P})f|\ge t} ) \leq\left(\frac{D(K) t}{\sqrt{V}}\right)^{V} e^{-2 t^{2}}$ with a constant $D(K)$ depending on $K$ only.
\end{center}
\end{lemma}
The explicit constant $D(K)$ can be found in \cite{Zhang06}, who studied the tail bounds for the suprema of \emph{the unbounded and non-IID empirical process}. \cite{Kong14} derived the rate of convergence for the Lasso regularized Cox models by using sharper concentration inequality for the suprema of empirical processes in Example \ref{tm:Talagrand} related to the negative log-partial likelihood function. In Example \ref{eg:Weighted}, we have
\[{\left\{ {{\rm{E}}|{U_i} - {L_i}{|^2}} \right\}^{1/2}} \le {\{ {E{[\frac{{W\left( {{X_i}} \right)}}{{{U_m}}}\left\{ {1\left( {Y \ge {t_i}} \right) - 1\left( {Y > {t_{i - 1}}} \right)} \right\}]^2}} \}^{1/2}} \le |P\left( {{t_{i - 1}} < Y \le {t_i}} \right)|^{1/2} = \sqrt \varepsilon,\]
which implies $N_{[~]}\left(\sqrt{\varepsilon}, \mathcal{F}, L_{2}(P)\right) \leq\lceil 1 / \varepsilon\rceil \leq 2 / \varepsilon~\text { for every } \varepsilon>0$. Hence, $N_{[~]}\left({\varepsilon}, \mathcal{F}, L_{2}(P)\right) \leq 2 / \varepsilon^2$. By applying Lemma \ref{tm:Talagrand} with $V=2,~K=\sqrt 2$, we have
\begin{center}
$P({\sup }_{0 \le t \le \tau } |\frac{1}{{{U_f}\sqrt n }}\sum\limits_{i = 1}^n {[1} \left( {{Y_i} \ge t} \right)W\left( {{X_i}} \right) - \mu \left( {t;W} \right)]| \ge t) \le \frac{{{D^2}(\sqrt 2)}}{2}{t^2}{e^{ - 2{t^2}}}.$
\end{center}
%Let $t = R\sqrt {\log \log n}$ for a constant $R>0$. We obtain
%$\frac{{{D^2}(\sqrt 2)}}{2}{t^2}{e^{ - 2{t^2}}} = \frac{{{D^2}(\sqrt 2)}}{2}({R^2}\log \log n){e^{\log {{(\log n)}^{ - {\rm{2}}{R^{\rm{2}}}}}}} = \frac{{{D^{\rm{2}}}(2){R^2}\log \log {\rm{n}}}}{{{\rm{2}}{{(\log n)}^{2{R^2}}}}}.$ Then it derives the non-asymptotic version of the LIL
%\begin{center}
%$P(\mathop {\sup }\limits_{0 \le t \le \tau } |\frac{1}{n}\sum\limits_{i = 1}^n 1 \left( {{Y_i} \ge t} \right){W\left( {{X_i}} \right)} - \mu \left( {t;W} \right)| \le \frac{{R\sqrt {\log \log n} }}{\sqrt n})
%\ge 1-\frac{{{D^{\rm{2}}}(\sqrt 2){R^2}\log \log {\rm{n}}}}{{{\rm{2}}{{(\log n)}^{2{R^2}}}}}.$
%\end{center}
%
%More result of the LIL in empirical processes are presented in Page268 of
%\cite{Vaart1998}.

In the following, we present an alternative method to derive constant-specified concentration inequalities for suprema of the empirical process of independent samples over an unbounded class of functions.  If the high-probability upper bound of \begin{center}
$\{  {\sup }_{f \in {\cal F}}\sqrt n|(\mathbb{P}_{n}- \mathbb{P})f|\}-  \E \{  {\sup }_{f \in {\cal F}}\sqrt n|(\mathbb{P}_{n}- \mathbb{P})f|\}$ available.
\end{center}
 Then, next two results are the \emph{symmetrization theorem} and the \emph{contraction theorem}, which are fundamental tools to get sharper bounds for suprema of empirical processes.

\begin{lemma}[Symmetrization theorem, Lemma 2.3.1 in \cite{Vaart1998}]\label{tm:Symmetrization}
Let  $\{\textit{\textbf{X}}_{i}\}_{i=1}^n$ be independent r.vs with values in some space $\mathcal{X}$ and $\mathcal{F}$ be a class of measurable real-valued functions on $\mathcal{X}$. Let $\{\epsilon_{i}\}_{i=1}^n$ be a Rademacher sequence with uniform distribution on $\{ - 1,1\}$, independent of $\{\textit{\textbf{X}}_{i}\}_{i=1}^n$ and $f\in \mathcal{F}$. If ${\rm{E}}|f(\textit{\textbf{X}}_{i})|< \infty$ for all $i$, then
\begin{center}
${\rm{E}}\{{\sup}_{f \in \mathcal{F}}\Phi ( \sum_{i=1}^{n}[ f(\textit{\textbf{X}}_{i})-{\rm{E}}f(\textit{\textbf{X}}_{i})] )\}\le {\rm{E}}\{{\sup}_{f \in \mathcal{F}}\Phi[2\sum_{i=1}^{n} \epsilon_{i}f(\textit{\textbf{X}}_{i})]\}$
\end{center}
for every nondecreasing, convex $\Phi(\cdot): \mathbb{R} \mapsto \mathbb{R}$ and class of measurable functions $\mathcal{F}$.
\end{lemma}
\begin{lemma}[Contraction theorem]\label{lm:Contraction}
Let $x_{1},...,x_{n}$ be the non-random elements of $\mathcal{X}$ and $\varepsilon_{1},...,\varepsilon_{n}$ be Rademacher sequence. Consider $c$-Lipschitz functions $g_{i}$, i.e. $\left| {{g_i}(s) - {g_i}(t)} \right| \le c \left| {s - t} \right|,\forall s,t \in \mathbb{R}$. Then for any function $f$ and $h$ in $\mathcal{F}$, we have
\begin{center}
${\rm{E}}_{\epsilon}[{\sup }_{f \in {\mathcal F}} | {\sum_{i = 1}^n {{\varepsilon_i}} [ {{g_i}\{f({x_i})\} - {g_i}\{h({x_i})\}} ]}|] \le 2c {\rm{E}}_{\epsilon}[{\sup }_{f \in {\mathcal F}} | {\sum_{i = 1}^n {{\varepsilon_i}\{f({x_i}) - h({x_i})\}} } |].$
\end{center}
\end{lemma}
A improved symmetrization theorem is studied in Theorem 3.4 of \cite{kashlak2018measuring}, which replaces the constant 2 to a tighter constant $1+O(1/{\sqrt n})$. By refine the constant $2c$ to $c$, a sharper contraction theorem is given in Proposition 4.3 and (4.9) in \cite{Bach21}. A gentle introduction to suprema of empirical processes and its statistical applications are nicely presented in \cite{Sen2018}.

\begin{example}[Expected bound for supremum of bounded EP]
${\rm{E}}[ {\mathop {\sup }\limits_{f \in \mathcal{F}} \frac{1}{{\sqrt n }}\sum\limits_{i = 1}^n f(X_{i}) } ]$ can easily be evaluated when $\mathcal{F}$ is a set linear functionals. Let $r>0$ and $\|\cdot\|$ denote the Euclidean norm on $\mathbb{R}^{p}$. Put
\begin{center}
 $r \mathbf{B}=\left\{\mathbf{a} \in \mathbb{R}^{p}:\|\mathbf{a}\| \le r\right\}$, $\mathcal{F}=\left\{f: \mathbb{R}^{p} \rightarrow \mathbb{R}: \exists~\mathbf{a} \in r \mathbf{B}, f(\mathbf{x})=\mathbf{a}^{T} \mathbf{x}\right\}.$
\end{center}
Using Symmetrization Theorem,
$$
\begin{aligned}
{\rm{E}}[ {\mathop {\sup }\limits_{f \in \mathcal{F}} \frac{1}{{\sqrt n }}\sum\limits_{i = 1}^n f(X_{i}) } ] &\le \mathrm{E}[\sup _{\mathbf{a} \in r \mathbf{B}}\frac{2}{\sqrt{n}} \sum_{i=1}^{n} \epsilon_{i} \mathbf{a}^{T} X_{i}] =\mathrm{E}[\sup _{\mathbf{a} \in r\mathbf{B}}\mathbf{a}^{T}(\frac{2}{\sqrt{n}} \sum_{i=1}^{n} \epsilon_{i} X_{i})]\\
(\text{Cauchy's inequality})~&\le r\mathrm{E}\|\frac{2}{\sqrt{n}} \sum_{i=1}^{n} \epsilon_{i} X_{i}\|\le r(\mathrm{E}\|\frac{2}{\sqrt{n}} \sum_{i=1}^{n} \epsilon_{i} X_{i}\|^2)^{1/2}~\\
&=\frac{ 2r}{\sqrt{n}} (\sum_{1 \leqslant i, j \leqslant n}{\rm{E}}\{ \mathrm{E}\left[\epsilon_{i} \epsilon_{j} X_{i}^{T} X_{j}\right]|\bm X\})^{1/2}\\
({\{\varepsilon_{i}\}_{i=1}^n~\text{are conditionally IID}})~&=\frac{ 2r}{\sqrt{n}} (\sum_{1 \leqslant i=j \leqslant n}{\rm{E}}\{\epsilon_{i} \epsilon_{j} \mathrm{E}[ X_{i}^{T} X_{j}]|\bm X\})^{1/2}= {2r (\frac{1}{n}\sum\limits_{i = 1}^n {\rm{E}} [X_i^T{X_i}])^{1/2} }.
\end{aligned}
$$
Finally, McDiarmid's inequality shows with prob. at least $1 - {e^{ - {{2M}^2}/{[\frac{1}{n}\sum\limits_{i = 1}^n {{{\left( {{b_i} - {a_i}} \right)}^2}]}} }}$,
\[\small \mathop {\sup }\limits_{f \in F} \frac{1}{n}\sum\limits_{i = 1}^n f ({X_i}) \le {\frac{1}{\sqrt n}{\rm{E}}[ {\mathop {\sup }\limits_{f \in \mathcal{F}} \frac{1}{\sqrt n}\sum\limits_{i = 1}^n f(X_{i}) } ]} + \frac{M}{{\sqrt n }} \le \frac{1}{{\sqrt n }}[{2r (\frac{1}{n}\sum\limits_{i = 1}^n {\rm{E}} [X_i^T{X_i}])^{1/2} } + M].\]
\textbf{Remark}: {In general, ${\rm{E}}[ {\mathop {\sup }\limits_{f \in \mathcal{F}} \frac{1}{{\sqrt n }}\sum\limits_{i = 1}^n f(X_{i}) } ]$ is bounded by the \emph{uniform entropy integral} evaluated by VC dimension of the general $\mathcal{F}$, see Theorem 3.5.4 in Gin{\'e} and Nickl(2015).}
\end{example}

To further bound $\E \{  {\sup }_{f \in {\cal F}}\sqrt n|(\mathbb{P}_{n}- \mathbb{P})f|\}$ in Lemma \ref{tm:Symmetrization} with $\Phi(t)=|t|$, Theorem 3.5.4 in \cite{Gine15} gave a constants-specified upper bound for the expectation of suprema of unbounded empirical processes.

\begin{lemma}[Moment bound for suprema of unbounded empirical processes]\label{lm:ContractionEntropy}
Let $\mathcal{F}$ be a countable class of measurable functions with $0 \in \mathcal{F},$ and let $F$ be a strictly positive envelope for $\mathcal{F}$. Assume that $J(\mathcal{F}, F, t):=\int_{0}^{t} \sup _{Q} \sqrt{\log [2 N(\mathcal{F}, L_{2}(Q), \tau\|F\|_{L_{2}(Q)})]} d \tau<\infty$ for some (for all) $t>0$. Given $\mathcal{X}$-valued IID
r.vs $\{X_{i}\}_{i=1}^n$ with law $P$ s.t. $P F^{2}<\infty$. Set $U=\max_{1 \leq i \leq n} F\left(X_{i}\right)$ and $\delta=\sup _{f \in \mathcal{F}} \sqrt{P f^{2}} /\|F\|_{L^{2}(P)} .$ Then, for $A_{1}=8 \sqrt{6} \text { and } A_{2}=2^{15} 3^{5 / 2}$,
\begin{equation}\label{eq:EEntropy}
\E \{\sup{}_{f \in {\cal F}}\sqrt n |(\mathbb{P}_{n}-{P})f|\} \leq A_{1}\|F\|_{L^{2}(P)} J(\mathcal{F}, F, \delta)\vee[{A_{2}\|U\|_{L^{2}(P)} J^{2}(\mathcal{F}, F, \delta)}/({\sqrt{n} \delta^{2}})].
\end{equation}
\end{lemma}
The bound in \eqref{eq:EEntropy} matches the non-asymptotically sub-exponential CLT in \eqref{eq:subEclt}, and it reveals that $\sup{}_{f \in {\cal F}}\sqrt n |(\mathbb{P}_{n}-{P})f|$ has the sub-exponential behaviour, although with a huge parameter (the constant $A_{2}=2^{15} 3^{5 / 2}$ is so large). Recently, Theorem 2 in \cite{Baraud2020} sharpened bound \eqref{eq:EEntropy} when $\mathcal{F}$ takes values in $[-1,1]$. {\color{black}{The \emph{uniform entropy integral} $J(\mathcal{F}, F, t)$ can be evaluated by VC dimension of $\mathcal{F}$, see Theorem 7.11 in \cite{Sen2018}. Applying \cite{Adamczak2008}'s tail inequalities for the summation $Z_n:=\sup _{f \in \mathcal{F}}n|(\mathbb{P}_{n}-{P})f|$}} with unbounded $\mathcal{F}$, they obtained following result.
\begin{lemma}[Tail estimates for suprema of empirical processes under sub-Weibull norms]\label{lm:Adamczak}
Let $\{X_{i}\}_{i=1}^n$ be independent $\mathcal{X}$-valued r.vs  and let $\mathcal{F}$ be a countable class of measurable functions $f: \mathcal{X} \rightarrow \mathbb{R}$. For some $\alpha \in(0,1]$, assume $\|\sup _{f \in \mathcal{F}} | f\left(X_{i}\right)-\mathrm{E} f\left(X_{i}\right)| \|_{\psi_{\alpha}}<\infty$ for every $i$. Define
$
\sigma_n^{2}=\sup _{f \in \mathcal{F}} \sum_{i=1}^{n} \mathrm{Var} f\left(X_{i}\right).
$
For all $\eta \in (0,1)$ and $\delta>0,$ then there exists a constant $C_{\alpha, \eta, \delta}>0$ s.t.
 both ${P}(Z_n\geq(1+\eta) \mathrm{E} Z_n+t)$ and ${P}(Z_n\leq(1-\eta) \mathrm{E}Z_n-t)$ are bounded by
  {\color{black}{
\begin{center}
$\delta_{n,t,\eta,\delta}(\sigma_n^{2},\alpha):=\exp (-\frac{t^{2}}{2(1+\delta) \sigma_n^{2}})+3 \exp (-(\frac{t}{C_{\alpha, \eta, \delta}\|\max\limits_{i} \sup _{f \in \mathcal{F}}|f\left(X_{i}\right)-\mathrm{E} f\left(X_{i}\right)|\|_{\psi_{\alpha}}})^{\alpha})$ for all $t \geq 0$.
\end{center}}}
\end{lemma}
So, Lemmas \ref{lm:ContractionEntropy} and \ref{lm:Adamczak} give ${P}({n}^{-1}Z_n\geq(1+\eta){n}^{-1/2}[\text{Right hand side of }\eqref{eq:EEntropy}]+t)\le {P}(Z_n\geq(1+\eta){n}^{1/2}\mathrm{E}[{n}^{-1/2} Z_n]+nt)\le \delta_{n,nt,\eta,\delta}(\sigma_n^{2},\alpha)$. We have
\begin{align*}
P(\sup{}_{f \in {\cal F}}|(\mathbb{P}_{n}-{P})f|\le (1+\eta){n}^{-1/2}[\text{Right hand side of }\eqref{eq:EEntropy}]+t)\ge 1-\delta_{n,nt,\eta,\delta}(\sigma_n^{2},\alpha),
\end{align*}
which extends Example \ref{eg:u} from bound to unbound empirical processes.

The constant-unspecific version of Lemma \ref{lm:ContractionEntropy} (Lemma 19.36-19.38 in \cite{Vaart1998} or other versions) has wide applications in deriving the rate of convergence for kernel density estimations, M-estimators in high-dimensional and increasingly-dimensional regressions; see \cite{Gine15}, \cite{Buhlmann11}, \cite{Pan2020} and references therein.

 We illustrate a proof for a weak version of DKW inequality based on the L\'evy inequality and symmetrization theorem.
\begin{example}[A proof of DKW inequality, \cite{Oliveira16}]
A weak version of DKW inequality is
\begin{equation*}\label{eq:KDW}
P(\sup{}_{x\in \mathbb{R} }|\mathbb{F}_{n}(x)-F(x)|>\varepsilon )\leq 4e^{-n\varepsilon ^{2}/8}\quad \forall \varepsilon >0.
\end{equation*}
The proof is divided into 3 steps.
\begin{lemma}[L\'evy inequality, p72 in \cite{chow1997probability}] \label{lemma:levy}
Let $\{X _ { i } \}_{i=1}^n$ be independent r.vs, let $S _ { k } = \sum _ { i = 1 } ^ { k } X _ { i }$ and be ${\rm{Med}}(X)$ as the median of $X$, then
$P\{ \max _ { 1 \leq k \leq n } [ S _ { k } - {\rm{Med}} ( S _ { k } - S _ { n }) ] \geq x \} \leq 2 P \{ S _ { n } \geq x \}$.
\end{lemma}
\begin{remark}
It is worth noting that the inequalities of Hoeffding, sub-exponential, Bernstein and Sub-weibull also hold for the maximum of the partial-sums $ \max _ { 1 \leq k \leq n }  S _ { k }$ by virtue of Doob's sub-martingale inequality (Section 7.2.c, 7.3, 7.5 in \cite{Lin11}).
\end{remark}
\textbf{Step1}. We denote functional class $\mathcal{F} = \{ \frac{ 1 }{n} 1_{(-\infty, x]}(t) : x \in \mathbb{R}  \} \cup \{ \frac{ -1 }{n} 1_{(-\infty, x]}(t) : x \in \mathbb{R}  \}$. By applying symmetrization theorem for the increasing convex function $\Psi(x) = e^{\theta x}$, then
\begin{align*}
&~~~~\mathrm{E}[e^{\theta \sup_{x \in \mathbb{R}}|\mathbb{F}_{n}(x)-F(x)|}]
= {\rm{E}}\{ \Psi (  \underset{f \in \mathcal{F}}{\sup} \sum_{i=1}^{n}[f(X_i)-{\rm{E}}f(X_i) ] )\}= {\rm{E}}\{ \underset{f \in \mathcal{F}}{\sup} \Psi ( \sum_{i=1}^{n}[ f(X_i)-{\rm{E}}f(X_i)])\}\\
&\le {\rm{E}}\{ \underset{f \in \mathcal{F}}{\sup}
\Psi(2\sum_{i=1}^{n} \epsilon_{i}f(X_i) ) \}=\mathrm{E}[\sup _{b \in\{ \pm 1\}, x \in \mathbb{R}} e^{\frac{2 \theta b}{n} \sum\limits_{i=1}^{n}  \epsilon_{i}   {1}_{(-\infty, x]}(X_i)}]=\mathrm{E}[\mathrm{E}(\sup _{b \in\{ \pm 1\}, x \in \mathbb{R}}e^{\frac{2 \theta b}{n} \sum\limits_{i=1}^{n}  \epsilon_{i}   {1}_{(-\infty, x]}(X_i)}| \textit{\textbf{X}})].
\end{align*}
where  $\textit{\textbf{X}}:=(X_1,\cdots,X_n)^T$ and $b$ is a binary r.v. with value ${\pm 1 }$.

\textbf{Step2}. Next, conditioning on $\textit{\textbf{X}}:=(X_1,\cdots,X_n)^T=(x_1,\cdots,x_n)^T=:\textit{\textbf{x}}$, we analyze the last term in \textbf{Step1}. We permute the indices $1,2, \ldots, n,$ so that $x_{1} \leq x_{2} \leq \ldots \leq x_{n}$. No matter how to permute, the permutation doesn't change the law of $\sum_{j=1}^{n} \epsilon_{j} {1}_{(-\infty, x]}\left(x_{j}\right)$. Now,
$$
\sum_{j=1}^{n} \epsilon_{j} {1}_{(-\infty, x]}\left(x_{j}\right)=\left\{\begin{array}{ll}{0,} & {\text {if } x<x_{1}},~(i=0) \\ {\sum_{j=1}^{i} \epsilon_{j}} & {\text {if } x_{i} \leq x<x_{i+1},~(1 \leq i<n)} \\ {\sum_{j=1}^{n} \epsilon_{j}} & {\text { if } x \geq x_{n}},~(i=n)\end{array}\right.
$$
and thus we have $\sup\limits_{b \in\{ \pm 1\}, x \in \mathbb{R}} e^{\frac{2 \theta b}{n} \sum_{j=1}^{n} \epsilon_{j} {1}_{(-\infty, x]}(x_{j})}=\sup\limits_{b \in\{ \pm 1\}, 0 \leq i \leq n} e^{\frac{2 \theta b}{n} \sum_{j=1}^{i} \epsilon_{j}}\leq \max\limits_{0 \leq i \leq n}\{e^{\frac{2 \theta}{n} \sum_{j=1}^{i} \epsilon_{j}}+e^{\frac{-2 \theta}{n} \sum_{j=1}^{i} \epsilon_{j}}\}.$ Note that $\{\epsilon_{i}\}_{i=1}^n$ are symmetric about 0, we get:
\begin{align}\label{eq:ESUP}
\mathrm{E}[\sup _{b \in\{ \pm 1\}, x \in \mathbb{R}} e^{\frac{2 \theta b}{n} \sum_{i=1}^{n}  \epsilon_{i}   {1}_{(-\infty, x]}(X_{i})}|\textit{\textbf{X}}=\textit{\textbf{x}}]
\le 2 \mathrm{E}[ \max_{0 \le i \le n} e^{ \frac{2\theta}{n} \sum_{j=1}^i \epsilon_{i}} ]= 2 \mathrm{E}[e^{ \frac{2\theta}{n}  \max\limits_{0 \le i \le n} \sum_{j=1}^i \epsilon_{i}} ]\le 4 e^{\frac{2\theta^2}{n}}
\end{align}
where the last inequality will be proved in \textbf{Step3}.

\textbf{Step3}. Let $\{ {\varepsilon _i}\} _{i = 1}^n$ be i.i.d sequence of \textbf{Rademacher r.vs} taking values on $\{1,-1\}$, its MGF is $\mathrm{E}e^{\theta \epsilon_{j}}\le \exp (\frac{\theta^{2}}{2})$ by Hoeffding's lemma. Setting $S_k  = \sum_{j=1}^{k} \epsilon_{j}$, we note that the distribution of $\{ {\varepsilon _i}\} _{i = 1}^n$ is symmetric and so is $S_k - S_n$, which implies that $\text{Median}(S_k - S_n) = 0$.
	By L\'evy inequality (Lemma \ref{lemma:levy}), we have $P(\max_{0 \le k \le n} S_k \ge t) \le 2 P(S_n \ge t).$ Then applying identity for differentiable function $f$
\begin{align}\label{eq:differentiable}
\mathrm{E}[f(X) {1}_{\{X \geq 0\}}]&=\mathrm{E}[(f(0)+\int_{0}^{X} f^{\prime}(t) d t) {1}_{\{X \geq 0\}}]=f(0) \mathrm{E}[{1}_{\{X \geq 0\}}]+\mathrm{E}[\int_{0}^{\infty} f^{\prime}(t) {1}_{\{X \geq t\}} d t {1}_{\{X \geq 0\}}]\nonumber\\
&=f(0) \mathrm{E}[{1}_{\{X \geq 0\}}]+\int_{0}^{\infty} f^{\prime}(t) \mathbf{P}[X \geq t] d t
\end{align}
and L\'evy inequality, we get
\begin{align*}
\mathrm{E}[\exp (\theta \max_{0 \leq i \leq n} \sum_{j=1}^{i} \epsilon_{j})]
&=1 + \theta\int_{0}^{\infty} e^{\theta t} P(\max_{0 \le i \le n} S_i \ge t ) dt\le 1 + 2\theta\int_{0}^{\infty} e^{\theta t} P(S_n \ge t ) dt\\
&=  1+ 2 \mathrm{E} [ (e^{\theta S_n} -1 ){1}_{\{S_n \ge 0\}}  ]\le 2 \mathrm{E}e^{\theta S_n} \le 2 e^{\frac{\theta^{2} n}{2}}~[\text{Hoeffding's lemma}].
\end{align*}
where the last equality is by \eqref{eq:differentiable}: $\int_{0}^{\infty} f^{\prime}(t) \mathbf{P}[X \geq t] d t=\mathrm{E}[f(X) {1}_{\{X \geq 0\}}]-f(0) \mathrm{E}[{1}_{\{X \geq 0\}}]$ and the second last inequality is by L\'evy inequality (for $\max\limits_{0 \leq i \leq n} S_i\ge 0$ \footnote{Here we define $S_0=0$ as the symmetric random walk starting at 0.} and $S_n$): $P(S_n \ge 0) \ge 1/2$.

 Therefore by \eqref{eq:ESUP}, we have sub-Gaussian MGF of suprema of empirical processes
\[  \mathrm{E}[e^{\theta \sup _{x \in \mathbb{R}}\left|\mathbb{F}_{n}(x)-F(x)\right|}]\le \mathrm{E}[\mathrm{E}(\sup {}_{b \in\{ \pm 1\}, x \in \mathbb{R}}e^{\frac{2 \theta b}{n} \sum_{i=1}^{n}  \epsilon_{i}   {1}_{(-\infty, x]}(X_i)}| \textit{\textbf{X}})] \le  4 e^{\frac{2\theta^2}{n}}.   \]
Finally, by Chernoff's inequality, we obtain  the DKW inequality \eqref{eq:KDW}:
$P(\sup _{x\in \mathbb{R} }|\mathbb{F}_{n}(x)-F(x)|>\varepsilon )\leq \inf _{\theta>0} 4 e^{\frac{2 \theta^{2}}{n}-\theta \epsilon}=4 e^{-n\varepsilon ^{2}/8}$ for every $\varepsilon >0.$
\end{example}
\section{Concentration for High-dimensional Statistics}

With the emergence of high-dimensional (HD) data such as the gene expression data,  there are renewed interests on the CIs. One aspect of the HD data is such that the number of variables $p$ can be comparable to or even greater than the sample size $n$. This section provides results in three commonly encountered settings: \textit{increasing-dimensional} ($p_n=o(n)<n$), \textit{large-dimensional} ($p_n=O(n)$) and \textit{high-dimensional} setting ($p_n \gg n,~p_n=e^{o(n)}$).

\subsection{Linear models with diverging number of covariates}\label{se;lm}

Suppose that we have an $n$-dimensional random vector $\emph{\textbf{Y}}$ which contains $n$ responses $\{Y_i\}_{i=1}^n$ to $p$ covariates ${\emph{\textbf{X}}_i} = ({x_{i1}}, \cdots ,{x_{ip}})^T$, respectively. The $n$ copies of ${\emph{\textbf{X}}_i}$ as row vectors make a $n \times p$ design matrix $\textbf{X} = ({\emph{\textbf{X}}_1^T}, \cdots ,{\emph{\textbf{X}}_n^T})^T$. The conditional expectation  ${\mathrm{E}[Y_i|{\emph{\textbf{X}}_i}]}$ is linearly related to a coefficient vector $\boldsymbol{\beta}^{*}= {(\beta _1^*, \cdots ,\beta _p^*)^T}$ such that
%.  The \emph{linear model} is given by the matrix form
\begin{equation} \label{eq:LMs}
{\bm Y} = {\mathbf X}\bm\beta^* + \bm \varepsilon
\end{equation}
where $\{\varepsilon _i\}_{i=1}^n$ in the error vector $\boldsymbol{\varepsilon}  = {({\varepsilon _1}, \cdots ,{\varepsilon _n})^T}$ are IID  with zero mean and finite variance $\sigma ^2$.
%  = {\rm{Var}}(\varepsilon _1)$.
 The $\boldsymbol{\beta}^{*}$  needs to be estimated.

This subsection only considers the case that $p$ is increasing but $p<n$. The ordinary least square (OLS) estimator is

\iffalse
For the case $n<p \to \infty$ , it is the HD statistical inferences which would be discussed in Section \ref{linear}, \ref{Poisson} and \ref{test}.
\fi

\begin{equation}\label{eq:ols}
{\hat {\boldsymbol{\beta}}}_{LS} %= \mathop {{\rm{argmin}}}\limits_{\boldsymbol{\beta} \in {\mathbb{R}^p}} \sum\limits_{i = 1}^n {{( {{Y_i} - \sum\limits_{j = 1}^p {{x_{ij}}} {\beta _j}} )^2}}
 = \mathop {{\rm{argmin}}}{}_{{\boldsymbol{\beta}} \in {\mathbb{R}^p}} ||\textit{\textbf{Y}} - {\textbf{X}\boldsymbol{\beta}||_2^2}.
\end{equation}
%where $\left\|  \cdot  \right\|$ is the $\ell_2$ norm of the vector, namely $\| \bm{a}\| _2= {(\sum\nolimits_{i = 1}^n {a_i^2} )^{1/2}}$ for a vector $\bm{a} = {({a_1}, \cdots {a_n})^T}$.

%Set $Q(\boldsymbol{\beta})=\|\textit{\textbf{Y}} - \mathbf{X}\boldsymbol{\beta}\|_2^2$. Note that $Q(\boldsymbol{\beta})$ is convex, and its gradient w.r.t. $\bm{\beta}$ is $\dot Q(\boldsymbol{\beta})=2 \mathbf{X}^T (\mathbf{X}\boldsymbol{\beta} -\textit{\textbf{Y}})$ and thus the Hessian matrix w.r.t. $\bm{\beta}$ is $\ddot Q(\boldsymbol{\beta})=2 \mathbf{X}^T \mathbf{X}$.
 Assume ${\rm{rank}}(\mathbf{X})=p$, which is not hard to meet since $p < n$,  %the OLS estimator
$\hat {\boldsymbol{\beta}}_{LS}=(\mathbf{X}^T\mathbf{X})^{-1}\mathbf{X}^T\textit{\textbf{Y}}$
is the unique solution of the %optimization problem
\eqref{eq:ols}.  The following result for the OLS estimator {is well known.}  % given below.

\begin{lemma} %[Exercises 1.6.5 in \cite{Giraud2014}]
 \label{tm:basic-ols}
Under the assumptions on the linear models and %we assume that $\bm X_{1}, \cdots, \bm X_{n}$ be fixed $\mathbb{R}^{p}$ design vectors. Let
the rank of $\textbf{X}$ is $p$, then\\ %we have the following fact concerning quadratic forms\\
	(i) Let $\mathbf{A}$ be a $p \times n$ matrix, then
${\rm E}\|\mathbf{A}\boldsymbol{\varepsilon}\|_2^2 = {\rm E}({\boldsymbol{\varepsilon}^T}{\mathbf{A}^T}\mathbf{A}\boldsymbol{\varepsilon}) = {\sigma ^2}{\rm{tr}}({\mathbf{A}^T}\mathbf{A})$.

\noindent (ii)(\textbf{\rm The curse of dimensionality}) The \emph{mean square error} and the average \emph{in-sample $\ell_{2}$ risk} of the OLS estimator are
${\rm E}\|\hat {\boldsymbol{\beta}}_{LS}  - {\boldsymbol{\beta} ^*}\|_2^2 = {\rm{tr}}((\mathbf{X}^T\mathbf{X})^{-1}){\sigma ^2}~\text{and}~\frac{1}{n}{\rm E}\|\mathbf{X}(\hat {\boldsymbol{\beta}}_{LS}- {\boldsymbol{\beta} ^*})\|_2^2 =\frac{p{\sigma ^2}}{n}.$
\end{lemma}
\begin{proof}
(i) As $\|\mathbf{A}{\bm\varepsilon}\|_2={\bm\varepsilon}^T\mathbf{A}^T\mathbf{A}{\bm\varepsilon}$ is scalar, we have
$ {\bm\varepsilon}^T\mathbf{A}^T\mathbf{A}{\bm\varepsilon} ={\rm{tr}}({\bm\varepsilon}^T\mathbf{A}^T\mathbf{A}{\bm\varepsilon})={\rm{tr}}({\bm\varepsilon}{\bm\varepsilon}^T \mathbf{A}^T\mathbf{A}).$
Therefore, by linearity of the trace operator, we have
$$
{\rm E}\|\mathbf{A}{\bm\varepsilon}\|_2^2={\rm E}({\bm\varepsilon}^T\mathbf{A}^T\mathbf{A}{\bm\varepsilon})={\rm E}[{\rm{tr}}({\bm\varepsilon} {\bm\varepsilon}^T\mathbf{A}^T\mathbf{A})]={\rm{tr}}[{\rm E}({\bm\varepsilon}{\bm\varepsilon}^T) \mathbf{A}^T\mathbf{A}]=\sigma^2{\rm{tr}}(\mathbf{A}^T\mathbf{A}).
$$

(ii) Substituting $\bm{Y}=\mathbf{X}\bm\beta^*+\bm\varepsilon$ into the formula of $\hat{\bm \beta}$, we have
\begin{align}\label{eq:lse}
\hat{\bm{\beta}}=(\mathbf{X}^T\mathbf{X})^{-1}\mathbf{X}^T\bm{Y}&=(\mathbf{X}^T\mathbf{X})^{-1}\mathbf{X}^T(\mathbf{X}\bm\beta^*+\bm\varepsilon)=\bm{\beta}^*+(\mathbf{X}^T\mathbf{X})^{-1}\mathbf{X}^T\bm{\varepsilon}.
\end{align}
Therefore, $\hat{\bm{\beta}}-{\bm{\beta}}^*=(\mathbf{X}^T\mathbf{X})^{-1}\mathbf{X}^T\bm\varepsilon$, which readily derives the first results of (ii) by letting $\mathbf{A}=(\mathbf{X}^T\mathbf{X})^{-1}\mathbf{X}^T$ in (i). From (i), according to (\ref{eq:lse}), we have
\begin{align}\label{eq:lsetr}
  \mathrm{E}\|\mathbf{X}(\bm{\hat{\beta}}-\bm{\beta}^*)\|_2^2 &=\mathrm{E}\|\mathbf{X}(\mathbf{X}^T\mathbf{X})^{-1}\mathbf{X}^T\bm\varepsilon\|_2^2 ={\rm{tr}}[\mathbf{X}(\mathbf{X}^T\mathbf{X})^{-1}\mathbf{X}^T\mathbf{X}(\mathbf{X}^T\mathbf{X})^{-1}\mathbf{X}^T]\sigma^2 \nonumber\\
   & ={\rm{tr}}[(\mathbf{X}^T\mathbf{X})^{-1}\mathbf{X}^T\mathbf{X}]\sigma^2 ={\rm{tr}}(\mathbf{I}_p)\sigma^2=p\sigma^2.
\end{align}
\end{proof}
\begin{remark}As $p,n \to \infty $ with $p < n$, part (ii) implies that the OLS estimator may had poor performance if $p/n \to c>0$. The average in-sample $\ell_{2}$-risk tends to zero if $p_{n}=o(n)$.
\end{remark}

Put $\hat {\boldsymbol{\beta}}:=\hat {\boldsymbol{\beta}}_{LS}$. Let $\{\lambda_{i}\left(\mathbf{X}^T\mathbf{X}\right) \}_{i=1}^k$ be the eigenvalue values of $\mathbf{X}^T\mathbf{X}$. Markov's inequality and Lemma \ref{tm:basic-ols} with $\mathbf{A}=\left(\mathbf{X}^T\mathbf{X}\right)^{-1} \mathbf{X}^T$ implies
\begin{align*}
 P\{\|\hat{\bm\beta}-{\bm\beta}^*\|_2>t\} \le \frac{\sigma^{2} \operatorname{tr}[\left(\mathbf{X}^T\mathbf{X}\right)^{-1}]}{t^{2}}=\frac{\sigma^{2}}{t^{2}} \sum_{i=1}^{p} \frac{1}{\lambda_{i}\left(\mathbf{X}^T\mathbf{X}\right)}  \leq \frac{\sigma^{2}}{t^{2}} \frac{p}{\lambda_{\min }\left(\mathbf{X}^T\mathbf{X}\right)}=:{\delta _n},
\end{align*}
which implies that,  with probability greater than $1-{\delta _n}$,  %  non-asymptotical behavior the $\ell_2$ estimation error $\|\hat{\bm\beta}-{\bm\beta}^*\|_2$
\begin{align}\label{eq:ls1e}
\|\hat{\bm\beta}-{\bm\beta}^*\|_2 \le \sigma \sqrt {{p}/{n}}\cdot  [{\delta _n}{\lambda _{\min }}( {{\textstyle{1 \over n}}{{\bf{X}}^T}{\bf{X}}})]^{-1/2}.
\end{align}

Assume that $p:=p_n=o(n^r)$ as $n \rightarrow \infty,~p_n<n$. We specific two groups of regularity conditions and the value of $r$ such that the $l_2$ consistency ($\|\hat{\bm\beta}-{\bm\beta}^*\|_2\stackrel{p}{\longrightarrow}0$) is true.

(1) By Lemma \ref{tm:basic-ols}, if $\frac{1}{n}\mathbf{X}^T\mathbf{X}$ is uniformly positive definite ($\exists~ c>0$ s.t. $\frac{1}{n}\mathbf{X}^T\mathbf{X}\succ c\mathbf{I}_p$) then
\begin{center}
$\frac{p{\sigma ^2}}{n}=\frac{1}{n}\mathrm{E}\|\mathbf{X}(\bm{\hat{\beta}}-\bm{\beta}^*)\|_2^2=\mathrm{E}[(\bm{\hat{\beta}}-\bm{\beta}^*)^T\frac{1}{n}\mathbf{X}^T\mathbf{X}(\bm{\hat{\beta}}-\bm{\beta}^*)]\ge c\mathrm{E}\|\hat{\bm\beta}-{\bm\beta}^*\|_2^2.$
\end{center}
If $p=o(n)$, then $\mathrm{E}\|\hat{\bm\beta}-{\bm\beta}^*\|_2^2=o(1)$ which implies $\|\hat{\bm\beta}-{\bm\beta}^*\|_2=o_p(1)$.

(2) From \eqref{eq:ls1e}, if $p = o({\lambda _{\min }}\left( {{\mathbf{X}^{T}}\mathbf{X}} \right))$, we have $\|\hat{\bm\beta}-{\bm\beta}\|_2=o_p(1)$. In this case, if we consider: ``$\frac{1}{n}\mathbf{X}^T\mathbf{X}$ is positive definite" in (1), it also leads to $p= o({\lambda _{\min }}\left( {{\mathbf{X}^{T}}\mathbf{X}} \right))=o(n)$.
%\fn{do you mean (1) ?}
%As an application of Corollary \ref{thm:zhangt}, we evaluate the non-asymptotical upper bounds for excess loss in fixed-design linear regressions via the OLS estimator, where the number of parameter may be divergent.

In \eqref{eq:LMs} with fixed design, suppose that the $\varepsilon_{1}, \ldots, \varepsilon_{n}$ are sub-Gaussian zero-mean noise for which there exists a $\sigma>0$ such that $\mathrm{E}e^{\sum_{i=1}^{n} \alpha_{i}\varepsilon_i} \le e^{\sigma^{2} \sum_{i=1}^{n} \alpha_{i}^{2}},~\forall \alpha_{1}, \ldots, \alpha_{n} \in \mathbb{R}$. Suppose that the Gram matrix $\bm{S}_n:=\frac{1}{n}\mathbf{X}^T\mathbf{X}$ is invertible. The \textit{excess in-sample prediction error} $R(\hat{\bfbeta})$ is the difference between the expected squared error for $\bm X_{i}^{T}\hat{\bfbeta}$ and for $\bm X_{i}^{T}\bfbeta^*:$
\begin{align}\label{olsR}
R(\hat{\bfbeta}):&=\frac{1}{n}\{\mathrm{E}[ \sum_{i=1}^{n}(\bm X_{i}^{T} \boldsymbol{\beta}-Y_{i})^{2}]-\mathrm{E}[ \sum_{i=1}^{n}(X_{i}^{T} \bfbeta^*-Y_{i})^{2}]\}|_{\boldsymbol{\beta}=\boldsymbol{\hat\beta}}\nonumber\\
&=\frac{1}{n}\{\mathrm{E}[ \sum_{i=1}^{n}(\bm X_{i}^{T}{\bfbeta}-\bm X_{i}^{T}\bfbeta^*+\bm X_{i}^{T}\bfbeta^*-Y_{i})^{2}]|_{\boldsymbol{\beta}=\boldsymbol{\hat\beta}}-\mathrm{E}[ \sum_{i=1}^{n}(X_{i}^{T} \bfbeta^*-Y_{i})^{2}]\}\nonumber\\
&=\frac{1}{n}\|\mathbf{X}(\hat {\boldsymbol{\beta}}  - {\boldsymbol{\beta} ^*})\|_2^2 +\frac{1}{n}\mathrm{E}[ \sum_{i=1}^{n}(\bm X_{i}^{T}{\bfbeta}-\bm X_{i}^{T}\bfbeta^*)\cdot \varepsilon_i]|_{\boldsymbol{\beta}=\boldsymbol{\hat\beta}}=\frac{1}{n}\|\mathbf{X}(\mathbf{X}^T\mathbf{X})^{-1}\mathbf{X}^T\bm\varepsilon\|_2^2,
\end{align}
which is a quadratic form of sub-Gaussian vector.

By Corollary \ref{thm:zhangt} with $\mathbf{A}:=\mathbf{X}(\mathbf{X}^T\mathbf{X})^{-1}\mathbf{X}^T/\sqrt{n}$, $\bfxi:=\bm\varepsilon,~\bm \mu =\bm 0$ and $\bm{\Sigma}:=\mathbf{A}^{T} \mathbf{A}=\mathbf{X}(\mathbf{X}^T\mathbf{X})^{-1}\mathbf{X}^T/n$,
\[\operatorname{tr}(\bm{\Sigma})=\operatorname{tr}((\mathbf{X}^T\mathbf{X})^{-1}\mathbf{X}^T\mathbf{X})/n=p/n,~~\operatorname{tr}\left(\bm{\Sigma}^{2} \right)=p/n^2,~~\|\bm{\Sigma}\|_2=1/n.
\]
where last identity is due to $\mathbf{X}(\mathbf{X}^T\mathbf{X})^{-1}\mathbf{X}^T$ being a projection matrix. Thus %(eigenvalues of projection matrix are zero or one)$. Thus,
$P[R(\hat{\bfbeta})>\frac{\sigma^{2}(p+2 \sqrt{p t}+2 t)}{n}] \le e^{-t}$, i.e. with probability $1-e^{-t}$,
%the non-asymptotic upper bound for the excess loss is
\begin{center}
$R(\hat{\bfbeta})\le \frac{\sigma^{2}(p+2 \sqrt{p t}+2 t)}{n}.$
\end{center}
For Gaussian noise, $\mathrm{E}R(\hat{\bfbeta})=\frac{\sigma^{2} p}{n}$ in Lemma \ref{tm:basic-ols}, so $P\{R(\hat{\bfbeta})-\mathrm{E}R(\hat{\bfbeta})\le \frac{\sigma^{2}(2 \sqrt{p t}+2 t)}{n}\}\ge 1-e^{-t}$.

\subsection{Non-asymptotic Bai-Yin theorem for random matrix}\label{Bai-Yin}

Let $\mathbf{A}$ be a $p \times p$ Hermitian matrix with real eigenvalues: $\lambda_{\max}:=\lambda_1 \ge \cdots \ge \lambda_p=:\lambda_{\min}$. The \emph{empirical spectral distribution} (ESD) of $\mathbf{A}$ is
\begin{center}
$F_{\mathbf{A}}(x)=\frac{1}{p}\mathop{\sum_{j=1}^p {\rm{1}}(\lambda_j\leq x)}, $
\end{center}
which resembles the EDF of IID samples.  Let $\{{\mathbf{A}}_n\}_{n \geq 1}$ be a sequence of $p \times p$ Hermitian random matrices indexed by the sample size $n$, and $F_{{\mathbf{A}}_n}$ be the ESD of ${\mathbf{A}}_n$.

A major interest in random matrix theory is to investigate the convergence of $F_{{\mathbf{A}}_n}$ as a sequence of distributions to a limit $F$. Recall the $F$ is said to be the limit of $\{F_{{\mathbf{A}}_n}\}_{n \geq 1}$ if
$\mathop{\lim_{n\to \infty}F_{{\mathbf{A}}_n}}(x) =F(x)~\hbox{for all}~x \in C(F), $ where $C(F)$ is the set of continuous point set of $F(x)$. In multivariate statistics, it is of interest to study the \emph{sample covariance matrix} $\mathbf{S}_{n}:=\frac{1}{n}\mathbf{X}{\mathbf{X}^T}$ where the double array $\mathbf{X}=\left\{X_{i j}, i=1, \ldots, p ; j=1, \ldots, n\right\}$ contains zero-mean IID r.vs $\{X_{i j}\}$ with variance $\sigma^{2}.$ Suppose that the dimensions $n$ and $p$ grow to infinity while $p/n$ converges to a constant in $[0,1]$. \cite{Marcenko67} gives the limit behavior of {the ESD} of $\mathbf{S}_{n}$. \cite{Bai1993} obtained a strong version of the Mar{\v{c}}enko-Pastur law.

\begin{corollary}[Bai-Yin theorem]\label{Bai1993}
Let $\mathbf{X}$ be an $n \times p$ random matrix whose entries are independent copies of a r.v. with zero mean, unit variance, and finite fourth moment ($\mathrm{E} | X_{11}|^{4}<\infty$). As $n \rightarrow \infty, p \rightarrow \infty, p / n \rightarrow y \in(0,1)$, then
\begin{center}
$\mathop{\lim_{n \to \infty}}  \lambda_{\min}(\mathbf{S}_{n}) =\sigma^2(1-\sqrt y)^2,~~~~
	\mathop{\lim_{n \to \infty}}  \lambda_{\max}(\mathbf{S}_{n}) =\sigma^2(1+\sqrt y)^2~a.s..$
\end{center}
\end{corollary}

Note that $\lambda_i(\mathbf{S}_{n})=  \lambda_{i}(\mathbf{X}/{\sqrt n }) $ for all $i$, Bai-Yin's law asserts that if $\sigma^2=1$:
  $
  \lambda_{\min}(\mathbf{X}/{\sqrt n }) =1 - \sqrt {{p}/{n}}  + o(\sqrt {{p}/{n}} ), ~
 \lambda_{\max}(\mathbf{X}/{\sqrt n }) =1 + \sqrt {{p}/{n}}  + o(\sqrt {{p}/{n}} )~
  \text{a.s.}.
  $

Theorem 4.6.1 in \cite{Vershynin18} studies the non-asymptotic upper and lower bounds of the extreme eigenvalues of $\mathbf{S}_{n}$ with independent sub-exponential entries, but the bounds contained un-specific constants. We give a constant-specified version:
%We assume that the rows in random matrices are \emph{isotropic random vectors}, namely $\mathrm{Var}(\bm{Y})=\mathbf{I}_p.$

% We see that $\bm{X}$ is isotropic if $
%\mathrm{E}\langle \bm{X}, \boldsymbol{a}\rangle^{2}=\|\boldsymbol{a}\|_{2}^{2}~\text { for all } \boldsymbol{a} \in \mathbb{R}^{n}.$

\begin{proposition}[Constants-specified non-asymptotic Bai-Yin theorem] \label{NA-BY}
  Let $\mathbf{X}$ be an $n \times p$ matrix whose rows $\bm{X}_i$ are independent
  sub-Gaussian random vectors in $\mathbb{R}^p$ with $\mathrm{Var}(\bm{X}_i)=\mathbf{I}_p$. Define $Z_i:=|\langle {\boldsymbol{X}_i,\boldsymbol{x}} \rangle|,~\forall~\boldsymbol{x} \in S^{n-1}$. Further assume that $\{ {Z_i^2-1}\} _{i = 1}^n$ are $\operatorname{subE}(\theta)$, then
\begin{equation}\label{P1}
{P} \{ \big\| {n}^{-1}\mathbf{X}^T\mathbf{X}-\mathbf{I}_p \big\|\leq 2c\theta\max \left( \delta , \delta ^ { 2 } \right) \} \geq 1-2 e^{-ct^2},~~t \ge 0
\end{equation}
  where $\delta  =2c(\sqrt {{p}/{n}}  + {t}/{{\sqrt n }})$ with $t = c\theta\max \left( \delta , \delta ^ { 2 } \right)$ and $c \geq {2n \log 9}/{p}$. Moreover,
\begin{equation}\label{P2}
  {P}\left\{{1-t^2} \le  \lambda_{\min}(\mathbf{S}_{n}) \le  \lambda_{\max}(\mathbf{S}_{n}) \le {1+t^2}\right\}\geq 1 - 2 e^{-ct^2}.
\end{equation}
\end{proposition}
Proposition \ref{NA-BY} does not require ${p}/{n} \rightarrow y \in(0,1)$ as in %the asymptotical Bai-Yin's law in
Corollary \ref{Bai1993}. %The relation of $n$ and $p$ in Proposition \ref{NA-BY} can be arbitrary, even for finite $n$ or for $p \gg n$. The proof is based on 3 steps.
\begin{proof}
\textbf{Step1}. We introduce a \emph{counting measure} for measuring the complexity of a set in some space.
%It is similar to the bracketing number for function class in Section \ref{bracketing}.
The \emph{covering number} $\mathcal{N}(K,\varepsilon )$ is the smallest number of
  closed balls centered at $K$ with radii $\varepsilon$ whose union covers $K$. For some $\varepsilon \in [0,1)$, a subset $\mathcal{N}_{\varepsilon} \subset \mathbb{R}$ is an \emph{$\varepsilon$ -net} for $S^{n-1}$ if for all $\boldsymbol{x} \in S^{n-1},$ there is an $\boldsymbol{y} \in \mathcal{N}_{\varepsilon},$ such that $\|\boldsymbol{x}- \boldsymbol{y}\|<\varepsilon$.
%  The definition of $\mathcal{N}(K,\varepsilon )$ and $\varepsilon$-net are originating from functional analysis, and
  We use the following results in Lemma 5.2 and 5.4 of \cite{Vershynin2010}.

  %\footnote{HM: Two lemmas are fundamental for CIs of random vector and contains notations and def. spectral norm.}
\begin{lemma}[Covering numbers of the sphere]\label{net cardinality}
$\mathcal{N}(S^{n-1},\varepsilon) \le ( 1 + \frac{2}{\varepsilon})^n $ for every $\varepsilon >0$.

\end{lemma}
\begin{lemma}[Computing the spectral norm on a net] \label{norm on net general}
  Let $\mathbf{B}$ be an $p \times p$ matrix. Then
\begin{center}
  $
\big\|\mathbf{B}\big\|:= \max_{||\textbf{x}||_2 =1}\big\|\mathbf{B}\boldsymbol{x}\big\|_2 =\sup_{\boldsymbol{x} \in S^{p-1}} |\langle {\mathbf{B}\boldsymbol{x},\boldsymbol{x}} \rangle|
  \le (1 - 2\varepsilon)^{-1} \sup_{\boldsymbol{x} \in \mathcal{N}_{\varepsilon}} |\langle {\mathbf{B}\boldsymbol{x},\boldsymbol{x}} \rangle|.
  $
\end{center}
\end{lemma}
\noindent Lemma \ref{norm on net general} shows that
$
\big\| \frac{1}{n}\mathbf{X}^T\mathbf{X}-\mathbf{I}_p \big\|
\le 2 \max_{\boldsymbol{x} \in \mathcal{N}_{1/4}} \big| \frac{1}{n} \|\mathbf{X}\boldsymbol{x}\|_2^2 - 1 \big|.
$ Indeed, note that
$ \langle \frac{1}{n} \mathbf{X}^T\mathbf{X}\boldsymbol{x} - \boldsymbol{x}, \boldsymbol{x} \rangle = \langle \frac{1}{n} \mathbf{X}^T\mathbf{X}\boldsymbol{x}, \boldsymbol{x} \rangle - 1 = \frac{1}{n} \|\mathbf{X}\boldsymbol{x}\|_2^2 - 1. $ By setting $\varepsilon = 1/4$ in Lemma~\ref{norm on net general}, we get
\begin{equation}\label{eq:SW}
  \big\| {n}^{-1}\mathbf{X}^T\mathbf{X}-\mathbf{I}_p \big\| \leq ( 1 - 2 \varepsilon ) ^ { - 1 } \sup {}_ {\boldsymbol{x} \in \mathcal { N } _ { \varepsilon } } | \langle {n}^{-1} \mathbf{X}^T\mathbf{X}\boldsymbol{x} -\boldsymbol{x}, \boldsymbol{x} \rangle | = 2 \sup{}_ {  \boldsymbol{x}  \in \mathcal { N } _ { 1 / 4 } } | {n}^{-1} \|  \mathbf{X} \boldsymbol{x}  \| _ { 2 } ^ { 2 } - 1 |.
\end{equation}
By \eqref{eq:SW}, we have
\begin{align}\label{eq:NNN}
&{P}\{ \big\|  {n}^{-1}\mathbf{X}^T\mathbf{X}-\mathbf{I}_p \big\|\geq 2t \}\leq {P} \{ 2 \sup_ {  \bm{x}  \in \mathcal { N } _ { 1 / 4 } }| {n}^{-1} \|  \mathbf{X} \boldsymbol{x}  \| _ { 2 } ^ { 2 } - 1 | \geq 2t \}
\leq \sum_{  x  \in \mathcal { N } _ { 1 / 4 } } {P}\{ | {n}^{-1}\|  \mathbf{X} \boldsymbol{x}  \| _ { 2 } ^ { 2 } - 1 | \geq t\}  \nonumber\\
    & \leq  \mathcal { N } ( S ^ { n - 1 } , {1}/{4} ) {P} \{ | {n}^{-1} \|   \mathbf{X} \boldsymbol{x}  \| _ { 2 } ^ { 2 } - 1 | \geq t\}\leq   9^n  {P} \{ |{n}^{-1}\|   \mathbf{X} \boldsymbol{x}  \| _ { 2 } ^ { 2 } - 1 | \geq t\},~~\forall~{  x  \in \mathcal { N } _ { 1 / 4 } },
\end{align}
where the last inequality follows Lemma~\ref{net cardinality} with $\varepsilon = 1/4$.

\textbf{Step2}. It is sufficient to bound ${P} \{ | \frac { 1 } { n } \|   \mathbf{X} \boldsymbol{x}  \| _ { 2 } ^ { 2 } - 1 | \geq t\} $. Let $Z_i:=|\langle {\boldsymbol{X}_i,\boldsymbol{x}} \rangle|,~\forall~\boldsymbol{x} \in S^{n-1}$. Observe that
$\|\mathbf{X}\boldsymbol{x}\|_2^2 = \sum_{i=1}^n |\langle {\boldsymbol{X}_i,\boldsymbol{x}} \rangle|^2 =: \sum_{i=1}^n Z_i^2.$ Apply the sub-exponential concentration inequality in Corollary \ref{sub-exponentialConcentration},
$P({n}^{-1}| \|  \mathbf{X} \boldsymbol{x}  \| _ { 2 } ^ { 2 } - 1 || \ge t )=P( {n}^{-1}| {\sum_{i = 1}^n {(Z_i^2 - 1)} } | \ge t )\le 2e^{  { - \frac{n}{2}( {\frac{t^2}{\theta^2} \wedge \frac{t}{\theta}})} } .$
Specially, let $t = c\theta\max \left( \delta , \delta ^ { 2 } \right) =c\theta[  \delta \mathrm{I}_{\{\delta \leq 1\}} + \delta ^ { 2 }\mathrm{I}_{\{\delta > 1\}}]$ with $\delta :=2c( {{p}/{n}}  +{t}/{{\sqrt n }})$. From \eqref{eq:NNN},
\begin{align*}
 &~~~~ {P}\{ \big\| n^{-1}\mathbf{X}^T\mathbf{X}-\mathbf{I}_p \big\|\geq 2t \}\leq  9^n  \mathrm{P}\{|{n}^{-1}\|  \mathbf{X} \boldsymbol{x}  \| _ { 2 } ^ { 2 }  - 1| \geq c \theta\max \left( \delta , \delta ^ { 2 } \right) \}\\
   & \leq 2\cdot 9^ne^{[ - \frac{cn}{2} \min\{\delta^2 \mathrm{I}_{\{\delta \leq 1\}} + \delta ^ { 4 }\mathrm{I}_{\{\delta > 1\}}, \delta \mathrm{I}_{\{\delta \leq 1\}} + \delta ^ { 2 }\mathrm{I}_{\{\delta > 1\}}\}}  = 2 \cdot 9^ne^{ -\frac{cn}{2} \delta ^2} =e^{ -  \frac{c}{2}  (\sqrt{p} + t)^2} \leq 2 \cdot 9^ne^ {- c( p + t ^ { 2 } ) /2}
\end{align*}
where the last inequality is obtained by using the inequality $( a + b ) ^ { 2 } \geq a^2 + b^2 $ for $a, b \geq 0$. For $c \geq {n \log 9/p}$, $2 \cdot 9^n e^{-c( p+t^2)} \leq 2 e^{-ct^2}$, which proves \eqref{P1} .

\textbf{Step3}. To show \eqref{P2}, the
$\max_{||\boldsymbol{x}||_2 =1}|\|\frac{1}{\sqrt n}\mathbf{X}\boldsymbol{x}\|_2^2-1|=\max_{||\boldsymbol{x}||_2 =1}\big\|(\frac{1}{n}\mathbf{X}^T\mathbf{X}-\mathbf{I}_p)\boldsymbol{x}\big\|_2^2=\big\| \frac{1}{n}\mathbf{X}^T\mathbf{X}-\mathbf{I}_p \big\|^2\leq t^2$ implies that ${1-t^2} \le  \lambda_{\max}(\mathbf{S}_{n}) \le {1+t^2}.$ Similarly, for $\lambda_{\min}(\mathbf{S}_{n})$,
\begin{center}
$\min_{||\boldsymbol{x}||_2 =1}|\|\frac{1}{\sqrt n}\mathbf{X}\boldsymbol{x}\|_2^2-1|=\min_{||\boldsymbol{x}||_2 =1}\big\|(\frac{1}{n}\mathbf{X}^T\mathbf{X}-\mathbf{I}_p)\boldsymbol{x}\big\|_2^2 \le \max_{||\boldsymbol{x}||_2 =1}\big\|(\frac{1}{n}\mathbf{X}^T\mathbf{X}-\mathbf{I}_p)\boldsymbol{x}\big\|_2^2\leq t^2.$
\end{center}
So $\lambda_{\min}(\mathbf{S}_{n}) \in [{1-t^2},{1+t^2}]$ and
$\{\| \mathbf{X}^T\mathbf{X}-\mathbf{I}_p\|^2\leq t^2\}~\subset~\left\{{1-t^2} \le \lambda_{\min}(\mathbf{S}_{n}) \le  \lambda_{\max}(\mathbf{S}_{n}) \le {1+t^2}\right\}.$
Then ${P}\{{1-t^2} \le  \lambda_{\min}(\mathbf{S}_{n}) \le  \lambda_{\max}(\mathbf{S}_{n}) \le {1+t^2}\}\geq {P} \{\big\| \frac{1}{n}\mathbf{X}^T\mathbf{X}-\mathbf{I}_p \big\|^2\leq t^2 \} \geq 1 - 2 e^{-ct^2}.$
\end{proof}

\subsection{Oracle inequalities for penalized linear models}\label{linear}

This section introduces the proofs of the error bounds from the perspective of Lasso penalized linear models with the $\ell_2$-loss function. When $p>n$, the OLS estimator is no longer available as $\frac{1}{n} \sum_{i = 1}^n {{\textit{\textbf{X}}_i}{\textit{\textbf{X}}_i^T}}$ is of invertible. A common way for obtaining a plausible estimator for the true parameter $\bfbeta^*$ is by adding penalized function to the square loss function. For $0< q\leq\infty$, we write $\| \bfbeta\|_{q}:=(\sum_{i=1}^{p}|\beta_{i}|^{q})^{1/q}$ as the $\ell_{q}$-norm
for $\bfbeta \in \mathbb{R}^p$. If $q=\infty$, $\| \bfbeta\|_{\infty}:=\max_{i=1,...,p}|\beta_{i}|$; if $q=0$, $\| \bfbeta\|_{0}:=\sum_{i = 1}^p {{\rm{1(}}{\beta _i} \ne 0{\rm{)}}} $. There are two types \textit{statistical guarantees} of $\hat \bfbeta$ as mentioned in \cite{Bartlett12}.
\begin{enumerate}
\item \textbf{Persistence}: $\hat \bfbeta$ performs well on a new sample ${\bm X^*}\stackrel{d}{=}{\bm X}$ (equal in distribution), i.e.
${\rm{E}}\{[{\bm X^*}( \hat \bfbeta- {\bfbeta ^*} )]^2|{\bm X^*}\}\to 0$.

\item  \textbf{$\ell_q$-consistency} ($q\ge 1$): $\hat \beta$ approximates $\beta^{*}$, i.e. with high probability $\|\hat \bfbeta- {\bfbeta ^*}\|_q\to 0$.
\end{enumerate}
The persistence and $\ell_1$-consistency are respectively obtained by error bounds:
\begin{center}
$\|\hat \bfbeta  - {\bfbeta ^*}\|_1 \le {O_p}(s{\lambda _n}),~~~~{\rm{E}}\{[{\bm X^*}( \hat \bfbeta- {\bfbeta ^*} )]^2|{\bm X^*}\} \le {O_p}(s{\lambda _n^2})$,~\text{(says \textit{oracle inequalities})}
\end{center}
where ${\lambda _n}\to 0$ is a tuning parameter and $s:=\|\bfbeta^*\|_{0}$. In the following, we focus on the $\ell_1$ estimation and prediction consistencies for the penalized linear models. Let $\lambda  > 0$ be a tuning parameter, the \textit{Lasso estimator} \citep{Tibshirani1996} for Model \eqref{eq:LMs} is
\begin{align}\label{eq:lasso}
{\boldsymbol{\hat \beta}}_L={ \rm{argmin}}{}_{\boldsymbol{\beta}  \in {\mathbb{R}^p}} \{ {\| {\bm{Y} - \textbf{X}\boldsymbol{\beta} } \|_2^2}/n + \lambda {\left\| \boldsymbol{\beta} \right\|_1}\}.
\end{align}
By sub-derivative techniques in convex optimizations, the Karush-Kuhn-Tucker (KKT) condition of Lasso optimization function is
\begin{center}
$
\left\{
\begin{aligned}\label{eq:kktl}
2{[ {{{\bm X}^T}({\bm Y} - {\mathbf X}{\hat \bfbeta}_L)} ]_j}/n = - {\lambda } {\rm{sign}}(\hat\beta_{Lj})\, \text{  if } \hat\beta_j\neq 0,\\
2|{[ {{{\mathbf X}^T}({\bm Y} - {\mathbf X}{\hat \bfbeta} )}_L ]_j}|/n \le {\lambda } \qquad\text{  if } \hat\beta_{Lj}=0
\end{aligned}
\right.
$.
\end{center}
which implies
${\| {\frac{1}{n}{{\mathbf X}^T}({\bm Y} - {\mathbf X}{\hat \bfbeta})} \|_\infty } \le \frac{\lambda}{2} $. Another approach to get the Lasso-like sparse estimator is attained by Dantzig selector (DS)
%if we minimize $\ell_{1}$-norm among all $\beta$ satisfying the KKT condition with ${\hat \bfbeta}$ replacing by ${\bfbeta}$:
\begin{align}\label{eq:DS}
{{\hat \bfbeta }_{DS}} = {\arg \min }_{\boldsymbol{\beta}  \in {\mathbb{R}^p}} \{ {\left\| \bfbeta  \right\|_1}:{\| {{{\mathbf X}^T}({\bm Y} - {\mathbf X}\bfbeta )} \|_\infty/n } \le {\lambda}/{2}\}.
\end{align}
see \cite{Candes2007}. Lasso and DS are capable of producing sparse estimates with only a few (hence sparse) nonzero coefficients among the $p$ coefficients of the covariates.
%and this property is called sparse estimation.
The idea of Lasso and DS was presented in a geophysics literature \citep{Levy81}. By \eqref{eq:DS}, we get $\|\hat{\beta}_{DS}\|_{1} \le\|\hat{\beta}_{L}\|_{1}$, which signifies that the DS may be more sparse than the Lasso.

It is well-known that $\bm{\Sigma}:=\frac{1}{n} \sum_{i = 1}^n {{\textit{\textbf{X}}_i}{\textit{\textbf{X}}_i^T}}$
%(i.e. the correlation between the covariates)
is singular when $p>n$. To obtain oracle inequalities for the Lasso estimator with the minimax optimal rate \citep{Ye10}, the restricted eigenvalues proposed in \cite{Bickel09} is usually needed. Let $S(\boldsymbol{\beta}^{*}): = \{ j:{\beta_j^{*}} \ne 0,~\boldsymbol{\beta}^{*} = ({\beta_1^{*}}, \cdots,{\beta_p^{*}})^T\}$ and ${s} := \left| {S(\boldsymbol{\beta}^{*})} \right|$.
% be the number of non-sparse index with support: .
For any vector $\bm{b} \bm \in \mathbb{R}^p$ and any index set $H \subset \{1,2,\cdots,p\}$, define the sub-vector indexed by $H$ as $\bm{b}_H = (\cdots ,{\tilde b_j}, \cdots )^T \in {\mathbb{R}^p}$  with ${\tilde b_j}=  b_j$ if $j \in H$ and $\tilde  b_j = 0$ if $j \notin H$. Define the \textit{conic set} for a sparse $\boldsymbol{\beta}^{*}$ with support $S(\boldsymbol{\beta}^{*})$:
\begin{equation} \label{eq:compare3}
{\rm{C}}(\eta ,S(\boldsymbol{\beta}^{*}))=\{ \bm{b} \in {\mathbb{R}^p}:{\| {{\bm{b}_{{S(\boldsymbol{\beta}^{*})^c}}}} \|_1} \le \eta {\| {{\bm{b}_{S(\boldsymbol{\beta}^{*})}}} \|_1}\},~\eta>0.
\end{equation}
Denote the \emph{restricted eigenvalue condition} (RE) as
 $RE(\eta ,S(\boldsymbol{\beta}^{*}),{\bm \Sigma}) = \mathop {\inf }\limits_{0 \ne {\bm{b}} \in {\rm{C}}(\eta ,S(\boldsymbol{\beta}^{*}))} \frac{{{{({{\bm{b}}^T}{\bm \Sigma} {\bm{b}})}^{1/2}}}}{{{\left\| \bm{b} \right\|_2}}} > 0$
 for any $p\times p$ matrix ${\bm \Sigma}$. In the following, we present a modified version of Theorem 7.2 in \cite{Bickel09} from Lemma 2.5 of \cite{Li17} beyond Gaussian noise.

\begin{proposition}[The rate of convergence of the Lasso]\label{thm:lassoo}
Suppose that $\mathbf X$ is the fixed design matrix and the error sequence $\{ {\varepsilon _i}\} _{i = 1}^n  \stackrel{\rm{IID}}{\sim} N(0, \sigma^2) $ or $\{ {\varepsilon _i/\sigma}\} _{i = 1}^n  \stackrel{\rm{IID}} \sim 2-$strongly log-concave distribution satisfying Lemma \ref{3.16}. Let $\{\bm X_{ (j)}\}_{j=1}^p\in {\mathbb{R}^n}$ be column vectors of $\mathbf{X}$. We assume that $\frac{1}{n}\bm X_{(j)}^T{\bm X_{(j)}} = 1$. If $\lambda  = A\sigma \sqrt {{{\log p}}/{n}} $ satisfies the KKT condition for $\bm\beta^*$,
\begin{equation} \label{eqn:ineq1}
\|\mathbf X^{T}(\bm Y - \mathbf X \bm\beta^*)/n\|_\infty \leq \lambda/2.
\end{equation}
\begin{enumerate}[\rm{(}1\rm{)}]
    \item Then the estimated error $\bm u := \bm{\hat \beta}_L  - {\bm\beta ^*}$ satisfies $\|\bm u_{S(\boldsymbol{\beta}^{*})^c}\|_1 \leq 3\|\bm u_{S(\boldsymbol{\beta}^{*})}\|_1$, i.e.  $\bm u \in {\rm{C}}(3 ,S(\boldsymbol{\beta}^{*}))$.

\item Suppose that $\mathbf X$ satisfies the RE condition $\gamma:={\rm{RE}}(3, S(\boldsymbol{\beta}^{*}),\frac{1}{n} \sum_{i = 1}^n {{\textit{\textbf{X}}_i}{\textit{\textbf{X}}_i^T}})>0$. We have \emph{non-asymptotic oracle inequalities} with probability greater than $1-2p^{1-\frac{A^2}{8}}$:
\begin{align}
&(a).\|\bm{\hat \beta}_L - \bm\beta^*\|_1 \leq  \frac{12A\sigma}{\gamma}s\sqrt{\frac{\log p}{n}} ;~~(b).\|\bm{\hat \beta}_L- \bm\beta^*\|_2^2 \leq \frac{{9A{\sigma ^2}}}{{{\gamma ^2}}}\frac{{s\log p}}{n}; \label{thm:event1}\\
&(c).\frac{1}{n} \|\mathbf {X}(\bm{\hat \beta}_L- \bm\beta^*)\|_2^2 \leq \frac{{9A{\sigma}}}{{{\gamma}}}\frac{{s\log p}}{n},~A>2\sqrt{2}. \label{thm:event2}
\end{align}
\end{enumerate}
\end{proposition}
\begin{proof}
The proof consists of 3 steps. : \emph{1. Checking $\bm{\hat \beta}_L - {\bm\beta ^*}$ be in cone set by using definition of Lasso and KKT conditions; 2. Verifying the high probability of the KKT condition; 3. Deriving the oracle inequalities from restricted eigenvalue condition with some elementary inequalities.}\\
\textbf{Step1}: By the Lasso optimization \eqref{eq:lasso},
\begin{align}
(2n)^{-1}\|\bm Y-\mathbf X\hat{\bm\beta}_L\|_2^2+\lambda\|\hat{\bm\beta}_L\|_1 \leq (2n)^{-1}\|\bm Y-\mathbf X \bm\beta^{*}\|_2^2+\lambda\|\bm\beta^{*}\|_1.
\end{align}
From $\frac{\|\bm Y-\mathbf X\hat{\bm\beta}_L\|_2^2}{2n}=
\frac{1}{2n}\|\mathbf X\bm\beta^{*}+\bm\varepsilon-\mathbf X\hat{\bm\beta}_L\|_2^2=\frac{1}{2n}\|\mathbf X
\bm\beta^{*}-\mathbf X\hat{\bm\beta}_L\|_2^2+\frac{\|\bm\varepsilon\|_2^2}{2n}-\frac{1}{n}\bm\varepsilon^{T}
\mathbf X(\hat{\bm\beta}_L-\bm\beta^{*})$
and
$\frac{1}{2n}\|\bm Y-\mathbf X\beta^{*}\|_2^2=\frac{\|\bm\varepsilon\|_2^2}{2n},$ thus $\frac{1}{2n}\|\mathbf X\bm\beta^{*}-\mathbf X\hat{\bm\beta}_L\|_2^2+\frac{\|\bm\varepsilon\|_2^2}{2n}-\frac{1}{n}\bm\varepsilon^{T}\mathbf X(\hat{\bm\beta}-\bm\beta^{*})+\lambda\|\hat{\bm\beta}_L\|_1 \leq \frac{\|\bm\varepsilon\|_2^2}{2n}+\lambda\|\bm\beta^{*}\|_1$.
Then,
\begin{align}\label{eq:compare0}
(2n)^{-1}\|\mathbf X(\hat{\bm\beta}_L-\bm\beta^{*})\|_2^2+\lambda||\hat{\bm\beta}_L\|_1
\leq {n}^{-1}\bm\varepsilon^{T}\mathbf X(\hat{\bm\beta}_L-\bm\beta^{*})+\lambda\|\bm\beta^{*}\|_1.
\end{align}
%The \eqref{eq:compare0} is considerably simple derived by the minimization for Lasso optimization, but it is important.
The \eqref{eq:compare0} is usually called the \emph{basic inequality} in the proof of Lasso oracle inequalities. The first term in the left side of inequality \eqref{eq:compare0} is the empirical prediction error, while on the right side, $ \frac{1}{n}\bm\varepsilon^{T}\mathbf X(\hat{\bm\beta}-\bm\beta^{*}) $ is random and $\lambda||\bm\beta^{*}||_1$ is still fixed and unknown. For $\frac{1}{n}\bm\varepsilon^{T}\mathbf X(\hat{\bm\beta}-\bm\beta^{*})$, if we can get a sharper upper bound and it approaching 0 as $n \to \infty$, then we can achieve a sharper oracle inequality in below. By \eqref{eqn:ineq1},
\begin{align}\label{eq:compare}
\frac{\|\mathbf X(\hat{\bm\beta}-\bm\beta^{*})\|_2^2}{2n}+\lambda\|\hat{\bm\beta}\|_1\leq \|\frac{1}{n}\bm\varepsilon^{T}\mathbf X\|_{\infty}\|\hat{\bm\beta}-\bm\beta^{*}\|_1+\lambda\|\bm\beta^{*}\|_1\leq \frac{\lambda}{2}\|\hat{\bm\beta}-\bm\beta^{*}\|_1+\lambda\|\bm\beta^{*}\|_1.
\end{align}

\noindent Let $S:=S(\boldsymbol{\beta}^{*})$ and notes that
$||\hat{\bm\beta}_{S}||_1=||\bm\beta^{*}_{S}+(\hat{\bm\beta}_{S}-\bm\beta_{S}^{*})||_1\geq||\bm\beta^{*}_{S}||_1-||\hat{\bm\beta}_{S}-\bm\beta_{S}^{*}||_1$, then
\begin{align}\label{eq:triangle}
||\hat{\bm\beta}||_1=||\hat{\bm\beta}_{S^{c}}||_1+||\hat{\bm\beta}_{S}||_1\geq
||\bm\beta^{*}_{S}||_1-||\hat{\bm\beta}_{S}-\bm\beta_{S}^{*}||_1+||\hat{\bm\beta}_{S^{c}}||_1.
\end{align}
From \eqref{eq:compare}, we get $\|\bm u_{S^c}\|_1\leq 3\|\bm u_{S}\|_1$ be checking
\begin{align}\label{eq:triangle1}
&0\leq {(2n)}^{-1}||\mathbf X(\hat{\bm\beta}-\bm\beta^{*})\|_2^2\leq{\lambda}\|\hat{\bm\beta}-\bm\beta^{*}||_1/2+\lambda\|\bm\beta^{*}\|_1-\lambda\|\hat{\bm\beta}\|_1\nonumber\\
&\leq \frac{\lambda}{2}\{\|\hat{\bm\beta}_{S}-\bm\beta^{*}_{S}\|_1+\|\hat{\bm\beta}_{S^{c}}||_1\}+\lambda\|\bm\beta^{*}_{S}\|_1-\lambda\{
	\|\bm\beta^{*}_{S}\|_1-\|\hat{\bm\beta}_{S}-\bm\beta_{S}^{*}\|_1+\|\hat{\bm\beta}_{S^{c}}\|_1
			\}~~[\text{By}~\eqref{eq:triangle}]\nonumber\\
&=\frac{3\lambda}{2}\|\hat{\bm\beta}_{S}-\bm\beta^{*}_{S}\|_1-\frac{\lambda}{2}\|\hat{\bm\beta}_{S^{c}}\|_1=:\frac{3\lambda}{2}\|\bm u_{S}\|_1-\frac{\lambda}{2}\|\bm u_{S^{c}}\|_1.
\end{align}
%Here, we only show that the KKT condition has probability greater than $1-2p^{1-\frac{A^2}{8}}$, and the remainder proof can be found in Lemma 2.5 of \cite{Li17}.

\textbf{Step2}: The Gaussian error vector $\boldsymbol{\varepsilon}$ enables us to get the Gaussian concentration around its mean, we can shows that \eqref{eqn:ineq1} occurs with a high probability. So next we need to check the Lipschitz condition in Lemma~\ref{lem:caussiancon}.  Use Lemma~\ref{lem:caussiancon}, it implies that
\begin{equation}\label{eq:kktlip}
P({n}^{-1}| {\bm X_{(j)}^T(\bm Y - \mathbf X{\bm\beta ^*})| \ge t}) \le 2e^{  - {{n{t^2}}}/{{2{\sigma ^2}}}},~\forall~j.
\end{equation}
under the presupposition $\|{\bm X_{(j)}}\|_2^2=\bm X_{(j)}^T{\bm X_{(j)}} = n$. The Lipschitz condition depends on the design matrix $ \mathbf X $. The different types of CIs require different assumptions on the design matrix (the random design is allowed if we adopt empirical process theory). In Lemma~\ref{lem:caussiancon}, put
$f_j(\bm a): =\frac{1}{n}\bm X_{(j)}^T(\sigma \bm a  +\mathbf X{\bm\beta ^*})$. Then, Cauchy's inequality implies,
\begin{center}
$f_j(\bm a) - f_j(\bm b)
%&=\frac{1}{n}\bm X_{(j)}^T(\sigma \bm a  - \mathbf X{\bm\beta ^*})- \frac{1}{n}\bm X_{(j)}^T(\sigma \bm b  - \mathbf X{\bm\beta ^*})\\
\le\frac{\sigma }{n}|\bm X_{(j)}^T{{(\bm b - \bm a)}}|\le \frac{\sigma }{{n }}\|{\bm X_{(j)}}\|_2 \cdot \|\bm b -\bm a\|_2=\frac{\sigma }{{\sqrt n }}\|\bm b -\bm a\|_2~\forall~j.$
\end{center}
Hence, $f_j(\bm a)$ is ${\sigma }/{{\sqrt n }}$-Lipschitz for all $j$. Recall $\lambda  = A\sigma \sqrt {{{\log p}}/{n}} $. So \eqref{eq:kktlip} implies
\begin{center}
$P(\|\frac{1}{n}\mathbf X^{T}(\bm Y-\mathbf X \bm\beta^{*})\|_{\infty}\geq \frac{\lambda}{2})\le \sum_{j=1}^p P( { \frac{1}{n}| {\bm X_{(j)}^T(\bm Y - \mathbf X{\bm\beta ^*})}| \ge \frac{1}{2}A\sigma\sqrt {\frac{{{\log}p}}{n}}} )\leq 2 p^{1-\frac{A^2}{8}}.$
\end{center}
By Lemma \ref{3.16}, \eqref{eq:kktlip} is also held for $\{ {\varepsilon _i/\sigma}\} _{i = 1}^n\sim$ $2$-strongly log-concave distributions.
%So far, we have presented all the preparations required to prove of the oracle inequalities.

\textbf{Step3}: Next we can start on the proof based on cone set condition \eqref{eq:compare3}. Since the $\mathbf X$ satisfies RE condition $\gamma:={\rm{RE}}(3, S,{n}^{-1}\sum_{i = 1}^n {{\textit{\textbf{X}}_i}{\textit{\textbf{X}}_i^T}})>0$, by \eqref{eq:compare3} we have
\begin{center}
$\gamma\|\bm u\|_2^2\leq\frac{1}{n}\|\mathbf X\bm u\|_2^2\stackrel{\eqref{eq:triangle1}}{\le}\lambda(3\|{\bm u}_{S}\|_1-\|{\bm u}_{S^c}\|_1)\leq 3\lambda\|{\bm u}_S\|_1\leq 3\lambda\sqrt{s}\|{\bm u}_{S}\|_2\leq 3\lambda\sqrt{s}\|{\bm u}\|_2,$
\end{center}
where the second last inequality is by Cauchy's inequality. Therefore,
\begin{center}
$\|\hat{\bfbeta}_L-\bm\beta^{*}\|_2^2=:\|{\bm u}\|_2^2\leq\frac{9\lambda^2 s}{\gamma^2}=\frac{9A^2\sigma^2}{\gamma^2}\frac{s\log p}{n},~~\frac{\|\mathbf X(\hat{\bm\beta}_L-\bm\beta^{*})\|_2^2}{n}=:\frac{\|\mathbf X\bm u\|_2^2}{n}\leq\frac{9\lambda^2s}{\gamma}=\frac{9A^2\sigma^2}{\gamma}\frac{s \log p}{n}.$
\end{center}
Lastly,  by Cauchy's inequality and $\|\bm u_{S^c}\|_1\leq 3\|\bm u_{S}\|_1$,
\begin{center}
$\|\hat{\bm\beta}_{\lambda}-\bm\beta^{*}\|_1=:\|{\bm u}\|_1=\|\bm u_{S^c}\|_1+\|\bm u_{S}\|_1\leq 4\|\bm u_{S}\|_1\leq 4\sqrt{s}\|{\bm u}\|_2\leq \frac{12\lambda s}{\gamma}=\frac{12A\sigma}{\gamma}s\sqrt{\frac{\log p}{n}}$.
\end{center}
\end{proof}

According to \eqref{eq:ls1e}, the OLS with diverging number of covariates has the convergence rate $O(\sqrt {{p}/{n}} )$ under the minimal eigenvalue condition  ${\lambda _{\min }}\left( {{\mathbf{X}^{T}}\mathbf{X}} \right)=O(n)$. In contrast, due to the sparse restriction and the RE condition in Proposition \ref{thm:lassoo}, the factor $\sqrt {\log p} $ is much more small that the factor $\sqrt {p} $ in the convergence rate \eqref{eq:ls1e}. Under the RE condition, Proposition \ref{thm:lassoo} reveals that Lasso is $\ell_2$-consistent if $\frac{{s\log p}}{n}\to 0$, and $s\sqrt {\frac{{\log p}}{n}} \to 0$ guarantees $\ell_1$-consistency. Theorem 7.1 in \cite{Bickel09} also gives oracle inequalities \eqref{thm:event1} and \eqref{thm:event2} for the DS estimator \eqref{eq:DS}.

\subsection{High-dimensional Poisson regressions with random design}\label{Poisson}
%The responses in linear models are continuous variables in Section \ref{linear}. However, the category in classification or grouping may be infinite with index by the non-negative integers. The categorical variables is treated as countable number for distincting categories or groups. For example, random count predictors include the number of patients, the bacterium in the unit region, or stars in the sky and so on. In another point of view, any continuous variables can be approximated with a given accuracy level.

%Different from uncountable continuous variables, the discretized a variable is countable number of values which can be ordered.

%\begin{example} [Poisson regressions]
%For modeling count data $\{Y_i\}_{i=1}^n$ by the covariates $\{\bm X_i\}_{i=1}^n$,
The Poisson regression \citep{McCullagh83} is a model for nonnegative integers response variables, i.e. ${Y_i} \stackrel{\rm IID}{\sim} \mathrm{Poisson}(\lambda_i),$ where $\log (\lambda _i) = {\bm X_i^T}\bm\beta$ for $i=1,\cdots,n$. We presume that the $\{\bm X_i\}_{i=1}^n$ are IID r.vs on some space $\mathcal{X}$, and we observe $n$ copies of
$\{ ({Y_i},{\emph{\textbf{X}}_i})\}_{i = 1}^n \sim ({Y},{\emph{\textbf{X}}})\in \mathbb{R}\times \mathbb{R}^p$.
%In applications, the test data set is a new design $\textbf{X}^*$ which is an independent copy of $\textbf{X}$, thus it requires the randomness assumption of the design matrix.
%It worth noting that
%The Poisson response is a mean-variance dependent generalized linear models  $\mathrm{E}(Y_i|\bm X_i)=\mathrm{Var}(Y_i|\bm X_i)={e^{{\bm X_i^T}\bm\beta}}>0.$
The average negative log-likelihood of Poisson regressions is $\ell_n(\bfbeta ): =  - \frac{1}{n}\sum_{i = 1}^n {[{Y_i}{\bm X_i^T}\bm\beta  - {e^{{\bm X_i^T}\bm\beta}}]}$ and the Lasso penalized estimator is
\begin{equation}\label{eq:enp}
\boldsymbol{\hat \beta}:= \boldsymbol{\hat \beta} ({\lambda })=\mathop {\rm{argmin}}{}_{\boldsymbol{\beta}   \in {\mathbb{R}^p}} \{\ell_n(\boldsymbol{\beta}) + \lambda {{\left\|\boldsymbol{\beta}\right\|}_1}\}~\text{with  a turning parameter}~{\lambda } > 0.
\end{equation}

%\cite{Tutz11} gives a comprehensive introduction of basic and advanced concepts of categorical regression in machine learning view.
Lemma 4.2 in \cite{Buhlmann11} shows the first-order conditions for the optimization in \eqref{eq:enp}.
\begin{lemma}[Necessary and sufficient condition]\label{lem:iff}
 Let $j \in \{ 1,2, \cdots ,p\} $ and ${\lambda } > 0$. Then, a necessary and sufficient condition for the Lasso estimates \eqref{eq:enp}  is
\begin{eqnarray}\label{eq:kkt}
\left\{
\begin{aligned}
{n^{-1}\sum{}_{i = 1}^n {{{{X_{ij}}({Y_i} - {e^{{\textit{\textbf{X}}_i^T} \boldsymbol{\hat \beta} }})}}} }= - {\lambda } {\rm{sign}}(\hat\beta_j) \quad\, \text{  if } \hat\beta_j\neq 0,\\
|{n^{-1}\sum{}_{i = 1}^n {{{{X_{ij}}({Y_i} - {e^{{\textit{\textbf{X}}_i^T} \boldsymbol{\hat \beta} }})}}} }| \le {\lambda} \qquad\quad\text{  if } \hat\beta_j=0.
\end{aligned}
\right.
\end{eqnarray}
\end{lemma}
Let $ l(Y,{\emph{\textbf{X}}},{{\bfbeta }})=- Y{\bm{X}^T}\bfbeta+e^{{\textit{\textbf{X}}^T{\bfbeta }}}$ be the Poisson loss function. The true coefficient ${{{\bfbeta }} ^{*}}$ is the minimizer of the expected Poisson loss, i.e.
\begin{equation}\label{eq:oracle}
{\bfbeta } ^{*} = \argmin{}_{\boldsymbol{\beta}  \in {{\mathbb{R}}^{p}}} {
\mathrm{{E}}} l(Y,{\emph{\textbf{X}}},{{\bfbeta }}).
\end{equation}

The KKT condition of  the $\ell _{1}$-penalized likelihood is for the estimated parameter. But, here we use the true
parameter version of the KKT conditions: $|{\frac{1}{n}\sum_{i = 1}^n {{{{X_{ij}}({Y_i} - {\rm{E}}{Y_i})}}} }| \le {\lambda},~~j=1,...,p$ by replacing ${e^{{\textit{\textbf{X}}_i^T} \boldsymbol{\hat \beta} }}$ by ${\rm{E}}{Y_i}=e^{{\textit{\textbf{X}}_i^T} {\bfbeta } ^{*}}$ to approximate the estimated version \eqref{eq:kkt}.
%Lasso Poisson regression is similar to Lasso linear models in Section \ref{linear} by showing some high-probability events like the KKT condition.
To motivate the next two propositions concerning high-probability events, let us consider the following notations and the decomposition of empirical process.

The Poisson loss $l(\bfbeta ,\bm{X},Y) = {l_1}(\bfbeta ,\bm{X},Y) + {l_2}(\bfbeta ,\bm{X})$ is decomposed into two parts
%(a linear function of responses and a non-linear function of coveriates),
where ${l_1}(\bfbeta ): = {l_1}(\bfbeta ,\bm{X},Y) := - Y{\bm{X}^T}\bfbeta $ and ${l_2}(\bfbeta): = {l_2}(\bfbeta,\textit{\textbf{X}}) :=e^{{\textit{\textbf{X}}^T}\bfbeta } $ is free of response.  Let $\mathbb{P}l({{\bfbeta }}) := {\rm E}l(\bfbeta,\bm{X},Y)$ be the expected {loss}.
%, where the expectation is under the randomness of $(\textit{\textbf{X}} ,Y)$
 We are {interested in} the centralized empirical loss $\left( \mathbb{P}_{n}-\mathbb{P}\right)  l(\bfbeta)$ representing  fluctuations between the expected and empirical  losses.
%It will lead to the convergence rate of $\left( \mathbb{P}_{n}-\mathbb{P}\right)  l(\bfbeta)$ based on CIs.
Note that
% To analysis this convergence rate, we break down the empirical process into two parts:
\begin{equation}\label{eq:EPP}
\left( \mathbb{P}_{n}-\mathbb{P}\right)  l(\bfbeta)=\left( \mathbb{P}_{n}-\mathbb{P}\right) l_{1}(\bfbeta)+\left( \mathbb{P}_{n}-\mathbb{P}\right)  l_{2}(\bfbeta),
\end{equation}
which is crucial in attaining the convergence rate of $\|\bm{\hat \beta} - \bm\beta^*\|_1$. Motivated from rate of convergence theorem [Theorem 3.2.5 of \cite{van96}] for M-estimation with functional parameter in some metric space, we study the upper bounds (or the rate) for the first and second part of the difference of the centralized empirical process between $\bfbeta^{*}$ and $\hat{\bfbeta}$: $(\mathbb{P}_{n}-\mathbb{P})( l_{m}(\bfbeta^{*})-l_{m}(\hat{\bfbeta}))$, for $m=1,2$.
%It will be shown that  $(\mathbb{P}_{n}-\mathbb{P})( l_{m}(\bfbeta^{*})-l_{m}(\hat{\bfbeta}))$ has the \emph{stochastic Lipschitz properties} w.r.t. ${\|{{\hat \bfbeta }} - \bfbeta ^*\|_1}$ with high %probability.

\begin{proposition}[Convergence rate of $(\mathbb{P}_{n}-\mathbb{P})( l_{1}(\bfbeta^{*})-l_{1}(\hat{\bfbeta}))$]\label{prop:upbound1}
Suppose that
\begin{equation}\label{eq:H1}
\mathop {\sup }{}_{1 \le i \le \infty} {\| \textit{\textbf{X}}_i \|_\infty } \le L <\infty~\text{\rm a.s.}~\text{and}~\|\bfbeta^* \|_1 \le B.
\end{equation}
In the event of
${\cal A} := \bigcap_{j = 1}^p {\{ {| {\frac{1}{n}\sum_{i = 1}^n {{{{X_{ij}}({Y_i} - {\rm{E}}{Y_i})}}} } |\le \frac{{{\lambda }}}{4}} \}}, $ we have
\begin{equation}\label{eq:L1}
(\mathbb{P}_{n}-\mathbb{P})( l_{1}(\bfbeta^{*})-l_{1}(\hat{\bfbeta}))\le \frac{{{\lambda }}}{4} {\|{{\hat \bfbeta }} - \bfbeta ^*\|_1}.
\end{equation}
If $\lambda  \ge \max \{ \frac{{16{A^2}L\log (2p)}}{{3n}},8AL{e^{LB/2}}\sqrt {\frac{{\log (2p)}}{n}} \} $ with $A>1$,  we have $P(\mathcal{A}) \ge 1- {(2p)^{1 - {A^2}}}$.
\end{proposition}
\begin{proof}
Note that, on the event ${\cal A}$
\begin{align*}
(\mathbb{P}_{n}-\mathbb{P}) ( l_{1}(\bfbeta^{*})-l_{1}(\hat{\bfbeta})) &=\frac{{ - 1}}{n}\sum\limits_{i = 1}^n {({Y_i}}  - {\rm{E}}{Y_i})\textit{\textbf{X}}_i^T({\bfbeta ^*} - \hat \bfbeta ) = \sum\limits_{j = 1}^p {({{\hat \beta }_j} - \beta _j^*)}{\frac{1}{n}\sum\limits_{i = 1}^n {{{{X_{ij}}({Y_i} - {\rm{E}}{Y_i})}}} } \\
& \le \sum\limits_{j = 1}^p {|{{\hat \beta }_j} - \beta _j^*|}\cdot |{\frac{1}{n}\sum\limits_{i = 1}^n {{{{X_{ij}}({Y_i} - {\rm{E}}{Y_i})}}} }| \stackrel{\mathcal{A}}{\le}  \frac{{{\lambda }}}{4} {\|{{\hat \bfbeta }} - \bfbeta ^*\|_1}.
\end{align*}

Next, we show that $\mathcal{A}$ is a high probability event if the ${\lambda}$ is well chosen. For $j=1,...,p$ and $i=1,...,n$, $P({{\mathcal{A}}^c})  \le \sum\limits_{j = 1}^p P \{ {|{\frac{1}{n}\sum\limits_{i = 1}^n {{{{X_{ij}}({Y_i} - {\rm{E}}{Y_i})}}} } | > \frac{{{\lambda }}}{4}}\}.$
Given $\textbf{X}$, $\{S_{nj}(Y,X):={\frac{1}{n}{{{{X_{ij}}({Y_i} - {\rm{E}}{Y_i})}}} }\}_{i = 1}^n$ are conditional independent for each $j=1,...,p$. Thus Corollary~\ref{col:Poisson} with $w_i={X_{ij}}/n$ gives
 \begin{equation}\label{eq:majo-proba-A-bis}
\begin{aligned}
{P}(|S_{nj}(Y,X)|\geq t|\textbf{X})&\leq 2 \exp\{-{\frac {nt^{2}/2}{{\frac{1}{n}\sum\limits_{i = 1}^n e^{{\textit{\textbf{X}}_i^T} {\bfbeta } ^{*}}{\mathop {\max }\limits_{1 \le i \le n} X_{ij}^2}}+ {\mathop {\max }\limits_{1 \le i \le n} \frac{|X_{ij}|t}{3}}}}\} \le 2({e^{\frac{{ - n{t^2}}}{{4{L^2}{e^{LB}}}}}} \vee {e^{\frac{{ - 3nt}}{{4L}}}})
\end{aligned}
\end{equation}
where the last inequality is from ${{\mathop{\rm e}\nolimits} ^{ - \frac{a}{{b + c}}}} \le  {{\mathop{\rm e}\nolimits} ^{\frac{{ - a}}{{2b}}}}\vee{{\mathop{\rm e}\nolimits} ^{\frac{{ - a}}{{2c}}}}$ for any positive numbers $a, b$ and $c$.

Let $t=\frac{\lambda}{4}$. Assumptions \eqref{eq:H1} and \eqref{eq:majo-proba-A-bis} give for $j=1,...,p$
\begin{center}
${P}(|{\frac{1}{n}\sum\limits_{i = 1}^n {{{{X_{ij}}({Y_i} - {\rm{E}}{Y_i})}}} } |\geq \frac{\lambda}{4} )=\mathrm{E}{P}(|{\frac{1}{n}\sum\limits_{i = 1}^n {{{{X_{ij}}({Y_i} - {\rm{E}}{Y_i})}}} } |\geq \frac{\lambda}{4}|\textbf{X})\le 2\max \{ {e ^{\frac{{ - n\lambda^2}}{64L^2e^{LB}}}},{e ^{\frac{{ - 3n\lambda}}{16L}}}\},$
\end{center}
which implies that
$P({{\mathcal{A}}^c})  \le 2p\max \{ {e ^{\frac{{ - n\lambda^2}}{64L^2e^{LB}}}},{e ^{\frac{{ - 3n\lambda}}{16L}}}\} .$

Finally, if $\lambda  \ge \max \{ \frac{{16{A^2}L\log (2p)}}{{3n}},8AL{e^{LB/2}}\sqrt {\frac{{\log (2p)}}{n}} \}~(A>1)$, so
$P({{\mathcal{A}}^c})  \le {(2p)^{1 - {A^2}}}.$
\end{proof}
Next, we provide a crucial lemma to bound ${{({\mathbb{P}_n} - \mathbb{P})\left( {{l_2}({\boldsymbol{\beta} ^*}) - {l_2}(\boldsymbol{\beta} )} \right)}}$. Let ${\nu _n}(\boldsymbol{\beta} ,{\boldsymbol{\beta} ^*}): = \frac{{({\mathbb{P}_n} - \mathbb{P})\left( {{l_2}({\boldsymbol{\beta} ^*}) - {l_2}(\boldsymbol{\beta} )} \right)}}{{ {\|{\boldsymbol{\beta}} - \boldsymbol{\beta} ^*\|_1} }}$ the normalized empirical process indexed by $\boldsymbol{\beta}$. Denote the $\ell_1$-ball by ${{\cal S}_{M}}(\bfbeta ^*):= \left\{ {\bfbeta  \in {\mathbb{R}^p}:{\|{\bfbeta } - \bfbeta ^*\|_1}  \le {M}<\infty} \right\}$,  we define the \emph{local stochastic Lipschitz constant}:
\begin{center}
${Z_M}(\bfbeta^*):={{\rm{sup}}}_{\boldsymbol{\beta}\in {{\cal S}_M}(\boldsymbol{\beta}^*)} |{\nu _n}(\bfbeta ,{\bfbeta ^*})|~\text{and a random event}~{\cal B} := \{ {Z_M}(\bfbeta^*)\le {{{{\lambda _1}}}}/{4}\}$.
\end{center}
%which bounds ${Z_M}(\bfbeta^*)$ by the rescaled tuning parameter $\frac{{{{\lambda _1}}}}{4}$. \fn{no ending for prop 8.3 proof. HM: The ending of proving prop 8 is in page 48.}
It is easy to see
$| {{\nu _n}( \hat  \bfbeta ,{\bfbeta ^*})} | \le \mathop {\sup }_{{{\cal S}_M}(\boldsymbol{\beta}^*)} | {{\nu _n}( \hat \bfbeta ,{\bfbeta ^*})} | \le \frac{{{\lambda_1}}}{4}$, which gives $| {({\mathbb{P}_n} -\mathbb{P})({l_2}({\hat \bfbeta }) - {l_2}(\bfbeta^* ))} | \le \frac{{{\lambda_1}}}{4}{\|{ \hat \bfbeta } -\bfbeta ^*\|_1}$, provided that $\hat\bfbeta \in {{\cal S}_{M}}(\bfbeta ^*)$. Then we have follow result.

\begin{proposition}[Convergence rate of $(\mathbb{P}_{n}-\mathbb{P})( l_{2}(\bfbeta^{*})-l_{2}(\hat{\bfbeta}))$]\label{lem:upbound2}
Assume that there exists a large constant ${M}$ such that $\hat\bfbeta$ is in the $\ell_1$-ball ${{\cal S}_{M}}(\bfbeta ^*)$. Under assumption \eqref{eq:H1}, we have
\begin{equation}
P({Z_M}(\bfbeta^*) \ge {5AL e^{LB} }\sqrt {\frac{{\log 2p}}{n}} ) \le {(2p)^{ - {A^2}}}.
\end{equation}
If $\lambda  \ge 20AL{e^{LB}}\sqrt{\frac{{2\log 2p}}{n}}$, we get
$P\{| {({\mathbb{P}_n} -\mathbb{P})({l_2}({\hat \bfbeta }) - {l_2}(\bfbeta^* ))} | \le \frac{{{\lambda}}}{4}{\rm{(}} {\|{ \hat \bfbeta } -\bfbeta ^*\|_1}  )\} \ge 1-{(2p)^{ - {A^2}}}.$
\end{proposition}
\begin{proof}
In the first step, we apply following McDiarmid's inequality to ${Z_M}(\bfbeta^*)$ by showing that ${Z_M}(\bfbeta^*)$ is fluctuated of no more than $\frac{{2e^{LB}}}{{n}}$. Let us check it. Put
$\mathbb{P}_{n}:=\frac{1}{n}\sum_{j=1}^{n}1_{\textit{\textbf{X}}_{j},Y_{j}}$ and $\mathbb{P}_{n}^{'}:=\frac{1}{n}\sum_{j=1, j\neq i}^{n}1_{\textit{\textbf{X}}_{j},Y_{j}}+1_{\textit{\textbf{X}}_{i}^{'},Y_{i}^{'}},$
where $({\textit{\textbf{X}}_{i}^{'},Y_{i}^{'}})$ is the independent copy of $({\textit{\textbf{X}}_{j},Y_{j}})$.

Let $\textit{\textbf{X}}_{i}^{T}\tilde{\bfbeta}_{i}$ ($\textit{\textbf{X}}{'}_{i}^{T}\tilde{\bfbeta}_{i}$) be an intermediate point between $\textit{\textbf{X}}_{i}^{T}\bfbeta$ ($\textit{\textbf{X}}{'}_{i}^{T}\bfbeta$) and $\textit{\textbf{X}}_{i}^{T}\bfbeta^*$ ($\textit{\textbf{X}}{'}_{i}^{T}\bfbeta^*$) from the Taylor's expansion of function $F(x):=e^{x}$. It deduces
\begin{align*}
& ~~~~\mathop {\sup }\limits_{\boldsymbol{ \beta} \in {S_M}}\frac{{|({\mathbb{P}_n} - \mathbb{P})({l_2}({\bfbeta ^*})- {l_2}(\bfbeta ))|  }}{{\|{\bfbeta ^*} - \bfbeta \|_1 }}-\mathop {\sup }\limits_{\boldsymbol{ \beta} \in {S_M}}\frac{| ({\mathbb{P}{'}_n} - \mathbb{P})({l_2}({\bfbeta ^*}) - {l_2}(\hat \bfbeta ))| }{{\|{\bfbeta ^*} - \bfbeta \|_1 }} \\
& \le \mathop {\sup }\limits_{\boldsymbol{ \beta} \in {S_M}}\frac{|{{l_2}({\bfbeta ^*},{\textit{\textbf{X}}_i}) - {l_2}(\bfbeta ,{\textit{\textbf{X}}_i}) - {l_2}({\bfbeta ^*},{\textit{\textbf{X}}{'}_i}) + {l_2}(\bfbeta ,{\textit{\textbf{X}}{'}_i})}|}{{n\|{\bfbeta ^*} - \bfbeta \|_1}}\\
& \le \mathop {\sup }\limits_{\boldsymbol{ \beta} \in {S_M}}\frac{1}{n}{e^{\textit{\textbf{X}}_i^T\tilde \bfbeta }} \cdot \frac{{|\textit{\textbf{X}}_i^T{\bfbeta ^*} - \textit{\textbf{X}}_i^T \bfbeta| }}{{\|{\bfbeta ^*} - \bfbeta \|_1 }} +\mathop {\sup }\limits_{\boldsymbol{ \beta} \in {S_M}}\frac{1}{n}{e^{\textit{\textbf{X}}_i^T\tilde \bfbeta }} \cdot \frac{{|\textit{\textbf{X}}{'}_i^T{\bfbeta ^*} - \textit{\textbf{X}}{'}_i^T \bfbeta |}}{{\|{\bfbeta ^*} - \bfbeta \|_1 }}\le \mathop {\sup }\limits_{\boldsymbol{ \beta} \in {S_M}}\frac{{2Le^{LB}}}{n}\frac{{\|{\bfbeta ^*} - \bfbeta \|_1}}{{\|{\bfbeta ^*} - \bfbeta \|_1 }} = \frac{{2Le^{LB}}}{{n}}.
\end{align*}
where the first inequality stems from $ \left| {f(x)} \right| - \mathop {\sup }\limits_x \left| {g(x)} \right| \le \left| {f(x) - g(x)} \right|$ (and take suprema over $x$ again).

Apply McDiarmid's inequality to ${Z_M}(\bfbeta^*)$, we have
$P({Z_M}(\bfbeta^*) - {\rm{E}}{Z_M}(\bfbeta^*) \ge {\lambda}) \le e^{  - \frac{{{{n}}{{\lambda}^2}}}{2L^2e^{2LB}}} $.
Let ${(2p)^{ - {A^2}}}=\exp \{  - \frac{{{{n}}{{\lambda}^2}}}{2L^2e^{2LB}}\}$, we get $
{\lambda} \ge ALe^{LB}\sqrt {\frac{{2\log (2p)}}{n}}
$ for $A>0$, therefore
\begin{equation}
P({Z_M}(\bfbeta^*) - {\rm{E}}{Z_M}(\bfbeta^*) \ge {\lambda }) \le {(2p)^{ - {A^2}}}.
\label{proba1-gpl}
\end{equation}

The next step is to estimate the sharper upper bounds of ${\rm{E}}{Z_M}(\bfbeta^*)$ by Lemma \ref{tm:Symmetrization} with $\Phi(t)=|t|$ and Lemma \ref{lm:Contraction}. Note that ${({\mathbb{P}_n} - \mathbb{P})\left\{{l_2}({\bfbeta ^*}) - {l_2}( \bfbeta )\right\}}={{\mathbb{P}_n}\left\{{l_2}({\bfbeta ^*}) - {l_2}(\bfbeta )\right\}}-{\rm{E}}{\left\{{l_2}({\bfbeta^*}) - {l_2}( \bfbeta )\right\}}$, by symmetrization theorem, the expected terms is canceled. To see contraction theorem, for
${Z_M}(\bfbeta^*) = \mathop {\sup }\limits_{\boldsymbol{ \beta} \in {S_M}} \left\{ \frac{1}{{n\|{\boldsymbol{ \beta} ^*} - \boldsymbol{\beta} \|_1 }}{|\sum\limits_{i = 1}^n (e^{{\textit{\textbf{X}}_i^T} {\bfbeta } ^{*}}-e^{{\textit{\textbf{X}}_i^T} {\bfbeta } })-n{\rm{E}}{[{l_2}({\bfbeta ^*}) - {l_2}( \bfbeta )]}|}\right\}$, it is required to check the Lipschitz property of $g_i$ in Lemma~\ref{lm:Contraction} with ${\mathcal F}=\mathbb{R}^p$. Let $f({x_i})= {{x_i^T} \boldsymbol{\beta}}/{\|{\boldsymbol{ \beta} ^*} - \boldsymbol{\beta} \|_1 },~h({x_i})= {{x_i^T} \boldsymbol{\beta}}^*/{\|{\boldsymbol{ \beta} ^*} - \boldsymbol{\beta} \|_1 }$ and ${g_i}(t) =\frac{{e^{{t{\|{\boldsymbol{ \beta} ^*} - \boldsymbol{\beta} \|_1 }  }}}}{n\|{\boldsymbol{ \beta} ^*} - \boldsymbol{\beta} \|_1 }~(|t|\le LB/{\|{\boldsymbol{ \beta} ^*} - \boldsymbol{\beta} \|_1 })$. Then the function ${g_i}(t)$ here is $\frac{e^{LB}}{n }$-Lipschitz. In fact
\begin{center}
$\left| {{g_i}(s) - {g_i}(t)} \right| =\frac{e^{\tilde t}}{n }  \cdot |s-t|\le \frac{e^{LB}}{n}|s-t|,~t,s\in [-LB/{\|{\boldsymbol{ \beta} ^*} - \boldsymbol{\beta} \|_1 },LB/{\|{\boldsymbol{ \beta} ^*} - \boldsymbol{\beta} \|_1 }]$
\end{center}
where $\tilde t \in [-LB/{\|{\boldsymbol{ \beta} ^*} - \boldsymbol{\beta} \|_1 },LB/{\|{\boldsymbol{ \beta} ^*} - \boldsymbol{\beta} \|_1 }]$ is an intermediate point between $t$ and $s$ given by applying Lagrange mean value theorem.

The symmetrization theorem and the contraction theorem imply
\begin{align*}
{\rm{E}}{Z_{M}(\bfbeta^*)} & \le \frac{4e^{LB} }{n} {\rm{E}} (\underset{\beta \in \mathcal{S}_{M}}{\sup}\left\lvert\sum_{i=1}^{n}  \frac{\epsilon_{i}\textit{\textbf{X}}_{i}^{T}({\bfbeta^{*}}-\bfbeta)}{{\|\bfbeta  - {\bfbeta ^*}\|_1 }}  \right\rvert) \le \frac{4 e^{LB} }{n} {\rm{E}}(\underset{\beta \in \mathcal{S}_{M}}{\sup}\mathop {\max }\limits_{1 \le j \le p}|\sum_{i=1}^{n}  {\epsilon_{i}\textit{X}_{ij}}|\cdot \frac{\|\bfbeta  - {\bfbeta ^*}\|_1 }{{\|\bfbeta  - {\bfbeta ^*}\|_1 }}) \\
& \le \frac{4 e^{LB} }{n}{\rm{E}}( \mathop {\max }\limits_{1 \le j \le p}|\sum_{i=1}^{n}  {\epsilon_{i}\textit{X}_{ij}} |)= \frac{4 e^{LB} }{n}{\rm{E}}({\rm{E}}[ \mathop {\max }\limits_{1 \le j \le p}|\sum_{i=1}^{n}  {\epsilon_{i}\textit{X}_{ij}} ||\textbf{X}]).
\end{align*}
From Corollary~\ref{pp-Maximalbd}, with ${\rm{E}}_{\epsilon}[{\epsilon_{i}\textit{X}_{ij}}|\textbf{X}]=0$ we get
$\frac{4 e^{LB} }{n}{\rm{E}}({\rm{E}}[ \mathop {\max }\limits_{1 \le j \le p}|\sum_{i=1}^{n}  {\epsilon_{i}\textit{X}_{ij}} ||\textbf{X}])\le \frac{4 e^{LB} }{n}\sqrt {2\log 2p}  \cdot \sqrt {n{L^2}}  ={4 e^{LB} }L\sqrt {\frac{{2\log 2p}}{n}} .$
Thus, for $A\ge 1$,
\begin{equation}\label{proba2-gpl}
{\rm{E}}{Z_M}(\bfbeta^*) \le{4 e^{LB} L}\sqrt {\frac{{2\log 2p}}{n}} \le {4AL e^{LB} }\sqrt {\frac{{2\log 2p}}{n}}.
\end{equation}
With ${\lambda} \ge ALe^{LB}\sqrt {\frac{{2\log (2p)}}{n}}
$ and (\ref{proba2-gpl}), we conclude from (\ref{proba1-gpl})  that
$P({Z_M}(\bfbeta^*) \ge {5AL e^{LB} }\sqrt {\frac{{\log 2p}}{n}} ) \le P({Z_M}(\bfbeta^*) \ge {\lambda} + {\rm{E}}{Z_M}(\bfbeta^*)) \le {(2p)^{ - {A^2}}}.$ Finally, we complete the proof of Proposition \ref{lem:upbound2} by letting $\frac{{{\lambda}}}{4} \ge {5AL e^{LB} }\sqrt {\frac{{2\log 2p}}{n}}$ and setting $\bfbeta = \hat \bfbeta \in {Z_M}(\bfbeta^*)$.
\end{proof}

Let $S:={S(\boldsymbol{\beta}^{*})}$ for $\boldsymbol{\beta}^{*}$ defined in \eqref{eq:oracle} and $s:=\vert S\vert$. To obtain sharp oracle inequalities for Lasso penalized Poisson regression, we consider the following regularity conditions:
 \begin{itemize}
\item [\textbullet] (H.1): The covariate $\bm X$ is almost surely bounded $\parallel \bm X\parallel_{\infty} \le L$ a.s. for $L>0$;
\item [\textbullet]  (H.2): There exists a constant $B>0$ such that $\Vert \bfbeta^*\Vert_{1}\le B$;

\item [\textbullet] (H.3): (Stabil Condition) For $\Sigma:=\mathrm{E}(\bm X \bm X^{T})$, there exsist a $k \in (0,1)$ such that
\begin{center}
$\delta ^{T}\Sigma \delta\geqslant k\sum_{j\in S} \delta_{j}^{2}$
for any $\delta \in C(c_{0},S):=\lbrace \delta \in \mathbb{R}^{p}: \sum_{j \in {S}^{c} }\vert \delta_{j}\vert\le c_{0}\sum_{j \in S }\vert \delta_{j}\vert \rbrace $.
\end{center}
\end{itemize}

The \emph{Stabil Condition} (H.3) is denoted as $S(c_{0},S,k, \Sigma)$ which is a similar version of the RE condition in the Lasso linear models proposed in
\cite{Bunea08}. Due to the random variance, Poisson regression is more complex than the linear model with the constant variance assumption. {Thus,}  (H.1) and (H.2) are stronger {than those assumed }for the  linear models. Based on the high-probability event $\mathcal{A}$ and $\mathcal{B}$, we have the oracle inequalities for estimation and prediction for Lasso estimator $\boldsymbol{\hat \beta}$ in \eqref{eq:enp} for the Poisson regressions.

\begin{theorem}\label{le-gp-glm}
Assume conditions $(H.1)-(H.3)$ hold. Let $\lambda$ be chosen such that
\begin{equation}\label{lambda}
\lambda  \ge \max \{ \frac{{16{A^2}L\log (2p)}}{{3n}},8AL{e^{LB/2}}\sqrt {\frac{{\log (2p)}}{n}}, 20AL{e^{LB}}\sqrt {\frac{{2\log 2p}}{n}}\}~\text{for}~A>\sqrt{2}.
\end{equation}
Suppose that we have a new covariate vector $\bm  X^*$ (as the test data) which is an independent copy of $ X$ (as the training data), and ${\rm{E^*}}$ represents the expectation w.r.t. $\bm  X^*$ only, then
\begin{center}
$P({\rm{E^*}}{[{\bm X^{*}}( \hat \bfbeta- {\bfbeta ^*} )]^2}\le\dfrac{{12e^{10LB}}}{k} s\lambda^2)~\text{and}~P(\Vert \hat{\beta}-\beta^{*}\Vert_{1}\le \dfrac{4{e^{5LB}}}{k} s\lambda) \ge 1 - {(2p)^{1 - {A^2}}}-{(2p)^{ - {A^2}/2}}.$
\end{center}
\end{theorem}
The Theorem \ref{le-gp-glm} leads to the persistence and $\ell_1$-consistency if $\max\{s\lambda,s\lambda^2\}\to 0$.
\begin{proof}
The proof consists of three steps. The techniques are adapted from \cite{Zhang17}, \cite{Huang2020} and references therein.

{\bf{Step1: Check $\hat{\bfbeta}-\bfbeta^{*} \in C(3,S)$}}. From the definition of the Lasso estimates $\hat{\bfbeta}$ (see \eqref{eq:enp}),
\begin{equation}\label{eq:def}
{\mathbb{P}_n}l(\hat\bfbeta) + {\lambda }||\hat\bfbeta ||_1 \le {\mathbb{P}_n}l(\bfbeta^{*}) + \lambda ||{\bfbeta^*}||_1.
\end{equation}
By adding $\mathbb{P}(l(\hat\bfbeta)- l(\bfbeta^{*}))+\frac{\lambda}{2} ||\hat\bfbeta - {\bfbeta^*}||_1$ to both sides of (\ref{eq:def}), we have
\begin{equation*}\label{eq:def-putting}
\mathbb{P}(l(\hat\bfbeta )- l(\bfbeta^{*}))+ \dfrac{\lambda}{2}\|\hat\bfbeta  - {\bfbeta^*}\|_1  \le ({\mathbb{P}_n} - \mathbb{P})(l({\bfbeta^*})-l(\hat\bfbeta)) +\dfrac{\lambda}{2}\|\hat\bfbeta  - {\bfbeta^*}\|_1 + \lambda (||{\bfbeta^*} \|_1- ||\hat\bfbeta\|_1),
\end{equation*}
which leads %turns to \eqref{eq:lambda1} on the event $\mathcal{A}\bigcap\mathcal{B}$
\begin{align}\label{eq:lambda1}
\mathbb{P}(l(\hat\bfbeta)- l(\bfbeta^{*}))+ \frac{\lambda}{2}  ||\hat \bfbeta  - {\bfbeta ^*}||_1 & \le ({\mathbb{P}_n} - \mathbb{P})(l({\bfbeta ^*})-l(\hat\bfbeta )) +\frac{\lambda}{2}  ||\hat\bfbeta  - \bfbeta^*||_1+ \lambda (\|{\bfbeta ^*} \|_1 - \|\hat\bfbeta\|_1) \nonumber\\
& \le  {\lambda}  \|\hat\bfbeta  - {\bfbeta^*}\|_1+ \lambda (\|{\bfbeta ^*} \|_1 - \|\hat\bfbeta\|_1).
\end{align}
By the definition of $\beta^{*}$, $\mathbb{P}(l(\hat \bfbeta)- l(\bfbeta^{*}))\ge 0$. The above inequality and the fact: $|{{\hat \beta }_j} - \beta _j^*|+|\beta _j^*| - |{{\hat \beta }_j}|=0$ for $j\notin S$ and $|{{\hat \beta }_j}| - |\beta _j^*| \le |{{\hat \beta }_j} - \beta _j^*|$ for $j\in S$ lead to
\begin{align}\label{eq-WC1}
{\lambda }\|\hat\bfbeta - {\bfbeta ^*}\|_1/2 & \le  \lambda \|\hat\bfbeta  - {\bfbeta ^*}\|_1+ \lambda (\|\bfbeta ^* \|_1 -\|\hat\bfbeta\|_1) \le  2 \lambda \|(\hat\bfbeta  - {\bfbeta ^*})_{S}\|_1.
\end{align}
Thus,  $
\frac{\lambda }{2} ||(\hat \bfbeta  - {\bfbeta ^*})_{S^c}||_1 \le 1.5 \lambda ||(\hat \bfbeta- {\bfbeta^*})_{S}||_1 $ and then $\hat{\bfbeta}-\bfbeta^{*} \in C(3,S)$.

{\bf{Step2: Choosing ${\lambda }$}}. Since $\mathbb{P}(l(\hat\bfbeta)- l(\bfbeta^{*}))\ge 0$, \eqref{eq:lambda1} implies
\begin{equation}\label{eq:lambda2}
\begin{aligned}
{\lambda } \|\hat\bfbeta - {\bfbeta^*}\|_1/2 & \le \lambda\|\hat\bfbeta  - {\bfbeta^*}\|_1 +  \lambda (\|\bfbeta^* \|_1 - \|\hat\bfbeta\|_1)\\
& \le \lambda \|\hat\bfbeta \|_1 + \lambda \|{\bfbeta ^*}\|_1 + \lambda (\|{\bfbeta ^*} \|_1 - \|\hat\bfbeta\|_1) ={\rm{2}}\lambda\|{\bfbeta^*}\|_1 .
\end{aligned}
\end{equation}
Thus (H.2) implies $||\hat \bfbeta - {\bfbeta^*}|{|_1} \le {{{\rm{4}}B}} .$  After having shown Propositions \ref{prop:upbound1} and \ref{lem:upbound2}, we need the result on the high probability of the event $\mathcal{A}\bigcap\mathcal{B}$, %as given in the next proposition
whose proof is skipped.
\begin{proposition}\label{prop:upbound3}
Under the event $\mathcal{A}\bigcap\mathcal{B}$ with (H.1)-(H.3), we have $\hat\bfbeta \in {{\cal S}_{4B}}(\bfbeta ^*)$. And if $\lambda$ are chosen as \eqref{lambda}, then
$P({\cal A} \cap {\cal B}) \ge 1 - {(2p)^{1 - {A^2}}}-{(2p)^{ - {A^2}/2}}.$
\end{proposition}

{\bf{Step3: Error bounds from Stabil Condition}}. As ${\bm X^*}$ is an independent copy of $\bm{X}$,
\begin{align*}
&~~~~\mathbb{P}\{l(\hat \bfbeta ) - l({\bfbeta ^*})\}= {{\rm{E}}^{\rm{*}}}[{\rm{E}}\{l(\boldsymbol{\hat\beta}) - l({\bfbeta ^*})|{\bm X^{*}}\}]:={{\rm{E}}^{\rm{*}}}\{{\rm{E}} { {[ - Y{\textit{\textbf{X}}^{*T}}(\bfbeta  - {\bfbeta ^*}) +e^{{\textit{\textbf{X}}^{*T}} {\boldsymbol{\beta} }}-e^{{\textit{\textbf{X}}^T} \boldsymbol{\beta}^* }]|{\bm X^{*}}} }\}|_{\boldsymbol{\beta}=\boldsymbol{\hat\beta}}\\
&= { {{{\rm{E}}^{\rm{*}}}\{ {{\rm{E}}[ - Y|{\bm X^{*}}]{\bm X^{*T}}(\bfbeta  - {\bfbeta ^*}) +(e^{{\textit{\textbf{X}}^{*T}} {\boldsymbol{\hat\beta} }}-e^{{\textit{\textbf{X}}^T} \boldsymbol{\beta}^* })]|{\bm X^{*}} }\}}|_{\boldsymbol{\beta}=\boldsymbol{\hat\beta}}},~~({{\rm{E}}^{\rm{*}}}[Y|{\bm X^{*}}] ={{{e^{{\bm X^{*T}}{\boldsymbol{\beta}^*}}}}})\\
  &= {{\rm{E}}^{\rm{*}}}{{\{ {-{{{e^{{\bm X^{*T}}{\boldsymbol{\beta}^*}}}}} + {{{e^{{\bm X^{*T}}{\boldsymbol{\beta}^*}}}}} + {2}^{-1}{{{e^{{\bm X^{*T}} \boldsymbol{\tilde\beta} }}{{[{\bm X^{*T}}(\bfbeta  - {\bfbeta ^*})]}^2}}}} \}} |_{\boldsymbol{\beta}=\boldsymbol{\hat\beta}}}={2}^{-1}{{\rm{E}}^{\rm{*}}}{{\{ {{e^{{\bm X^{*T}} \boldsymbol{\tilde\beta} }}{{[{\bm X^{*T}}(\bfbeta  - {\bfbeta ^*})]}^2}} \}} |_{\boldsymbol{\beta}=\boldsymbol{\hat\beta}}},
\end{align*}
where ${{\textit{\textbf{X}}^{*T}} {\tilde\bfbeta } }=(1-t){\textit{\textbf{X}}^{*T}}{\bfbeta} ^*+t{\textit{\textbf{X}}^{*T}}{\hat\beta}$ is an intermediate point of ${\textit{\textbf{X}}^{*T}}$ and ${\textit{\textbf{X}}^{*T}}{\hat\beta}$ with $t\in [0,1]$.

Note that $\Vert \bfbeta^*\Vert_{1}\le B$ by (H.1) and $||\hat \bfbeta - {\bfbeta^*}||_1\le {{{\rm{4}}B}}$, (H.2) yields
\begin{equation*}\label{eq:bd}
|{{\textit{\textbf{X}}^{*T}} {\tilde\bfbeta } }| \le t|{\textit{\textbf{X}}^{*T}}{\hat\beta}-{\textit{\textbf{X}}^{*T}}{\bfbeta} ^*| + |{\textit{\textbf{X}}^{*T}}{\bfbeta} ^*| \le ||{\textit{\textbf{X}}^{*}}||_{\infty}\cdot ||\hat \bfbeta - {\bfbeta^*}||_1+ |{\textit{\textbf{X}}^{*T}}{\bfbeta} ^*|\le 4LB+LB=5LB,
\end{equation*}
which implies for $c :={e^{-5LB}}/2$
\begin{align}\label{equation-Steinwart}
\mathbb{P}\{l(\hat \beta ) - l({\beta ^*})\}\ge {\inf }_{\left| t \right| \le 5LB} {2}^{-1}{{\rm{E}}^{\rm{*}}}{{\{ {{e^{{\bm X^{*T}} \boldsymbol{\tilde\beta} }}{{[{\bm X^{*T}}(\bfbeta  - {\bfbeta ^*})]}^2}} \}} |_{\boldsymbol{\beta}=\boldsymbol{\hat\beta}}}=: c{\rm{E^*}}[{\bm X^{*T}}(\hat\bfbeta  - {\bfbeta ^*})]^2.
\end{align}
As ${\rm{E^*}}({\bm X^{*}}{\bm X^{*T}})=\boldsymbol{\Sigma}$, %the expected prediction error is
${\rm{E^*}}{[{\bm X^{*}}( \hat \bfbeta- {\bfbeta ^*} )]^2} = {(\hat \bfbeta  - {\bfbeta ^*})}\boldsymbol{\Sigma} (\hat \bfbeta  - {\bfbeta ^*}).$

Having checked the cone condition $C(3,S)$, we apply the Stabil Condition
\begin{equation}\label{eq-Stabil}
c{(\hat \bfbeta  - {\bfbeta ^*})}\boldsymbol{\Sigma} (\hat \bfbeta  - {\bfbeta ^*}) \ge ck||(\hat\bfbeta  - {\bfbeta ^*})_{S}||_2^2 .
\end{equation}
From (\ref{eq:lambda1}), (\ref{eq-WC1}) and (\ref{equation-Steinwart}), we get
\begin{equation}\label{eq-WC2}
c{\rm{E^*}}{[{\bm X^{*}}( \hat \bfbeta- {\bfbeta ^*} )]^2} + \frac{\lambda }{2} ||\hat \bfbeta  - {\bfbeta ^*}||_1 \le \mathbb{P}(l(\hat \bfbeta )- l(\bfbeta^{*}))+ \frac{\lambda }{2} ||\hat \bfbeta  - {\bfbeta ^*}||_1 \le 2 \lambda ||(\hat \bfbeta  - {\bfbeta ^*})_{S}||_1,
\end{equation}
which gives $ck||(\hat\bfbeta  - {\bfbeta ^*})_{S}||_2^2 +\frac{\lambda }{2} ||\hat \bfbeta  - {\bfbeta ^*}||_1 \le  2 \lambda ||(\hat \bfbeta  - {\bfbeta ^*})_{S}||_1$ by plugging \eqref{eq-Stabil} into \eqref{eq-WC2}. Then, employing Cauchy's inequality, we have
\begin{align}\label{equation-CS}
2ck||(\hat\bfbeta  - {\bfbeta ^*})_{S}||_2^2 +\lambda ||\hat \bfbeta  - {\bfbeta ^*}||_1 \le   4\lambda (s \cdot \|{{(\hat\bfbeta  - {\bfbeta ^*})}_S}\|_2^2)^{1/2}\le  4t{\lambda ^2}s + {\textstyle{1 \over t}}||(\hat\bfbeta  - {\bfbeta ^*})_{S}||_2^2,
\end{align}
where the last inequality is from the elementary inequality $2xy \le tx^{2}+y^{2}/t$  for all $t>0$. Let us set $t = {(2ck)^{ - 1}}$ in (\ref{equation-CS}), thus
$\|\hat \bfbeta  - {\bfbeta ^*}\|_1 \le 4t\lambda s  = \frac{{2\lambda s}}{{ck}}=\frac{{4e^{5LB}}}{k} s\lambda .$

To derive the oracle inequality of prediction error, from \eqref{eq-WC2}, we obtain
\begin{center}
$c{\rm{E^*}}{[{\bm X^{*}}( \hat \bfbeta- {\bfbeta ^*} )]^2}  \le 1.5 \lambda ||(\hat \bfbeta  - {\bfbeta ^*})_{S}||_1 \le 1.5 \lambda ||(\hat \bfbeta  - {\bfbeta ^*})||_1$
\end{center}
which implies ${\rm{E^*}}{[{\bm X^{*}}( \hat \bfbeta- {\bfbeta ^*} )]^2}\le  1.5 \lambda ||(\hat \bfbeta  - {\bfbeta ^*})||_1/c \le \frac{{3s\lambda^2}}{{c^2k}}=\frac{{12e^{10LB}}}{k} s\lambda^2$,
where the last inequality is from $\|\hat \bfbeta  - {\bfbeta ^*}\|_1 \le \frac{{4e^{5LB}}}{k} s\lambda.$
\end{proof}

For general losses beyond linear models, the crucial techniques in the non-asymptotical analysis of increasing-dimensional and high-dimensional regressions, which are \textit{Bahadur representation's for the M-estimator} \citep{Kuchibhotla18D,Pan2020} and \textit{concentration for Lipschitz loss functions} \citep{Buhlmann11,Zhang17}, respectively. In large-dimensional regressions with $p/n \to c$, the \textit{theory of random matrix} \citep{Yao15}, \textit{leave-one-out analysis} \citep{Leil18,Karoui2013} and \textit{approximate message passing} \citep{Karoui2013,Donoho2016,Karoui2018} play important roles for obtaining asymptotical results.
\section{Extensions}

%The exponential inequalities are sharp tail probability bounds that characterize how the sum of r.vs and the extreme r.vs fluctuate according to its expectations.

%Most of concentration results for mean in this paper are proved by the well-known Cramer-Chernoff method: Let $Z_{1}, \ldots, Z_{n}$ be $n$ independent centralized r.vs, and suppose there
%exists a convex function $g(t)$ and a domain $D_0$ containing $\{0\}$ such that
%$
%\mathrm{E}e^{t \sum_{i=1}^{n} Z_{i}} \leq e^{n g(t)},~\forall t \in D_0 \subset \mathbb{R}.
%$
%Denote $g^{*}(s)=\sup _{t \in D_0}\{t s-g(t)\}$ as the \textit{convex conjugate function} of $g,$ therefore the Chernoff's inequality implies
%$
%P(|\frac{1}{n} \sum_{i=1}^{n} Z_{i}|>s) \leq 2e^{-n g^{*}(s)},~\forall s>0
%$
%which has rich applications in HD statistics and random matrix theory.

The review has been focused on the sum of independent r.vs in the Euclidean space. However,  independence structure may not be suitable for some applications, for instance, econometrics, survival analysis, and graphical models. At the same time, the Euclidean valued r.vs may not be appropriate for functional data and image data.
{In the following we point out results in settings not covered to broaden this review. }

%These non-parametric models can also be treated as regression models with divergence (infinite) number of covariates in Section \ref{se;lm}.

%To study oracle inequalities for HD, non-parametric regressions, and functional regressions, there are considerable developments of sharp CIs, not limited to the independence structure or Euclidian space.

By CIs for the martingales, oracle inequalities have been proposed for Lasso penalized Cox models, see \cite{Huang13}. Some statistical models, such as the Ising model involving Markov's chains. \cite{Miasojedow2018} applied Hoeffding's inequality for Markov's chains to deal with this difficulty, see \cite{Fanj18} for a review. In time series analysis, \cite{Xie2018} studies the square-root Lasso method for HD linear models with $\alpha$, $\rho$, $\phi$-mixing or $m$-dependent errors. The Hoeffding's and Bernstein's CIs for weakly dependent summations can be found in \cite{Bosq1998}. Via sub-Weibull concentrations under $\beta$-mixing, non-asymptotic inequalities for estimation errors, and the prediction errors are obtained by \cite{Wong17} for the Lasso-regularized sparse VAR model with sub-Weibull innovations. U-Statistic is another dependent sum, and Example \ref{eg:u} provides a concentration result by McDiarmid's inequality. \cite{Borovskikh1996} introduces the concentration for the Banach-valued U-statistics.

In non-parametric regressions, the corresponding score functions may be r.vs in Banach (or Hilbert) space; see the monographs \cite{Ledoux91}, \cite{Yurinsky95} for introductions. Exponential tail bounds for Banach- or Hilbert-valued r.vs are indispensable for deriving sharp oracle inequalities of the error bounds, see \cite{ZhangT05}, \cite{Lei20}. Recently, Banach-valued CIs are applied to conceive non-asymptotic hypothesis testing for non-parametric regressions, see \cite{Yang20}. To extend the empirical covariance matrices from finite to infinite dimension, the sample covariance operator is treated as a random element in Banach spaces. The concentrations of empirical covariance operator also have been raised attention in kernel principal components analysis, and functional data analysis, see  \cite{Rosasco10}, \cite{Bunea15}.

Testing hypotheses on the regression coefficients are a necessity in measuring the
effects of covariates on the certain response variables. Scientists are interested in testing the significance of a large number of covariates simultaneously. From this backgrounds, \cite{Zhong11} proposed simultaneous tests for coefficients in HD linear models under the ``large $p$, small $n$'' situations by U-statistics motivated by \cite{Chen2010}. However, their HD tests are asymptotical without a non-asymptotic guarantee. Motivated by \cite{Arlot10}, \cite{ZhuBardic18} invents a new methodology for testing the linearity hypothesis in HD linear models, and the test they proposed does not impose any restriction of model sparsity. Based on the concentration of Lipschitz functions of Gaussian distributions or strongly log-concave distribution, \cite{Zhu18} developed a new concentration-based test in HD regressions. Recently, \cite{Wang20} studied non-asymptotical two-sample testing using \emph{Projected Wasserstein Distance}, via McDiarmid's inequality.

{\color{black}{In future, it would be essential and practical to study the estimator for the sub-exponential, sub-Gaussian, sub-Weibull and GBO norms as the unknown parameters when constructing non-asymptotical and data-driven confidence intervals; see \cite{Zhang23,Wang20E,Zhou25}.}}\\

\noindent{\bf Acknowledgments}\\
~~~\\
The authors thank Haoyu Wei, Xiang Li, Xiaoyu Lei, Qiuping Wang, Yanpeng Li, Chang Cui, Shengming Zhong, and Han Run for comments on the early versions of this paper. {\color{black}{The authors also thank two referees for helpful suggestions.}}
This research is funded by National Natural Science Foundation of China Grants 92046021, 12071013, 12026607 and 71973005, and LMEQF at Peking University.

%%%% Bibliography  %%%%%%%%%%

\end{document}